\documentclass[11pt, reqno]{amsart}

\usepackage{comment}
\usepackage[rgb,dvipsnames]{xcolor}
\usepackage{stmaryrd}
\usepackage{amsfonts}
\usepackage{amssymb}
\usepackage{amsmath,amsthm}
\usepackage{latexsym}
\usepackage{amscd}
\usepackage{esint}
\usepackage{mathrsfs}
\usepackage{dsfont}
\usepackage{setspace}
\usepackage{enumitem}
\usepackage[margin=1.25in]{geometry}
\usepackage{bm}

\usepackage{hyperref}
\usepackage{comment}
\usepackage{enumitem}

\newtheorem{prop}{Proposition}[section]
\newtheorem{thm}[prop]{Theorem}
\newtheorem{lemm}[prop]{Lemma}
\newtheorem{coro}[prop]{Corollary}

\newtheorem{defiprop}[prop]{Proposition-Definition}

\newtheorem*{claim*}{Claim}

\theoremstyle{definition}
\newtheorem{defi}[prop]{Definition}
\newtheorem{rmk}[prop]{Remark}

\newcommand{\CC}{\mathbb{C}}

\newcommand{\NN}{\mathbb{N}}

\newcommand{\RR}{\mathbb{R}}

\newcommand{\ZZ}{\mathbb{Z}}

\newcommand{\cA}{\mathcal A}
\newcommand{\cB}{\mathcal B}

\newcommand{\cE}{\mathcal E}

\newcommand{\cK}{\mathcal K}

\newcommand{\cN}{\mathcal N}

\newcommand{\cP}{\mathcal P}

\newcommand{\cS}{\mathcal S}
\newcommand{\cT}{\mathcal T}

\newcommand{\cV}{\mathcal V}

\DeclareMathOperator{\Span}{span}

\DeclareMathOperator{\supp}{supp}

\DeclareMathOperator{\Div}{div}
\DeclareMathOperator{\loc}{loc}

\DeclareMathOperator{\re}{Re}

\DeclareMathOperator{\dist}{dist}

\DeclareMathOperator{\Ric}{Ric}

\DeclareMathOperator{\Ind}{Ind}

\DeclareMathOperator{\Ker}{Ker}
\DeclareMathOperator{\Ran}{Ran}

\newcommand{\ep}{\varepsilon}

\newcommand{\Vol}{\text{Vol}}

\newcommand{\pa}[2]{\frac{\partial #1}{\partial #2}}

\newcommand{\paop}[1]{\pa{}{#1}}

\setlist[enumerate]{leftmargin = 2em}
\allowdisplaybreaks

\numberwithin{equation}{section}

\title[Existence of CMC spheres]{Existence of constant mean curvature 2-spheres in Riemannian 3-spheres}
\author{Da Rong Cheng}
\address{Department of Pure Mathematics, University of Waterloo, Waterloo, ON, N2L 3G1 Canada}
\email{drcheng@uwaterloo.ca}
\author{Xin Zhou}
\address{Department of Mathematics, Cornell University, Ithaca, NY 14853}
\email{xinzhou@cornell.edu}

\begin{document}

\dedicatory{Dedicated to Rick Schoen on the occasion of his 70th birthday}

\begin{abstract} 
We prove the existence of branched immersed constant mean curvature 2-spheres in an arbitrary Riemannian 3-sphere for almost every prescribed mean curvature, and moreover for all prescribed mean curvatures when the 3-sphere is positively curved. To achieve this, we develop a min-max scheme for a weighted Dirichlet energy functional.  There are three main ingredients in our approach: a bi-harmonic approximation procedure to obtain compactness of the new functional, a derivative estimate of the min-max values to gain energy upper bounds for min-max sequences for almost every choice of mean curvature, and a Morse index estimate to obtain another uniform energy bound required to reach the remaining constant mean curvatures in the presence of positive curvature. 
\end{abstract}

\maketitle

\section{Introduction}

In this paper, we address the question of constructing surfaces in a given three manifold with prescribed constant mean curvature and controlled topology, focusing on the fundamental case when the ambient space is a Riemannian 3-sphere. Constant mean curvature (CMC) surfaces constitute an important and widely-studied topic in differential geometry, and appear as models in many disciplines, including soap bubbles, gas-liquid interface and event horizons in general relativity. The search for CMC surfaces has produced an extensive literature (see for example \cite{Heinz54, Hildebrandt70, Brezis-Coron84, Struwe85, Struwe86, Struwe88, Kapouleas90, Ye91, Duzaar-Steffen96, Mahmoudi-Mazzeo-Pacard06, Rosenberg-Smith20, Breiner-Kapouleas17}), but no general existence theory was available for closed hypersurfaces with prescribed constant mean curvature until the very recent joint work of Zhu and the second author, \cite{Zhou-Zhu19, Zhou-Zhu20}. However, the existence results in \cite{Zhou-Zhu19} left open the question of whether the CMC surfaces have controlled topology in 3-dimensional ambient spaces, particularly in Riemannian 3-spheres. In the case of homogeneous 3-spheres, the work of Meeks-Mira-P\'erez-Ros \cite{MMPR13} gives the existence and uniqueness of immersed CMC 2-spheres with any prescribed mean curvature; see also \cite{Abresch-Rosenberg05, DanielMira13, Meeks13, MMPR20}, and the references therein. By contrast, the existence of CMC $2$-spheres in arbitrary Riemannian 3-spheres is only known for $H=0$ \cite{Sacks-Uhlenbeck81, Smith82}, and for very large $H$ \cite{Ye91, Pacard-Xu09}, but remains open for other $H>0$. Our first result fills in this gap in the branched immersed case for almost every prescribed mean curvature. 

\begin{thm}\label{thm:main1}
Given a Riemannian manifold $(S^3, g)$ diffeomorphic to the standard 3-sphere, for almost every constant $H>0$, there exists a nontrivial branched immersed $2$-sphere with constant mean curvature $H$ and Morse index at most 1.
\end{thm}

Before stating our second result, we would like to mention a conjecture made by Rosenberg-Smith \cite[page 3]{Rosenberg-Smith20} in their treatise on the degree theory of immersed, prescribed curvature hypersurfaces. The conjecture states that ``{\em for any $H\geq 0$ and any metric $g$ on $S^3$ of positive sectional curvature, there exists an embedding of $S^2$ to $S^3$ of constant mean curvature $H$}". Our second result confirms their conjecture for branched immersed CMC 2-spheres. 
\begin{thm}\label{thm:main2}
If the Riemannian 3-sphere $(S^3, g)$ has positive Ricci curvature, then for every constant $H>0$, there exists a nontrivial branched immersed $2$-sphere with constant mean curvature $H$ and with Morse index 1.
\end{thm}
\begin{rmk}\label{rmk:main2}
We actually prove a stronger statement than Theorem~\ref{thm:main2}, namely that the existence holds whenever $g$ and $H$ satisfy $Ric_{g}>- \frac{H^2}{2} g$.
\end{rmk}

\begin{rmk} We make two comments about our main theorems.
\vskip 1mm
\begin{enumerate}
\item[(1)] Concerning our choice of ambient space, we note that Theorem~\ref{thm:main1} and Theorem~\ref{thm:main2} are easily seen to hold with $S^3$ replaced by any spherical 3-manifold, which are of course the only closed, connected 3-manifolds admitting metrics with positive Ricci curvature. More generally, with only minor modifications to their proofs, Theorem~\ref{thm:main1} and Remark~\ref{rmk:main2} hold with $S^3$ replaced by any closed 3-manifold with non-trivial third homotopy group. Note that this includes all prime 3-manifolds which are non-aspherical. 
\vskip 1mm
\item[(2)] Our results generalize to the CMC setting in Riemannian $3$-spheres the celebrated existence theory of branched immersed minimal $2$-spheres by Sacks-Uhlenbeck \cite{Sacks-Uhlenbeck81}; see also \cite{Sacks-Uhlenbeck82, Lamm, Colding-Minicozzi08b}. Also closely related is the result of Struwe \cite{Struwe88}, who, using the heat flow of harmonic maps, constructed CMC disks with free boundary in a 3-dimensional Euclidean domain for almost every $H$ within an upper bound depending only on the radius of the domain. Thus, our results can also be viewed as generalizations of Struwe's result to Riemannian $3$-spheres. 
\end{enumerate}
\end{rmk}

\vspace{0.5em}
Below we recall the PDE of interest. Assuming that $(S^3, g)$ is isometrically embedded in $\RR^N$ for some large $N\in \NN$, the CMC surface we construct is parametrized by a smooth map $u: S^2\to (S^3, g)\subset \RR^N$ satisfying:
\begin{align}
& \Delta u - A(u)(\nabla u, \nabla u) = H *(u^* Q), \label{equ:CMC equation1}\\
& |u_x|^2 - |u_y|^2 = u_x\cdot u_y =0. \label{equ:CMC equation2}
\end{align}
Here $A$ and $Q$ are, respectively, the second fundamental form of the embedding $S^3\subset \RR^N$ and the cross product on $TS^3$ induced by the metric $g$ and the volume form $\Vol_g$. The $\ast$ following $H$ is the Hodge star operator on $S^2$, so that both sides of~\eqref{equ:CMC equation1} are functions $S^2 \to \RR^N$; (see Section \ref{S:perturbed functional} for more details). In~\eqref{equ:CMC equation2}, $(x, y)$ is a choice of isothermal coordinates on $S^2$, while the norm and inner product are of course the standard ones on $\RR^N$. Note also that since the domain is $S^2$, the weak conformality condition~\eqref{equ:CMC equation2} follows from~\eqref{equ:CMC equation1} by a classical argument using the Hopf differential. By the work of Gulliver~\cite{Gul73}, non-constant smooth solutions to~\eqref{equ:CMC equation1} and~\eqref{equ:CMC equation2} are branched immersions and parametrize surfaces with constant mean curvature $H$. Below we refer to~\eqref{equ:CMC equation1} as the CMC equation. 
 
\subsection*{Related backgrounds on CMC and minimal surfaces}
The local existence theory for CMC surfaces in $\RR^3$ with Plateau boundary conditions was initiated by Heinz \cite{Heinz54} and Hildebrandt \cite{Hildebrandt70}. The Rellich conjecture, which asserts the existence of at least two solutions to the CMC Plateau problem, was solved later by Brezis-Coron \cite{Brezis-Coron84} and Struwe \cite{Struwe85, Struwe86}. For the existence of closed CMC hypersurfaces, the boundary of isoperimetric regions are shown to be smoothly embedded CMC hypersurfaces (up to a singular set of codimension 7); see \cite{Almgren76, Morgan03}. However, this approach does not provide control on the value of the mean curvature, nor control on the topology in 3-manifolds. By perturbation arguments, one can generate foliations by closed CMC hypersurfaces from a given non-degenerate closed minimal hypersurface. Moreover, Ye \cite{Ye91}, Mahmoudi-Mazzeo-Pacard \cite{Mahmoudi-Mazzeo-Pacard06}, and others have constructed foliations by closed CMC hypersurfaces near minimal submanifolds of strictly lower dimensions; (see the survey article \cite{Pacard05}). Nevertheless, the CMC hypersurfaces constructed by these methods have either very small or very large mean curvatures. We also mention the gluing constructions by Kapouleas \cite{Kapouleas90} and Breiner-Kapouleas \cite{Breiner-Kapouleas17}, which produced many important examples of CMC surfaces in Euclidean spaces.  Using the min-max method, Zhu and the second author \cite{Zhou-Zhu19} established the full existence theory of closed CMC hypersurfaces with any given prescribed mean curvature in any closed Riemannian manifold of dimension between 3 and 7. This was later generalized to higher dimensions by Dey allowing a singular set of codimension 7 \cite{Dey19}. In addition, there are several exciting recent works on the existence of multiple CMC hypersurfaces based on the min-max theory in \cite{Zhou-Zhu19}: Dey calculated the asymptotics of the number of distinct closed CMC hypersurfaces of prescribed small mean curvatures in \cite{Dey19}, and Mazurowski constructed unstable closed CMC hypersurfaces of index 1 or 3 near a non-degereate closed minimal hypersurface \cite{Mazurowski20}, generalizing a gluing result by Pacard-Sun \cite{Pacard-Sunpre}.

By comparison, we have seen in recent years tremendous advancements in the existence theory of minimal hypersurfaces, which are CMC hypersurfaces with $H = 0$.  The Almgren-Pitts min-max theory \cite{Almgren62, Almgren65, Pitts81, Schoen-Simon81}, which provided the first general existence result for closed minimal hypersurfaces, was recently greatly improved and refined starting from the resolution of the Willmore conjecture by Marques-Neves \cite{Marques-Neves14}.  Yau's conjecture on the existence of infinitely many closed minimal surfaces was solved by Marques-Neves \cite{Marques-Neves17} assuming positive Ricci curvature, and then by Song in the general case \cite{Song18}.  For generic metrics, Irie-Marques-Neves \cite{Irie-Marques-Neves18}, Marques-Neves-Song \cite{Marques-Neves-Song19} respectively proved density and equidistribution results for closed minimal hypersurfaces, both using the Weyl Law for the area functional by Liokumovich-Marques-Neves \cite{Liokumovich-Marques-Neves18}. Using a cylindrical type Weyl Law introduced in \cite{Song18}, Song and the second author obtained scarring results for closed minimal hypsurfaces surrounding any closed stable hypersurface \cite{Song-Zhou20}. Around the same time, a Morse theory for the area functional was established: the second author \cite{Zhou19} proved the Multiplicity One Conjecture raised by Marques-Neves \cite{Marques-Neves16, Marques-Neves18}; (see also Chodosh-Mantoulidis \cite{Chodosh-Mantoulidis18}). When combined with \cite{Marques-Neves18}, this implies that, for bumpy metrics, there exists a closed minimal hypersurface of Morse index $p$ for each $p\in \NN$. Recently, the Morse inequalities for the area functional were proved for bumpy metrics by Marques-Montezuma-Neves \cite{Marques-Montezuma-Neves20}. 

Parallel to this line of works, Simon-Smith proved the existence of an embedded minimal 2-sphere in any Riemannian 3-sphere by adapting of the Almgren-Pitts theory \cite{Smith82} to sweepouts consisting of embedded spheres. Colding-De Lellis \cite{Colding-DeLellis03} generalized the Simon-Smith result to sweepouts by Heegaard splittings, and genus upper bounds were obtained by De Lellis-Pellandini \cite{DeLellis-Pellandini10} and Ketover \cite{Ketover19}. On the other hand, there are also the mapping approaches used by the second author \cite{Zhou10, Zhou17b} and Rivi\'ere \cite{Riviere17} to produce branched immersed minimal surfaces of bounded genus. Finally, we note that the CMC min-max theory by Zhou-Zhu \cite{Zhou-Zhu19, Zhou-Zhu20} follows the frameworks of Almgren-Pitts while using sweepouts by Caccioppoli sets. Unfortunately, the Simon-Smith theory cannot be adapted in the same way, as the crucial Meeks-Simon-Yau regularity of area minimizers in an isotopy class \cite{Meeks-Simon-Yau82} still remains open in the CMC context. (This has been claimed to hold true by Yau in \cite{Yau01}.) 

\subsection*{Overview of proofs}

We employ the mapping approach. Given a prescribed value $H>0$, the solutions of \eqref{equ:CMC equation1} are, at least formally, critical points of the following weighted Dirichlet energy functional (compare with~\cite[Equation (2.3)]{Struwe88})
\[
E_H(u) = \frac{1}{2} \int_{S^2} |\nabla u|^2 + H\cdot V(f_u) = : D(u) + H \cdot V(f_u),
\]
where $V(f_u)$ is the ``enclosed volume" of the map $u: S^2\to (S^3, g)$, which is well defined up to an integer multiple of the volume $\Vol(S^3, g)$; see Section \ref{SS:Definition and preliminaries} for more details.

There are two main difficulties in establishing a critical point theory for $E_H(\cdot)$. First of all, as is well-known, the Dirichlet energy is conformally invariant. Secondly, as the volume term $V(\cdot)$ can be very negative, the functional $E_H$ is not bounded from below, and nor does a bound on $E_H$ necessarily imply a bound on the Dirichlet energy.  To overcome the first difficulty, we perturb the Dirichlet energy by replacing it with 
\[
D_{\ep}(u) = \frac{1}{2}\int_{S^2}\ep^2 |\Delta u|^2 + |\nabla u|^2,
\]
as done by Lamm~\cite{Lamm} in the case $H = 0$; see also the work of Chang-Wang-Yang~\cite{CWY} and Wang \cite{Wang04MathZ, Wang04CVPDE, Wang} for fundamental results related to the regularity of the new functionals.

 The perturbation nonetheless doesn't resolve the second difficulty, and moreover introduces the task of deriving uniform estimates in order to pass to the limit as $\ep$ goes to zero. To address the remaining issues,  we use the monotonicity of the min-max values with respect to the mean curvature $H$, motivated by Struwe \cite{Struwe88}. Specifically, we show that the derivatives of the min-max values with respect to $H$ are bounded uniformly in $\ep$ for almost every $H>0$. We then establish, for each such $H$, a uniform bound on $D_{\ep}$ for min-max sequences of the perturbed functional that converge to the min-max value sufficiently fast. Combining this with two deformation arguments, a standard one using pseudo-gradient vector fields and a significantly more technical one inspired by the work of Marques-Neves~\cite{Marques-Neves16} and Song~\cite{Song19} in the Almgren-Pitts setting, we produce non-trivial critical points $u_\ep$ of the perturbed functional with $D_{\ep}(u_\ep)$ uniformly bounded from above and Morse index bounded by $1$.
 
Next, to study the limit as $\ep \to 0$, we adapt the analysis done by Lamm~\cite{Lamm} to obtain higher-order estimates on $u_\ep$ independent of $\ep$ where $D_{\ep}$ does not concentrate. An important consequence of these estimates is a positive lower bound on $D_{\ep}$, again independent of $\ep$, for non-trivial critical points. We now have a familiar dichotomy: on the one hand, if $D_{\ep}$ does not concentrate anywhere, then a subsequence of $u_\ep$ converges smoothly on $S^2$ to a solution of the CMC equation, which directly inherits the Morse index upper bound of $u_{\ep}$ and has to be non-constant thanks to the uniform lower bound. On the other hand, if $D_{\ep}$ concentrates somewhere, then a suitable rescaling of $u_{\ep}$ sub-converges smoothly to a non-constant, finite-energy solution of the CMC equation on $\RR^2$, which gives rise to a non-constant solution on $S^2$ by conformal invariance and removable singularity. Here an important point is that the parameter $\ep$ also gets rescaled. To ensure that the rescaled parameters still converge to zero, so that in the limit we get a solution to the unperturbed equation, we adapt a computation from~\cite{Lamm}, which essentially rules out non-constant solutions to the perturbed equations on $\RR^2$ with $D_{\ep}$ finite. The Morse index upper bound passes to the limit as well in this case, thanks to the logarithmic cut-off trick (see for instance Micallef-Moore~\cite{Micallef-Moore88}). To sum up, whether or not there is concentration of $D_{\ep}$, we obtain a non-constant solution to the CMC equation on $S^2$ with Morse index at most $1$, thereby finishing the proof of Theorem~\ref{thm:main1}. (Recall that weak conformality~\eqref{equ:CMC equation2} is automatic on $S^2$.)

The index bound obtained above figures prominently in the proof of Theorem~\ref{thm:main2}, which proceeds by approximation from the mean curvature values yielded by Theorem~\ref{thm:main1}, and consists mainly of two parts. First we extend a calculation of Eijiri-Micallef~\cite{EM} to transfer the Morse index bound to another bilinear form $B_H(u)$, which is a suitably modified version of the second variation of area. Secondly, the curvature assumption together with a standard conformal balancing argument (see for example the work of Li-Yau~\cite{LiYau}) applied to $B_H(u)$ yield an a priori bound on $D(u)$ for index $1$ solutions to the CMC equation. This bound and a straightforward modification of the $\ep$-regularity theorem for harmonic maps (see for instance Schoen~\cite[Theorem 2.2]{Sch}) in turn allow us finish the proof of Theorem~\ref{thm:main2} in more or less the same way as in Theorem~\ref{thm:main1}.

A couple of comments are in order. First of all, note that to get index control when producing the critical points $u_\ep$ above, we opted to bypass high-index critical points directly with a deformation procedure, rather than perturbing the functional a second time and applying Morse theory as in~\cite{Micallef-Moore88}. This is because in our case the latter approach would be complicated by the fact that, aside from its measure, we have little information about the set of mean curvatures for which there is a uniform $D_\ep$-bound on suitable min-max sequences, and it is unclear to us how to choose the second perturbation to have non-degeneracy of critical points and the uniform $D_\ep$-bound simultaneously.

Secondly, if we follow Sacks-Uhlenbeck~\cite{Sacks-Uhlenbeck81} and replace the Dirichlet energy instead by 
\[
\frac{1}{2}\int_{S^2}(1 + |\nabla u|^2)^{\alpha}, \quad \alpha>1,
\]
when perturbing $E_{H}$, then much of the analysis in~\cite{Sacks-Uhlenbeck81} would still hold. However, in the case that energy concentrates as $\alpha \to 1$, it is unclear to us whether the limit of the rescaled maps would be a solution to the CMC equation when $H \neq 0$. This boils down to comparing $\alpha - 1$ and the rate of rescaling, which seems harder than the corresponding step in the approach described above. (After the completing this work, we learned that Lamm \cite{Lamm10} has obtained a comparison between $\alpha-1$ and the rescaling radius under an entropy-type condition.)

\subsection*{Organization}

In Section~\ref{S:perturbed functional} we set up some notation before defining the enclosed volume and the perturbed functional $E_{H, \ep}$ and establishing their important properties in Sections~\ref{SS:Definition and preliminaries} and~\ref{subsec:loc-reduct}. In Section~\ref{subsec:1st-2nd variations} we derive the first and second variation formula of $E_{H, \ep}$. In Section~\ref{SS:smoothness} we show that critical points of $E_{H, \ep}$ are smooth and establish a version of the Palais-Smale condition for $E_{H, \ep}$. We return to the second variation of $E_{H, \ep}$ in Section~\ref{subsec:Morse} and show that there exist ``generalized Morse neighborhoods'' around critical points.

Section~\ref{S:non-trivial-critical-points} is devoted to finding non-constant critical points of the perturbed functionals with $D_{\ep}$ bounded independently of $\ep$, and with index at most $1$. In Section~\ref{SS:sweepouts} we introduce the class of admissible sweepouts and define the min-max values. Section~\ref{SS:struwe} establishes, for almost every $H$, a uniform bound on $D_{\ep}$ for suitable min-max sequences of $E_{H, \ep}$, as described in the overview. Then, in Section~\ref{SS:existence-non-trivial-critical}, we show how to extract non-constant critical points $u_\ep$ of $E_{H, \ep}$ out of these min-max sequences. Section~\ref{subsec:deformation} opens with the technical deformation lemma mentioned in the overview, which is then combined with results from the two previous sections to yield critical points with the desired properties.

In Section~\ref{S:convergence}, we analyze $u_{\ep}$ as $\ep \to 0$ and conclude the proof of Theorem~\ref{thm:main1}. The key estimate is Proposition~\ref{prop:eta-regularity}, which gives higher-order estimates under a smallness assumption and implies strong subsequential convergence of $u_{\ep}$ away from energy concentration points. The proof of the Theorem~\ref{thm:main1} is concluded at the end of the Section.

Section~\ref{sec:improved} is devoted to the proof of Theorem~\ref{thm:main2}. The two main ingredients, namely an index comparison result and an a priori energy bound for index $1$ critical points, are established respectively in Section~\ref{subsec:index-comparison} and Section~\ref{subsec:uniform-energy-bound}. The proof of Theorem~\ref{thm:main2} is completed at the very end.

Finally, Appendix~\ref{sec:proof of two lemmas} collects a number of standard estimates we need, along with indications of their proofs, for the reader's convenience.

\vspace{1em}
{\bf Acknowledgement}: X. Z. is partially supported by NSF grant DMS-1811293, DMS-1945178, and an Alfred P. Sloan Research Fellowship.  We would also like to thank Tobias Lamm for comments, especially for pointing out the reference \cite{Lamm10}. Thanks also go to Andre Neves for helpful conversations on constant mean curvature spheres. Finally we are grateful to the anonymous referees for very helpful comments.

\section{The perturbed functional}
\label{S:perturbed functional}

Below, $S^2$ denotes the $2$-sphere with the standard round metric, and $(S^3, g)$ denotes a Riemannian manifold diffeomorphic to the standard 3-sphere. We assume that $(S^3, g)$ is isometrically embedded in $\RR^N$ for some large $N\in \NN$, and denote by $A$ the second fundamental form of this embedding. Moreover, we let $\cV$ be a tubular neighborhood of $S^3$ in $\RR^N$ and assume that the nearest-point projection, denoted $\Pi: \cV \to S^3$, has bounded derivatives of all orders on $\cV$. For brevity, we write $P$ for the differential $d\Pi: \cV \to \RR^{N \times N}$. In particular, for $y \in \cV$, $P_y$ is the orthogonal projection onto $T_{\Pi(y)}S^3$. For a map $v : S^2 \to \cV$, we write $P_{v}$ for the composition $P \circ v : S^2 \to \RR^{N \times N}$.

The metric $g$ and the volume form $\Vol_g$ of $(S^3, g)$ allow us to define a cross product $Q \in \Gamma(\Lambda^2T^{\ast}S^3\otimes TS^3)$, which is a skew-symmetric bilinear form on $TS^3$ with value in $TS^3$, and is given by the following relation:
\begin{equation}\label{eq:cross-product}
\Vol_g(X, Y, Z) = g (Q(X, Y), Z) = Q(X, Y) \cdot Z.
\end{equation}
Here ``$\cdot$'' denotes the standard inner product on $\RR^N$. The right-hand side of the CMC equation~\eqref{equ:CMC equation1} can then be understood as a map $S^2 \to \RR^N$ since the Hodge star operator on $S^2$ transforms the pullback $u^{\ast}Q \in \Gamma(\Lambda^2 T^{\ast}S^2 \otimes u^{\ast}TS^3)$ into a section $\ast(u^{\ast}Q): S^2 \to u^{\ast}TS^3$, which we identify with an $\RR^N$-valued function on $S^2$ via the embedding $S^3 \to \RR^N$.

\subsection{Definition and preliminaries}
\label{SS:Definition and preliminaries}
As usual we define the Sobolev space 
\[ W^{2, 2}(S^2; S^3) = \{u \in W^{2, 2}(S^2; \RR^N)\ |\ u(x) \in S^3 \text{ for all }x \in S^2\}, \]
and equip it with the subspace topology coming from the $W^{2, 2}$-norm
\[
\|u\|_{2, 2}^2 := \int_{S^2}|u|^2 + |\nabla u|^2 + |\nabla^2 u|^2 d\Vol_{S^2} \text{, for }u \in W^{2, 2}(S^2; \RR^N),
\]
where the tensor norms on the right are with respect to the standard metric on $S^2$, and $\nabla$ is the Levi-Civita connection on $S^2$, applied to $u$ componentwise. The norms $\|\cdot\|_{k, p}$ for other choices of $k, p$ are defined similarly. Thanks primarily to the Sobolev embedding $W^{2, 2}(S^2; \RR^N) \to C^{0}(S^2; \RR^N)$, the space $W^{2, 2}(S^2; S^3)$ is a smooth, closed submanifold of $W^{2, 2}(S^2; \RR^N)$. For each $u \in W^{2, 2}(S^2; S^3)$, the tangent space of $W^{2, 2}(S^2, S^3)$ at $u$ can be identified with
\[
\cT_{u} = \{\psi \in W^{2, 2}(S^2; \RR^N)\ |\ \psi(x) \in T_{u(x)}S^3 \text{ for all }x \in S^2\},
\]
which is closed in $W^{2, 2}(S^2; \RR^N)$ and hence is a Hilbert space itself. For small enough balls $\cB_{u}$ around the origin in $\cT_u$, we define the maps $\Theta_u: \cB_{u} \to W^{2, 2}(S^2; S^3)$ by
\[
\Theta_u(\psi) = \Pi(u + \psi).
\]
Then the collection $\{(\Theta_u, \cB_u)\}$ is a smooth atlas for $W^{2, 2}(S^2; S^3)$. Restricting the inner product on $W^{2, 2}(S^2; \RR^N)$ to each $\cT_u$ gives a Riemannian structure on $W^{2, 2}(S^2; S^3)$. 

For $u \in W^{2, 2}(S^2; S^3)$, we let 
\[
\cE(u) = \{f \in C^{0}([0, 1] \times S^2; S^3)\ |\ f(0, \cdot) = \text{ constant},\ f(1, \cdot) = u\},
\]
which may be thought of as the set of continuous extensions of $u$ to a map from $B^3$ to $S^3$. We now define the perturbed functionals.

\begin{defi}
For $H > 0$ and $\ep > 0$ and $u \in W^{2, 2}(S^{2}; S^{3})$, $f \in \cE(u) \cap W^{1, 3}([0, 1] \times S^2; S^3)$, we define
\begin{equation}\label{eq:main perturbed functional}
E_{H, \ep}(u, f) = D_{\ep}(u) + H\cdot V(f),
\end{equation}
where 
\begin{equation}\label{eq:D_epsilon}
D_{\ep}(u) = \frac{1}{2}\int_{S^2}\ep^2 |\Delta u|^2 + |\nabla u|^2,
\end{equation}
and 
\begin{equation}\label{eq:enclosed volume}
V(f) = \int_{[0, 1] \times S^2}f^{\ast}\Vol_g.
\end{equation}
\end{defi}
Here $D_{\ep}(u)$ is a biharmonic regularization of the Dirichlet energy $D(u)=\frac{1}{2}\int_{S^2}|\nabla u|^2$. On the other hand, the term $V(f)$ can be viewed as the signed volume enclosed by $u(S^2)$, which is well-defined since $f \in W^{1, 3}([0, 1] \times S^2; S^3)$ by assumption. We note that a similar quantity was already introduced by Struwe~\cite{Struwe88} in the free-boundary case. Our usage of the term ``extension'', however, differs slightly from his. We summarize some basic properties of the enclosed volume below. (Compare with~\cite[p.24]{Struwe88}.)

\begin{lemm}\label{lemm:volume-properties} 
Fix any $u \in W^{2, 2}(S^2; S^3)$.
\begin{enumerate}
\item[(a)] If $f_0, f_1 \in \cE(u) \cap W^{1, 3}([0, 1] \times S^2; S^3)$ are two choices of extensions, then 
\[
\frac{V(f_0) - V(f_1)}{\Vol_g(S^3)} \in \ZZ.
\]
\vskip 1mm
\item[(b)] If $F:[0, 1] \to \cE(u) \cap W^{1, 3}([0, 1] \times S^2; S^3)$ is a path which is continuous in the $C^0 \cap W^{1, 3}$-topology, then 
\[
V(F(0)) = V(F(1)).
\]
\vskip 1mm
\item[(c)] There exists a universal constant $\delta_0$ such that if $f_0, f_1 \in \cE(u) \cap W^{1, 3}([0, 1] \times S^2; S^3)$ and $\|f_0 - f_1\|_{C^{0}} < \delta_0$ then $V(f_0) = V(f_1)$.
 \end{enumerate}
\end{lemm}

\begin{proof}
For part (a), we consider the concatenation of $f_0$ and the reverse of $f_1$. That is, the map
\[
g(t, \cdot) = \left\{
\begin{array}{cc}
f_0(2t, \cdot) & \text{ for }0 \leq t \leq 1/2,\\
f_1(2 - 2t, \cdot) & \text{ for }1/2 \leq t \leq 1.
\end{array}
\right.
\]
Then $g$ again lies in $C^0 \cap W^{1, 3}([0, 1] \times S^2; S^3)$ and induces a $C^0$-map from $S^3$ to itself, whose degree can be computed by $\frac{V(f_0) - V(f_1)}{\Vol_g(S^3)}$. Hence the latter is an integer.
\vskip 1mm
For part (b), we note that by assumption and part (a), the function $t \mapsto V(F(t))$ is continuous and takes values in the discrete set $\{V(F(0)) + k\Vol_g(S^3)\ |\ k \in \ZZ\}$, and hence must be constant. 
\vskip 1mm
To prove part (c), we choose $\delta_0$ small enough so that $\{y \in \RR^N\ |\ \dist(y, S^3) \leq 2\delta_0\} \subset \cV$. Then $tf_1 + (1 - t)f_0$ maps into $\cV$ for $t \in [0, 1]$, and the result follows by applying part (b) to $F(t) = \Pi \big( tf_1 + (1 - t)f_0 \big)$, which defines a map $[0, 1] \to \cE(u) \cap W^{1, 3}([0, 1]\times S^2; S^3)$ which is continuous in the $C^0\cap W^{1, 3}$-topology. 
\end{proof}

\begin{defi}
Part (c) of the previous lemma allows us to define $V(f)$ for $f \in \cE(u)$ which are not necessarily in $W^{1, 3}([0, 1] \times S^2; S^3)$, by letting 
\[
V(f) = V(\tilde{f}),
\]
with $\tilde{f}$ a choice of extension in $\cE(u) \cap W^{1, 3}([0, 1] \times S^2; S^3)$ such that $\|f - \tilde{f}\|_{C^{0}} < \delta_0/2$, where $\delta_0$ is as in Lemma~\ref{lemm:volume-properties}(c), by which we see that $V(f)$ is well-defined. 
\end{defi}

\begin{lemm}\label{lemm:volume-extension}
With the definition of $V$ extended as above, the conclusions of the previous lemma continue to hold with $\cE(u)$ in place of $\cE(u) \cap W^{1, 3}([0, 1] \times S^2; S^3)$, and with $C^{0}$-topology in place of $C^0 \cap W^{1, 3}$-topology in part (b).
\end{lemm}
\begin{proof} 
That (a) and (c) continue to hold is obvious. Part (b) follows by partitioning $[0, 1]$ sufficiently finely and applying part (c) repeatedly.
\end{proof}

Below we collect some basic estimates concerning $D_{\ep}$, the $W^{2, 2}$-norm and the projection $\Pi$ that will be used later. The proofs are recorded in Appendix \ref{sec:proof of two lemmas}.

\begin{lemm}\label{lemm:D-equivalence} 
There exists a universal constant $A_0$ such that for all $\ep \in (0, 1]$ and $u \in W^{2, 2}(S^2; S^3)$ there hold
\begin{enumerate}
\item[(a)] $\big| \int_{S^2}|\nabla^2 u|^2 - |\Delta u|^2 \big| \leq A_0 \int_{S^2}|\nabla u|^2$. (See also~\cite[Equation (2.6)]{Lamm}.)
\vskip 1mm
\item[(b)] $A_{0}^{-1}D_{\ep}(u) \leq \int_{S^2}|\nabla u|^2 + \ep^2 |\nabla^2 u|^2 \leq A_0 D_{\ep}(u)$ and  $\|u\|_{2, 2}^2 \leq A_0 \big( 1 + \ep^{-2}D_{\ep}(u)  \big)$.
\end{enumerate}
Conclusion (a) in fact holds for all $u \in W^{2, 2}(S^2; \RR^N)$.
\end{lemm}

\begin{lemm}
\label{lemm:projection-estimates}
Let $\cV, \Pi$ and $P$ be defined as in the beginning of the section. Then there exists some universal constant $A_1$ such that the following estimates hold for all $\ep \in (0, 1]$.
\begin{enumerate}
\item[(a)] For $\widetilde{v} \in W^{2, 2}(S^2; \cV)$, denote $v = \Pi(\widetilde{v})$. Then $v \in W^{2, 2}(S^2; S^3)$, and we have
\[
\int_{S^2} |\nabla v|^2 + \ep^2 |\nabla^2 v|^2 \leq A_1\big( 1 + \|\widetilde{v}\|_{2, 2}^2 \big) \int_{S^2}|\nabla\widetilde{v}|^2 + \ep^2 |\nabla^2\widetilde{v}|^2.
\]
Moreover, the projection $\Pi$, when viewed as a map $W^{2, 2}(S^2; \cV) \to W^{2, 2}(S^2; S^3)$, is locally Lipschitz. 
\vskip 1mm
\item[(b)] The map $u \mapsto P_u$ is locally Lipschitz from $W^{2, 2}(S^2; S^3)$ to $W^{2, 2}(S^2; \RR^{N \times N})$. For all $u \in W^{2, 2}(S^2; S^3)$ and $\widetilde{\psi} \in W^{2, 2}(S^2; \RR^N)$, write $\psi = P_u(\widetilde{\psi})$. Then $\psi \in \cT_u$, and we have
\[
\|\psi\|_{2, 2} \leq A_1\big( 1 + \|u\|_{2, 2}^2 \big)\|\widetilde{\psi}\|_{2, 2}.
\]
Moreover, for each fixed $\widetilde{\psi} \in W^{2, 2}(S^2; \RR^N)$, the map $u \mapsto P_u(\widetilde{\psi})$ from $W^{2, 2}(S^2; S^3)$ to $W^{2, 2}(S^2; \RR^N)$ is continuous.
\end{enumerate}
\end{lemm}

\subsection{Local reduction}\label{subsec:loc-reduct}
In this section we show how we may locally reduce the number of variables in $E_{H, \ep}$ so that it only depends on the map $u$. To set the stage, let $\cA$ be a simply-connected open set in $W^{2, 2}(S^2; S^3)$. Given $u_0 \in \cA$ and $f_0 \in \cE(u_0)$, for all $u \in \cA$, there is by connectedness a path $h: [0, 1] \to \cA$ with $h(0) = u_0$ and $h(1) = u$, and we denote by $f_u$ the extension in $\cE(u)$ obtained by concatenating $f_0$ with the map $(t, x) \mapsto h(t)(x)$, which is continuous by Sobolev embedding. We then define 
\[
E^{\cA}_{H, \ep}(u) = E_{H, \ep}(u, f_u).
\]

\begin{prop}\label{prop:local-reduct}
In the above notation, we have
\vskip 2mm
\begin{enumerate}
\item[(a)] $E^{\cA}_{H, \ep}(u)$ is well-defined for all $u \in \cA$. That is, the choice of the path $h$ is irrelevant.
\vskip 1mm
\item[(b)] $E^{\cA}_{H, \ep}$ is a $C^{2}$-functional on $\cA$. 
\end{enumerate}
\vskip 2mm
\end{prop}

\begin{proof}
Part (a) follows easily from the simply-connectedness of $\cA$ and Lemma~\ref{lemm:volume-properties}(b). For part (b), we fix $u \in \cA$ and take a local chart $(\Theta_u, \cB_u)$ as defined in the beginning of Section~\ref{SS:Definition and preliminaries}. For brevity, below we drop the subscript $u$ in $\Theta_u$ and $\cB_u$. 

By definition, to show that $E^{\cA}_{H, \ep}$ is $C^2$ near $u$, it suffices to show that $E^{\cA}_{H, \ep}\circ \Theta$ is $C^2$ on $\cB$. Using the path-independence established in (a), we may replace the extension $f_{\Theta(\psi)} \in \cE(\Theta(\psi))$ by the concatenation of $f_u \in \cE(u)$ with $(t, x) \mapsto \Pi(u(x) + t\psi(x))$. Denoting $\widetilde{f}(t, x) = \Pi(u(x) + t\psi(x))$, then we find that 
\begin{align}\label{eq:E-compose-chart}
E^{\cA}_{H, \ep}(\Theta(\psi)) =&\ \frac{1}{2}\int_{S^2}\ep^2 |\Delta (\Theta(\psi))|^2 + |\nabla (\Theta(\psi))|^2 \nonumber\\
&\ + H \cdot \int_{[0, 1] \times S^2} (\Vol_g)_{\widetilde{f}}\big( \widetilde{f}_t, \widetilde{f}_{x^1}, \widetilde{f}_{x^2}\big) dx^1 \wedge dx^2 dt\nonumber\\
&\  + H \cdot V(f_u).  
\end{align}
The first line on the right-hand side of~\eqref{eq:E-compose-chart} is $C^2$ on $\cB$ since it is the composition of a bounded quadratic form on $W^{2, 2}(S^2; \RR^N)$ with $\psi \mapsto \Theta(\psi)$, which is $C^2$ as a map from $\cB$ into $W^{2, 2}(S^2; \RR^N)$. (See for instance~\cite[Theorem 11.3, also Corollary 9.7, Lemma 9.9]{Pal}.) 

For the second line on the right-hand side of~\eqref{eq:E-compose-chart}, we let $\{\paop{y^i}\}_{i = 1}^N$ be the coordinate vectors on $\RR^N$ and define, for $i,j, k = 1, \cdots N$, the functions $\theta_{ijk}:\cV \to \RR$ by letting
\[
\theta_{ijk}(y) = (\Vol_g)_{\Pi(y)}\big( P_y(\paop{y^i}), P_y(\paop{y^j}), P_y(\paop{y^k}) \big).
\]
Note that each $\theta_{ijk}$ is smooth, 
and that for a map $w \in W^{2, 2}([0, 1] \times S^2; \cV)$, writing $dw = dw^i \paop{y^i}$, we have
\begin{align*}
(\Vol_g)_{\Pi(w)} (P_w(w_t), P_w(w_{x^1}), P_w(w_{x^2})) dx^1\wedge dx^2
= \theta_{ijk}(w)w^i_t w^j_{x^1} w^k_{x^2} dx^1\wedge dx^2.
\end{align*}
It is then not hard to see from the smoothness of $\theta_{ijk}$ that the following functional,
\begin{equation}\label{eq:volume-coordinates}
w \mapsto \int_{[0, 1]} \int_{S^2} \theta_{ijk}(w)w^i_t w^j_{x^1} w^k_{x^2} dx^1\wedge dx^2 dt,
\end{equation}
defined for $w \in W^{2, 2}([0, 1] \times S^2; \cV)$, is $C^2$. To finish, note that sending $\psi$ to the function $(t, x) \mapsto u(x) + t\psi(x)$ defines a smooth map from $\cB$ to $W^{2, 2}([0, 1] \times S^2; \cV)$.  
\end{proof}

\begin{rmk}\label{rmk:local-reduct}
We have the following remarks ready.
\begin{enumerate}
\item[(1)]
We call $E^{\cA}_{H, \ep}$ the {\em local reduction} of $E_{H, \ep}$ on $\cA$ induced by $(u_0, f_0)$. Note that we are suppressing from the notation the dependence on $(u_0, f_0)$, since the choice should always be clear from the context. 
\vskip 1mm
\item[(2)] By Lemma~\ref{lemm:volume-properties}(a) and Proposition~\ref{prop:local-reduct}(b), changing the choice of $u_0 \in \cA$ and $f_0 \in \cE(u_0)$ used to define the local reduction merely alters its value by an integer multiple of $H\cdot\Vol_g(S^3)$. More generally, on any connected subset of their common domain, two local reductions differ by a constant integer multiple of $H\cdot\Vol_g(S^3)$.  
\end{enumerate}
\end{rmk}

\subsection{The first and second variations of $E_{H, \ep}$}\label{subsec:1st-2nd variations}
Next we compute the first and second variations of $E_{H, \ep}$. We carry out the computation by first choosing a local reduction, but this choice turns out to be irrelevant, and the variations make sense globally.

We begin by explaining the framework for the computation.  Let $\cA \in W^{2, 2}(S^2; S^3)$ be a simply-connected open set on which a local reduction $E^{\cA}_{H, \ep}$ is defined. Since $E^{\cA}_{H, \ep}$ is a $C^{2}$-functional, at each $u \in \cA$ it has a differential, which is a bounded linear functional on $\cT_u$, and at critical points it has a well-defined Hessian, which is a bounded symmetric bilinear form on $\cT_u$. The action of the differential of $E^{\cA}_{H, \ep}$ at $u$, denoted $\delta E^{\cA}_{H, \ep}(u): \cT_u \to \RR$, may be seen by computing
\begin{equation}\label{eq:first-var-compute-scheme}
\delta E^{\cA}_{H, \ep}(u)(\psi) = \frac{d}{dt}\Big|_{t = 0}E^{\cA}_{H, \ep}(\Pi(u + t\psi)),\ \text{ for }\psi \in \cT_u.
\end{equation}
If $\delta E^{\cA}_{H, \ep}(u) = 0$, then the Hessian of $E^{\cA}_{H, \ep}$ at $u$, denoted $\delta^2 E^{\cA}_{H, \ep}: \cT_u \times \cT_u \to \RR$, can be computed by 
\begin{equation}\label{eq:second-var-compute-scheme}
\delta^2 E^{\cA}_{H, \ep}(u)(\psi, \psi) = \frac{d^2}{dt^2}\Big|_{t = 0}E^{\cA}_{H, \ep}(\Pi(u + t\psi)),\ \text{ for }\psi \in \cT_u.
\end{equation}
We are now ready to define the differential and Hessian of $E_{H, \ep}$.
\begin{defi}\label{defi:differential}
For $u \in W^{2, 2}(S^2; S^3)$, we define $\delta E_{H, \ep}(u): \cT_u \to \RR$ by letting 
\[
\delta E_{H, \ep}(u) = \delta E^{\cA}_{H, \ep}(u), 
\]
where $E^{\cA}_{H, \ep}$ is any local reduction on a simply-connected neighborhood $\cA$ that contains $u$. Note that such a neighborhood always exists since $W^{2, 2}(S^2; S^3)$ is a manifold. Moreover, $\delta E_{H, \ep}(u)$ is well-defined in view of the formula~\eqref{eq:first-var-compute-scheme} and Remark~\ref{rmk:local-reduct}(2). 
\end{defi}
\begin{defi}\label{defi:critical-point}
A map $u \in W^{2, 2}(S^2; S^3)$ is called a critical point of $E_{H, \ep}$ if $\delta E_{H, \ep}(u) = 0$. 
\end{defi}
\begin{defi}\label{defi:hessian}
Suppose $u$ is a critical point of $E_{H, \ep}$. Then by Definition~\ref{defi:critical-point}, for any local reduction $E^{\cA}_{H, \ep}$ on a neighborhood containing $u$, the Hessian $\delta^2 E^{\cA}_{H, \ep}(u)$ makes sense, and we define the Hessian of $E_{H, \ep}$ at $u$ by 
\[
\delta^2 E_{H, \ep}(u) = \delta^2 E^{\cA}_{H, \ep}(u).
\]
As with the differential, the Hessian is well-defined by~\eqref{eq:second-var-compute-scheme} and Remark~\ref{rmk:local-reduct}(2).
\end{defi}

\vspace{1mm}
Proceeding to the actual computation, we take $u \in W^{2, 2}(S^2; S^3)$, fix an extension $f \in \cE(u)$ and consider the local reduction induced by $(u, f)$ on a simply-connected neighborhood $\cA$ of $u$. For $\psi \in \cT_u$ and $t$ sufficiently small, using the extension of $\Pi(u + t\psi)$ obtained by concatenating $f$ with the map $(s, x) \mapsto \Pi(u + st\psi)$, we find that 
\begin{align}
&E^{\cA}_{H, \ep}(\Pi(u + t\psi)) - E^{\cA}_{H, \ep}(u) \label{eq:reduct-difference-quotient} \\
 =&\ D_{\ep}(\Pi(u + t\psi)) - D_{\ep}(u)\nonumber \\
&\ +H  \int_{0}^t\Big[\int_{S^2} (\Vol_{g})_{\Pi(u + s\psi)}\big( P_{u + s\psi} (\psi), (\Pi(u + s\psi))_{x^1}, (\Pi(u + s\psi))_{x^2} \big) dx^1 \wedge dx^2 \Big] ds. \nonumber
\end{align}
It is not hard to see that the $t$-derivative of the last integral at $t = 0$ is equal to
\begin{equation}\label{eq:V-variation}
H \int_{S^2} (\Vol_g)_u \big( P_u (\psi), u_{x^1}, u_{x^2}  \big) dx^1 \wedge dx^2 =H \int_{S^2} \psi \cdot \ast(u^{\ast}Q), 
\end{equation}
where the equality follows because $\psi \in \cT_u$.  
On the other hand, it is well-known that (see for instance~\cite{Wang, Lamm2})
\begin{align}\label{eq:D-variation}
\frac{d}{dt}\Big|_{t = 0}D_{\ep}(\Pi(u + t\psi)) 
&= \int_{S^2} \ep^2 \Delta u \cdot \Delta \psi + \langle \nabla u, \nabla \psi \rangle, \text{ for }\psi \in \cT_u.
\end{align}
Adding up~\eqref{eq:V-variation} and~\eqref{eq:D-variation} and recalling Definition~\ref{defi:differential}, we obtain the following first variation formula for $E_{H, \ep}$:
\begin{equation}\label{eq:first-var-tangential}
\delta E_{H, \ep}(u)(\psi) = \int_{S^2} \ep^2 \Delta u \cdot \Delta \psi + \langle \nabla u, \nabla \psi \rangle + H \psi \cdot \ast(u^{\ast}Q).
\end{equation}

To establish the regularity of critical points and derive a priori estimates in later sections, we also need a formula for $\delta E_{H, \ep}(u)(P_u(\psi))$, with $\psi$ varying in the bigger space $W^{2, 2}(S^2; \RR^N)$. To reduce notation, we write 
\[
G_{H, \ep}(u) = \delta E_{H, \ep} (u)\circ P_u,
\]
and define the norm of $G_{H, \ep}(u): W^{2, 2}(S^2; \RR^N) \to \RR$ by 
\begin{equation*}\label{eq:extension-norm}
\| G_{H, \ep}(u) \| =\sup \big\{ | G_{H, \ep}(u)(\psi)|  \ \big|\ \psi \in W^{2, 2}(S^2; \RR^N),\ \|\psi \|_{2, 2} \leq 1 \big\}.
\end{equation*}

\begin{prop}\label{prop:first-var-properties}
We have the following properties for $G_{H, \ep}$.
\begin{enumerate}
\item[(a)] (\cite[Proposition 2.2]{Wang})  
Given $u \in W^{2, 2}(S^2; S^3)$, we have for any $\psi \in W^{2, 2}(S^2; \RR^N)$ that 
\begin{align}
G_{H, \ep}(u)(\psi) =\ &\ep^2\int_{S^2} \Delta u \cdot \Delta \psi - A(u)(\nabla u, \nabla u) \cdot \Delta \psi +2 \langle \nabla (P_u)(\Delta u), \nabla \psi\rangle + \Delta (P_u)(\Delta u) \cdot \psi\nonumber\\
&+ \int_{S^2}\langle \nabla u, \nabla \psi \rangle + A(u)(\nabla u, \nabla u)\cdot \psi + H \int_{S^2} \psi \cdot \ast (u^{\ast}Q).\label{eq:first-var-weak}
\end{align}
\vskip 1mm
\item[(b)] For all $K > 0$ there exists $C_{K} > 0$ such that 
\[
\| G_{H, \ep}(u)  \| \leq C_K,
\]
whenever $\ep \leq 1$ and $u \in W^{2, 2}(S^2; S^3)$ with $H,\,  \|u\|_{2, 2} \leq K$. 
\vskip 1mm
\item[(c)] For all $K > 0$, there exists $C_{K} > 0$ such that
\[
\| G_{H, \ep}(u_1) - G_{H, \ep}(u_0)\| \leq C_{K} \|u_1 - u_0\|_{2, 2},
\]
whenever $\ep \leq 1$ and $H,\, \|u_0\|_{2, 2},\, \|u_1\|_{2, 2} \leq K$. 
\end{enumerate}
\end{prop}
\begin{proof}
For part (a), assume first that $u \in C^{\infty}(S^2; S^3)$, then we see from~\eqref{eq:first-var-tangential} that
\begin{equation}\label{eq:Euler-Lagrange-unsplit}
G_{H, \ep}(u)(\psi) = \int_{S^2}\big( \ep^2 P_u(\Delta^2 u) - P_u (\Delta u) + H \ast (u^{\ast}Q)\big) \cdot \psi.
\end{equation}
Now recall that 
\begin{equation}\label{eq:Delta-split}
P_u (\Delta u) = \Delta u - A(u)(\nabla u, \nabla u).
\end{equation}
On the other hand, as mentioned in~\cite{Wang}, the term $P_u (\Delta^2 u)$ can be computed as follows. We first note that by the above identity,
\begin{equation}\label{eq:Delta-2-split}
\Delta^2 u = \Delta \big( A(u)(\nabla u, \nabla u) + P_{u} (\Delta u) \big).
\end{equation}
For the second term on the right-hand side, we have
\begin{align*}
\Delta \big( P_u(\Delta u) \big) &= \Delta (P\circ u) (\Delta u) + 2 \nabla (P \circ u)(\nabla \Delta u) + (P \circ u)(\Delta^2 u)\\
&= 2\Div\big( \nabla(P\circ u)(\Delta u) \big) - \Delta(P\circ u)(\Delta u) + (P\circ u)(\Delta^2 u).
\end{align*}
Putting this back into~\eqref{eq:Delta-2-split} and rearranging, we see that
\[
P_u(\Delta^2 u) = \Delta^2 u - \Delta \big( A(u)(\nabla u, \nabla u) \big) -2 \Div \big( \nabla (P\circ u)(\Delta u) \big) + \Delta (P\circ u)(\Delta u).
\]
Substituting this along with~\eqref{eq:Delta-split} back into~\eqref{eq:Euler-Lagrange-unsplit}, and then integrating by parts, we get~\eqref{eq:first-var-weak} when $u$ is smooth. The general case follows by approximation. 
\vskip 1mm
For part (b), with the help of the following Sobolev inequalities applied to $u$ and $\psi$ where appropriate,
\begin{align*}
\|h\|_{\infty} &\leq C \|h\|_{2, 2},\\
 \|h\|_{1, p} &\leq C_{p}\|h\|_{2, 2} \text{ for all } p < \infty,
\end{align*}
we verify from~\eqref{eq:first-var-weak} and H\"older's inequality that there exists $C_{K} > 0$ such that
\[
|G_{H, \ep}(u)(\psi)| \leq C_{K}\|\psi\|_{2, 2},
\]
whenever $\psi \in W^{2, 2}(S^2; \RR^N)$, and $\ep, u$ are as in the statement. This proves (b).

\vskip 1mm
The proof of part (c) is also straightforward. Note that by part (b) and the Sobolev embedding of $W^{2, 2}$ into $C^0$, we need only consider the case where $\|u_1 - u_0\|_{0} < \delta_0$, where $\delta_0$ is as in Lemma~\ref{lemm:volume-properties}(c). Now we use~\eqref{eq:first-var-weak} to see that
\begin{align}
&\ G_{H, \ep}(u_1)(\psi) - G_{H, \ep}(u_0)(\psi) \nonumber\\
=&\  \int_{S^2} \ep^2\big(\Delta u_1 - \Delta u_0\big) \cdot \Delta \psi + \langle \nabla u_1 - \nabla u_0, \nabla \psi \rangle\nonumber\\
&\ +\ep^2\int_{S^2} 2\big\langle \nabla(P_{u_1})(\Delta u_1) - \nabla(P_{u_0})(\Delta u_0), \nabla \psi \big\rangle\nonumber\\
&\  +\ep^2\int_{S^2} \big(\Delta (P_{u_1}) (\Delta u_1) - \Delta (P_{u_0})(\Delta u_0)\big)\cdot \psi\nonumber\\
&\ -\ep^2 \int_{S^2} \big( A_{u_1}(\nabla u_1, \nabla u_1) - A_{u_0}(\nabla u_0, \nabla u_0) \big) \cdot \Delta \psi \nonumber\\
&\ +\int_{S^2} \big( A_{u_1}(\nabla u_1, \nabla u_1) - A_{u_0}(\nabla u_0, \nabla u_0) \big) \cdot \psi \nonumber\\
&\  + H \int_{S^2}\big( \ast(u_1^{\ast}Q) - \ast(u_0^{\ast}Q)  \big) \cdot \psi. \label{eq:first-var-difference}
\end{align}
The first integral on the right-hand side is obviously bounded by $\|u_1 - u_0\|_{2, 2} \|\psi\|_{2, 2}$ since $\ep \leq 1$ by assumption. Next we note that by our choice of $\delta_0$, it makes sense to define $u_{t} = \Pi ( tu_1 + (1 - t)u_0) \ \text{ for }t \in [0, 1]$. We can then estimate the remaining integrals, for example, using the fundamental theorem of calculus together with Sobolev inequalities.
\end{proof}

\vspace{1mm}
We next turn to computing the Hessian of $\delta^2 E_{H, \ep}$, assuming that $\delta E_{H, \ep}(u) = 0$. From ~\eqref{eq:second-var-compute-scheme}, \eqref{eq:reduct-difference-quotient}, ~\eqref{eq:V-variation} and~\eqref{eq:D-variation}, we know that for $\psi \in \cT_u$, we can compute $\delta^2 E_{H, \ep}(u)(\psi, \psi)$ by evaluating the $t$-derivative of the following expression at $t = 0$:
\begin{align*}
 \int_{S^2}&\ep^2 \Delta \big(\Pi(u + t\psi)\big) \cdot \Delta\big( P_{u + t\psi}(\psi) \big) + \langle \nabla\big( \Pi(u + t\psi) \big), \nabla \big( P_{u + t\psi}(\psi) \big) \rangle  \\
 & + H\int_{S^2} P_{u + t\psi}(\psi) \cdot Q_{\Pi(u + t\psi)}\big( (\Pi(u + t\psi))_{x^1}, (\Pi(u + t\psi))_{x^2} \big) dx^1\wedge dx^2.
\end{align*}
Carrying out the differentiation for the first line gives
\begin{equation}\label{eq:second-derivative-1}
\int_{S^2} \ep^2|\Delta \psi|^2 + \ep^2\Delta u \cdot \Delta \big( (dP)_u(\psi, \psi) \big) +  |\nabla \psi|^2 + \langle \nabla u, \nabla \big( (dP)_u(\psi, \psi) \big) \rangle.
\end{equation}
For the second line, using the fact that $Q$ is parallel with respect to the Levi-Civita connection on $S^3$, we get
\begin{equation}\label{eq:second-derivative-2}
H \int_{S^2} (dP)_u(\psi, \psi) \cdot Q_u (u_{x^1}, u_{x^2}) + \psi \cdot \big(  Q(D_{x^1}\psi, u_{x^2}) + Q(u_{x^1}, D_{x^2}\psi) \big) dx^1 \wedge dx^2,
\end{equation}
where the operator $D$ is defined by $D\psi = P_u(\nabla \psi)$. We now add~\eqref{eq:second-derivative-1} and~\eqref{eq:second-derivative-2}, split each occurrence of the term $(dP)_u(\psi, \psi)$ into its tangential and normal component with respect to $TS^3$ as $(dP)_u(\psi, \psi) = P_u\big( (dP)_u(\psi, \psi) \big) + \big( (dP)_u(\psi, \psi) \big)^{\perp}$, and observe that 
\[
\big( (dP)_u(\psi, \psi) \big)^{\perp} = A_{u}(\psi, \psi).
\]
Then we obtain
\begin{align}
\delta^2 & E_{H, \ep}(u)(\psi, \psi) =\ \delta E_{H, \ep}(u)\Big(P_u\big( (dP)_u(\psi, \psi) \big)\Big)\nonumber\\
&\ + \int_{S^2} \ep^2|\Delta \psi|^2 + \ep^2 \Delta u \cdot \Delta \big( A_{u}(\psi, \psi)\big) +  |\nabla \psi|^2 + \langle \nabla u, \nabla \big( A_{u}(\psi, \psi)\big) \rangle \nonumber\\
&\ + H \int_{S^2}  \psi \cdot \big(  Q(D_{x^1}\psi, u_{x^2}) + Q(u_{x^1}, D_{x^2}\psi) \big) dx^1 \wedge dx^2. \label{eq:second-var-long-formula}
\end{align}
Recalling that $u$ is a critical point of $E_{H, \ep}$ by assumption, we see that 
\begin{equation}\label{eq:critical-condition}
\delta E_{H, \ep}(u)\Big(P_u\big( (dP)_u(\psi, \psi) \big)\Big) = 0.
\end{equation}
Also, note that 
\[
\int_{S^2} |\nabla \psi|^2 + \langle \nabla u, \nabla \big( A_{u}(\psi, \psi)\big) \rangle = \int_{S^2} |D\psi|^2 + |(\nabla \psi)^{\perp}|^2 - \langle (\Delta u)^{\perp}, A_u (\psi, \psi) \rangle.
\]
Recalling that $(\nabla \psi)^{\perp} = A_u(\nabla u, \psi)$ and $(\Delta u)^{\perp} = A_u (\nabla u, \nabla u)$, we see from Gauss' equation that in fact
\begin{equation}\label{eq:intrinsic-1}
\int_{S^2} |\nabla \psi|^2 + \langle \nabla u, \nabla \big( A_{u}(\psi, \psi)\big) \rangle = \int_{S^2} |D\psi|^2 - R^{S^3}(\psi, \nabla u, \nabla u, \psi).
\end{equation}
Putting~\eqref{eq:intrinsic-1} and~\eqref{eq:critical-condition} back into~\eqref{eq:second-var-long-formula}, and then polarizing, we obtain the second variation formula of $E_{H, \ep}$. We summarize the result below.

\begin{prop}\label{prop:second-var-formula-intrinsic}
Let $u \in W^{2, 2}(S^2; S^3)$ be a critical point of $E_{H, \ep}$ and let $\psi, \xi \in \cT_u$. Then 
\begin{align}
\delta^2 E_{H, \ep}(u)(\psi, \xi)=&\ \ep^2\int_{S^2} \Delta \psi \cdot \Delta \xi + \Delta u \cdot \Delta \big( A_{u}(\psi, \xi)\big)\nonumber\\
&\ + \int_{S^2} \langle D\psi, D\xi \rangle- R^{S^3}(\psi, \nabla u, \nabla u, \xi) \nonumber\\
&\ + H \int_{S^2}  \psi \cdot \big(  Q(D_{x^1}\xi, u_{x^2}) + Q(u_{x^1}, D_{x^2}\xi) \big) dx^1 \wedge dx^2. \label{eq:2nd-var-intrinsic}
\end{align}
\end{prop}
\begin{rmk}
An integration by parts (using the fact that $Q$ is parallel) shows that the second line is symmetric with respect to $\psi, \xi$. It's also not hard to see that indeed $\delta^2 E_{H, \ep}(u)$ is a bounded bilinear form on $\cT_u$. 
For use in Sections~\ref{S:convergence} and~\ref{sec:improved}, we denote by $\delta^2 E_H(u)$ the bilinear form obtained by setting $\ep = 0$ in~\eqref{eq:2nd-var-intrinsic}.
\end{rmk}

\subsection{Regularity of critical points and the Palais-Smale condition}\label{SS:smoothness}
Two of the advantages of the perturbed functional $E_{H, \ep}$ are, first of all, that weak solutions in $W^{2, 2}(S^2; S^3)$ to $\delta E_{H, \ep}(u) = 0$ are smooth, and, secondly, that it satisfies a version of the Palais-Smale condition. These are standard facts, but the proofs are rather straightforward, so we include them below. We first address the regularity of critical points.

\begin{prop}\label{prop:weak-solution-smooth}
Let $u\in W^{2, 2}(S^2; S^3)$ be a critical point of $E_{H, \ep}$. Then $u$ is smooth. Moreover, for all $k$, the $W^{k, 2}$-norm of $u$ is bounded in terms of $k, \ep, H$ and $\|u\|_{2, 2}$. 
\end{prop}
\begin{proof}
To begin, note that since $\delta E_{H, \ep}(u) = 0$, we have $\delta E_{H, \ep}(u)(P_u(\psi)) = 0$ for all $\psi \in W^{2, 2}(S^2; \RR^N)$. This implies, by Proposition~\ref{prop:first-var-properties}, that for all $\psi \in W^{2, 2}(S^2; \RR^N)$ there holds
\begin{align}
&\ep^2 \int_{S^2} \Delta u \cdot \Delta \psi - A(u)(\nabla u, \nabla u) \cdot \Delta \psi\nonumber\\
& +\ep^2\int_{S^2}2 \langle \nabla (P\circ u)(\Delta u), \nabla \psi\rangle + \Delta (P\circ u)(\Delta u) \cdot \psi\nonumber\\
&+ \int_{S^2}\langle \nabla u, \nabla \psi \rangle + A(u)(\nabla u, \nabla u)\cdot \psi 
+H \int_{S^2} \psi \cdot \ast (u^{\ast}Q)  = 0.\label{eq:Euler-Lagrange-weak}
\end{align}

To continue, we write $v = \Delta u$. The crucial step is showing that $v \in W^{1, p}$ for all $p < 2$, which would kick start a bootstrapping argument. To that end, we define
\begin{align*}
f_1 &= \Delta u - A(u)(\nabla u, \nabla u) - H \ast(u^{\ast}Q),\\ 
f_2 &=- \Delta (P \circ u) (\Delta u),\\
F &=  \nabla \Big( A(u)(\nabla u, \nabla u)\Big) +2  \nabla (P \circ u)(\Delta u) .
\end{align*}
Note that $f_1 \in L^2$, $F \in L^{p}$ for all $p < 2$, while $f_2 \in L^1$, with their respective $L^p$-norms bounded in terms of the quantities listed in the statement of the proposition.

Now, since $v \in L^2$ is a distributional solution to 
\begin{equation}\label{eq:v-PDE}
\ep^2\Delta v  = f_1 + \ep^2 f_2 + \ep^2\Div F,
\end{equation}
we have by standard elliptic theory that $v \in W^{1, p}$ for all $p < 2$. (The term $f_2$ lies in $L^1$ only, but that still suffices for an $L^{{2, \infty}}$-estimate on $\nabla v$. See for example~\cite[Theorem 3.3.6]{Hel}.) Sobolev embedding then implies that $v \in L^{q}$ for all $q < \infty$, and consequently elliptic regularity yields $u \in W^{2, q}$ for all $q < \infty$. Next, suppose by induction that $u \in \cap_{q < \infty} W^{l, q}$. Then $f_1, f_2$ and $F$ all lie in $W^{l - 2, q}$ for all $q < \infty$. Consequently $v \in W^{l-1, q}$ by equation~\eqref{eq:v-PDE}, and hence $u \in W^{l + 1, q}$ for all $q < \infty$. This proves by induction that $u$ lies in $W^{k, q}$ for all $k \geq 1$ and $q < \infty$, and hence is smooth. Moreover, since every improvement of regularity above is accompanied by estimates, we also get the second conclusion of the Proposition.
\end{proof}

For later use, we record two standard results on solutions to \eqref{equ:CMC equation1}. Note that these satisfy~\eqref{eq:Euler-Lagrange-weak} with $\ep = 0$.
\begin{prop}\label{prop:conformal} 
Let $u: S^2 \to S^3$ be a smooth solution to~\eqref{equ:CMC equation1}. Then $u$ is weakly conformal.
\end{prop}
\begin{proof}
The proof uses the Hopf differential and is exactly the same as in the case $H = 0$. We need only notice that, in local coordinates, $\Delta u$ is orthogonal to both $u_x$ and $u_y$ thanks to the CMC equation. This implies that the Hopf differential $(u_{z} \cdot u_z) dz^2$ is holomorphic.
\end{proof}

\begin{prop}\label{prop:removable-singularity}
Let $u \in W^{1, 2}(S^2; S^3)$ be a weak solution to~\eqref{equ:CMC equation1} which is smooth away from finitely many points $p_1, \cdots, p_L$. Then in fact $u$ is smooth on all of $S^2$.
\end{prop}
\begin{proof} 
By conformal invariance we need only consider the case where the domain is a punctured disk $B\setminus \{0\} \subset \RR^2$. Then observe that we have all the ingredients necessary for the argument of~\cite[Theorem 3.6]{Sacks-Uhlenbeck81}, which yields the desired smoothness. These ingredients are the holomorphicity of the Hopf differential, an $\ep$-regularity theorem (which can be proven along the lines of~\cite[Theorem 2.2]{Sch}), and the fact that $|\Delta u| \leq C|\nabla u|^2$.
\end{proof}


Next we verify that $E_{H, \ep}$ satisfies a version of the Palais-Smale condition, subject to a bound on $D_{\ep}$. 
\begin{prop}\label{prop:PS}
Let $\{u_j\}$ be a sequence in $W^{2, 2}(S^2; S^3)$ satisfying
\begin{enumerate}
\item[(i)] $\|G_{H, \ep}(u_j)\| \to 0$ as $j\to \infty$.
\vskip 1mm
\item[(ii)] $D_{\ep}(u_j) \leq C$ for some $C>0$.
\end{enumerate}
Then, passing to a subsequence if necessary, $u_j$ converges strongly in $W^{2, 2}$ to a limit $u$ satisfying $\delta E_{H, \ep}(u) = 0$ and $D_{\ep}(u) \leq C$.
\end{prop}
\begin{rmk}
By Lemma~\ref{lemm:projection-estimates}(b), the assumption (i) is equivalent to the more standard assumption where the quantities $\|G_{H, \ep}(u_j)\|$ are replaced by the norms of $\delta E_{H, \ep}(u_j)$ as elements in the dual space of $\cT_{u_j}$.
\end{rmk}
\begin{proof}
This fact is already mentioned and used in~\cite{Lamm}. We indicate the main steps of the proof for the reader's convenience. First,  Lemma~\ref{lemm:D-equivalence}(a) and assumption (ii) imply that the sequence $u_{j}$ is bounded in $W^{2, 2}$. Hence, passing to a subsequence if necessary, we may assume that $u_{j}$ converges weakly in $W^{2, 2}$ to some limit $u \in W^{2, 2}(S^2; S^3)$. 
Next we want to upgrade the weak $W^{2, 2}$-convergence of $u_{j}$ to a strong $W^{2, 2}$-convergence. To do that we take $j, k$ large and consider
\[
R_{jk} = G_{H, \ep}(u_j)(u_j - u_k)- G_{H, \ep}(u_k)(u_j - u_k).
\]
With the help of~\eqref{eq:first-var-weak}, the weak convergence of the sequence $u_j$ and Sobolev embedding, we can show that 
\[
\Big| R_{jk} - \int_{S^2}\ep^2 |\Delta u_j - \Delta u_k|^2 + |\nabla u_j - \nabla u_k|^2 \Big| \to 0 \text{ as }j, k \to \infty.
\]
On the other hand, assumption (i) and the $W^{2, 2}$-boundedness of the sequence $u_j$ implies that $R_{jk} \to 0$ as $j, k \to \infty$, so we see that
\[
\int_{S^2}\ep^2 |\Delta u_j - \Delta u_k|^2 + |\nabla u_j - \nabla u_k|^2 \to 0 \text{ as }j, k \to \infty.
\]
Applying Lemma~\ref{lemm:D-equivalence}(a) to $u_j - u_k$, we infer from the above that $u_j$ converges strongly in $W^{2, 2}$. Therefore we may pass to limits in~\eqref{eq:first-var-weak} and use assumption (i) to get 
\[
G_{H,\ep}(u)(\psi) = \lim_{j \to \infty} G_{H, \ep}(u_j)(\psi) = 0, \text{ for all }\psi \in W^{2, 2}(S^2; \RR^{N}).
\] 
Hence $\delta E_{H, \ep}(u) = 0$ as asserted. Finally, the fact that $D_{\ep}(u)\leq C$ follows immediately from the strong $W^{2, 2}$-convergence of $u_j$ to $u$.
\end{proof}

\subsection{The Morse index and generalized Morse neighborhoods}\label{subsec:Morse}

In this section we establish some functional-analytic properties of the Hessian $\delta^2 E_{H, \ep}(u)$ at a critical point. These have two consequences. First of all, the Morse index of the critical point, to be defined shortly, is always finite. Secondly, we may invoke the generalized Morse lemma (see for instance~\cite[Theorem 8.3]{MW}) to find a suitable neighborhood of the critical point on which $E_{H, \ep}$ takes on a particularly simple form after a change of coordinates. Note that we do not require the critical point to be non-degenerate.

\begin{defi}\label{defi:index-non-degen}
Let $u$ be a critical point of $E_{H, \ep}$. The Morse index of $u$ as a critical point of $E_{H, \ep}$, denoted $\Ind_{H, \ep}(u)$, is defined to be the supremum of $\dim V$ over all finite-dimensional subspaces $V$ of $\cT_u$ on which the Hessian $\delta^2 E_{H, \ep}(u)$ restricts to be negative-definite.
\end{defi}
\begin{rmk}\label{rmk:index}
Let $u$ be a smooth solution to~\eqref{equ:CMC equation1}. Then formally it is a critical point of~\eqref{eq:main perturbed functional} with $\ep = 0$, and we define its Morse index, denoted $\Ind_H(u)$, in exactly the same way as above, with $\delta^2 E_{H, \ep}(u)$ replaced by $\delta^2 E_H(u)$. 
\end{rmk}

The next lemma shows that the index as defined in Definition~\ref{defi:index-non-degen} is always finite and establishes some properties of $\delta^2 E_{H, \ep}(u)$ to be used later.
\begin{lemm}\label{lemm:finite-index} 
In the notation of Definition~\ref{defi:index-non-degen}, suppose furthermore that $A: \cT_u \to \cT_u$ is the bounded linear operator associated with the bilinear form $\delta^2 E_{H, \ep}(u)$ via the inner product on $\cT_u$. Then
\vskip 1mm
\begin{enumerate}
\item[(a)] The operator $A: \cT_u \to \cT_u$ is a self-adjoint Fredholm operator. In particular, $\cT_u$ splits orthogonally into $\Ker A \oplus \Ran A$.
\vskip 1mm
\item[(b)] There exists a sequence of real numbers $\lambda_i \to \infty$ and a basis $\{\psi_i\}$ of $\cT_u$ such that 
\[
\delta^2 E_{H, \ep}(u)(\psi_i, \cdot) = \lambda_i (\psi_i, \cdot)_{L^2},
\]
and that $(\psi_i, \psi_j)_{L^2} =\delta_{ij}$. In particular $\Ind_{H, \ep}(u)$ is finite.
\end{enumerate}
\end{lemm}
\begin{proof}
First note that $A$ is self-adjoint because the bilinear form $\delta^2 E_{H, \ep}(u)$ is symmetric. Next we establish the following G$\mathring{a}$rding-type inequality:
\begin{equation}\label{eq:Garding}
(A\psi, \psi)_{\cT_u} = \delta^2 E_{H, \ep}(u)(\psi, \psi) \geq c_1\int_{S^2} |\nabla^2 \psi|^2 - c_2 \int_{S^2} |\psi|^2.
\end{equation}
(The dependences of $c_1, c_2$ on $\ep$ and $u$ are suppressed in the notation since the latter are fixed.) Note that since $u$ is smooth by Proposition~\ref{prop:weak-solution-smooth}, we can go to~\eqref{eq:second-var-long-formula} and carry out the differentiation in the terms $\Delta \big( A_{u}(\psi, \psi)\big)$ and $\nabla \big( A_{u}(\psi, \psi)\big)$. Recalling further that 
\[
D\psi = \nabla \psi + A_{u}(\nabla u, \psi), 
\]
and using Lemma~\ref{lemm:D-equivalence}(a) together with Young's inequality, it is not hard to see that there are constants $c_1, c_2$ independent of $\psi$ so that 
\[
\delta^2 E_{H, \ep}(u)(\psi, \psi) \geq c_1\int_{S^2} |\nabla^2 \psi|^2 - c_2 \int_{S^2} |\nabla \psi|^2 + |\psi|^2.
\]
Then, an integration by parts shows that $\int_{S^2} |\nabla \psi|^2 \leq \int_{S^2} |\Delta \psi | |\psi|$. Hence, adjusting the constants $c_1, c_2$ if necessary, another application of Young's inequality leads to~\eqref{eq:Garding}, which together with self-adjointness and the compact embedding 
\[
\cT_u \to \{\psi \in L^2(S^2; \RR^N)\ |\ \psi(x) \in T_{u(x)}S^3 \text{ for a.e. }x \in S^2\}
\]
yield all the remaining assertions of the lemma. 
\end{proof}

Next we continue to assume that $u$ is a critical point of $E_{H, \ep}$. Let $\cA$ be a simply-connected neighborhood of $u$, and consider a local reduction $E^{\cA}_{H, \ep}$ induced by some extension $f \in \cE(u)$. Since $E^{\cA}_{H, \ep}$ is a $C^2$-functional and since $\delta^2 E^{\cA}_{H, \ep} = \delta^2 E_{H, \ep}$ on $\cA$, by Lemma~\ref{lemm:finite-index}(a) we may apply the generalized Morse Lemma on Hilbert-Riemannian manifolds (see for example~\cite[Theorem 8.3]{MW}) to find a homeomorphism $\Psi_1$ from a neighborhood of zero in $\cT_u$ onto a neighborhood of $u$ in $\cA$, with $\Psi_1(0) = u$, such that
\[
E^{\cA}_{H, \ep}(\Psi_1(\xi)) = e_1(\xi_0)  + \delta^2 E_{H, \ep}(u)(\xi_{\perp}, \xi_{\perp}) \text{ for }\xi \text{ close to $0$ in $\cT_u$,}
\]
where we've written $\xi = \xi_0 + \xi_{\perp}$ according to the decomposition $\cT_u = \Ker A \oplus \Ran A$, and $e_1$ is a $C^2$ function whose first and second derivatives vanish at the origin. This together with the diagonalization provided by Lemma~\ref{lemm:finite-index}(b) imply that, passing to a smaller neighborhood in $\cA$ if necessary, we obtain what we refer to as a generalized Morse neighborhood of $u$, defined below.

\begin{defiprop}[Generalized Morse neighborhood]
\label{defiprop:Morse-neighborhood}
Each critical point $u$ of $E_{H, \ep}$ possesses a {\em generalized Morse neighborhood},  which is a neighborhood $\cA \subset W^{2, 2}(S^2; S^3)$ of $u$ with the following properties:
\begin{enumerate}
\item[(i)] $\cA$ is simply-connected, so that any choice of $f \in \cE(u)$ induces a local reduction $E^{\cA}_{H, \ep}$.
\vskip 1mm
\item[(ii)] There exists a ball $\cB$ centered at zero in $\cT_{u}$, a homeomorphism $\Psi: \cB \to \cA$ with $\Psi(0) = u$, and a direct sum decomposition 
\begin{equation}\label{eq:Morse-decomp}
\cT_{u} = H_0 \oplus H_- \oplus H_+,
\end{equation}
of $\cT_u$ into closed subspaces, with $\dim H_0 = \dim \Ker A$ and $\dim H_{-} = \Ind_{H, \ep}(u)$, so that for any $f \in \cE(u)$ there holds
\begin{equation}\label{eq:generalized-Morse}
\widetilde{E}(\psi) = e(\psi_0) + \|\psi_{+}\|_{\cT_u}^2 - \|\psi_{-}\|_{\cT_u}^2 \text{ for all }\psi \in \cB,
\end{equation}
where we've written $\psi = \psi_0 + \psi_{-} + \psi_+$ with respect to the decomposition~\eqref{eq:Morse-decomp}, and $\widetilde{E} = E^{\cA}_{H, \ep} \circ \Psi$. Also, the function $e$ is again $C^2$ and its first and second derivatives at the origin vanish. (Note that the difference $\widetilde{E}(\psi) - \widetilde{E}(0)$ is independent of the choice of extension $f \in \cE(u)$, so if the above holds for one extension it holds for all.)
\end{enumerate}
\end{defiprop}

\section{Existence of non-trivial critical points of the perturbed functional}\label{S:non-trivial-critical-points}

In this section, we show how to find non-constant critical points with bounded Morse index for the perturbed functionals with $D_{\ep}$ bounded independently of $\ep$.

\subsection{Admissible sweepouts and the min-max value}\label{SS:sweepouts}
We introduce the admissible sweepouts and explain how to define the min-max value. For each continuous path $\gamma:[0, 1] \to W^{2, 2}(S^2; S^3)$ with $\gamma(0), \gamma(1)$ being constant maps, we consider the induced map $h_{\gamma}: S^{3} \to S^3$. The continuous embedding of $W^{2, 2}(S^2; S^3)$ into $C^{0}(S^2; S^3)$ guarantees that $h_\gamma$ is continuous, and hence has a degree, denoted $\deg(h_\gamma)$. The space of admissible sweepouts is then given by (compare with~\cite[p.28]{Struwe88}) 
\[
\cP = \{\gamma \in C^{0}([0, 1]; W^{2, 2}(S^2; S^3))\ |\ \gamma(0) = \text{constant}, \gamma(1) = \text{constant}, \deg(h_\gamma) = 1\}.
\]
For each $\gamma \in \cP$ and $t \in [0, 1]$, the map $f_{\gamma, t}: [0, 1] \times S^2 \to S^3$ given by
\[
f_{\gamma, t}(s, x) = \gamma(ts)(x),
\]
belongs to $\cE(\gamma(t))$, and we associate a min-max value to each $E_{H, \ep}$ by letting 
\begin{equation}\label{eq:min-max value}
\omega_{H, \ep} = \inf_{\gamma \in \cP}\sup_{t \in [0, 1]} E_{H, \ep}(\gamma(t), f_{\gamma, t}).
\end{equation}
Moreover, for later use we introduce the following collections. Given $H, \ep$, for $\alpha, C > 0$, we define 
\begin{align*}
\cP_{\alpha, C} =\ &\text{ the collection of sweepouts in }\cP \text{ satisfying }\\
&\ \text{(i) } \max\{ E_{H, \ep}(\gamma(t), f_{\gamma, t}): t \in [0, 1]\} \leq \omega_{H, \ep} + \alpha \text{, and }\\
&\ \text{(ii) } D_{\ep}(\gamma(t)) \leq C  \text{ whenever }E_{H, \ep}(\gamma(t), f_{\gamma, t}) \geq \omega_{H, \ep} - \alpha.
\end{align*}
Next, for $C > 0$ we define the subset $\cK_{C} \subset W^{2, 2}(S^2; S^3)$ by
\[
\begin{aligned}
\cK_{C} = \{ u \in W^{2, 2}(S^2; S^3)\ | & \ \delta E_{H, \ep}(u) = 0, D_{\ep}(u) \leq C \text { and }\\
&\ E_{H, \ep}(u, f) = \omega_{H, \ep} \text{ for some }f \in \cE(u)\}.
\end{aligned}
\]
By Proposition~\ref{prop:PS} and a direct computation using concatenations and Lemma~\ref{lemm:volume-properties}(a), we see that $\cK_{C}$ is a compact subset of $W^{2, 2}(S^2; S^3)$. Also, both $\cP_{\alpha, C}$ and $\cK_{C}$ of course depend on $H, \ep$, but we have suppressed that dependence in the notation as the choice should always be clear from the context.

The basic properties of the min-max values are established in the Lemma below.
\begin{lemm}\label{lemm:min-max-finite} In the above notation,
\vskip 1mm
\begin{enumerate}
\item[(a)] For any $H \in \RR$ and $\ep > 0$, and any $\gamma \in \cP$, the function $t \mapsto E_{H, \ep}(\gamma(t), f_{\gamma, t})$ is continuous.
\vskip 1mm
\item[(b)] $0 \leq \omega_{H, \ep} < \infty$ for all $H \in \RR$ and $\ep > 0$.
\vskip 1mm
\item[(c)] $\omega_{H, \ep}$ is a measurable function in $(H, \ep)$.
\end{enumerate}
\end{lemm}

\begin{proof}
For part (a), fix $t_0 \in [0, 1]$ and choose a simply-connected neighborhood $\cA$ of $\gamma(t_0)$. Recalling the definition of the local reduction $E^{\cA}_{H, \ep}$ induced by $(\gamma(t_0), f_{\gamma, t_0})$, it's not hard to see that 
\[
E_{H, \ep}(\gamma(t), f_{\gamma, t}) = E^{\cA}_{H, \ep}(\gamma(t)),
\]
for $t$ sufficiently close to $t_0$. The asserted continuity then follows immediately from Proposition~\ref{prop:local-reduct}(a).

For part (b), we view $S^3$ as the set $\{ x_1^2 + x_2^2 + x_3^2 + x_4^2 = 1\}$ and consider the path $\gamma$ given by parametrizing $S^3 \cap \{x_4 = 2t - 1\}$ for $t \in [0, 1]$. It is clear that, with the parametrizations suitably chosen, we have $\gamma \in \cP$, so the latter is non-empty. Thanks to part (a), this implies that $\omega_{H, \ep}$ is finite. On the other hand, the first inequality in (b) is a trivial consequence of the fact that $E_{H, \ep}(\gamma(0), f_{\gamma, 0}) = 0$ for all $\gamma \in \cP$. Finally, part (c) follows from the fact that for each $\gamma \in \cP$, the function $(H, \ep) \mapsto \sup\limits_{t \in [0, 1]}E_{H, \ep}(\gamma(t), f_{\gamma, t})$ is continuous.
\end{proof}

\subsection{Derivative estimate and existence of nice sweepouts}\label{SS:struwe}
In this section, we modify the monotonicity trick of Struwe~\cite{Struwe88} to produce, for almost every $H > 0$, sweepouts enjoying nice estimates which will help us obtain critical points of $E_{H, \ep}$ for a sequence $\ep$ going to zero in the two sections to follow. The basic idea (see~\cite[Equations (4.2), (4.8), Lemma 4.1]{Struwe88}) is that the monotonicity properties of the min-max value along with standard real analysis imply derivative estimates which can be exploited to give energy bounds.

In our case we need the derivative estimates to be uniform with respect to the parameter $\ep$. This is the content of Proposition~\ref{prop:min-max-value}(c) below.
\begin{prop}\label{prop:min-max-value}
The following holds for $\omega_{H, \ep}$ as a function of $H$ and $\epsilon$.
\begin{enumerate}
\item[(a)] For each $\ep \geq 0$, the function $H \mapsto \omega_{H, \ep}/H$ is non-increasing.
\vskip 1mm
\item[(b)] For each $H > 0$, the function $\ep \mapsto \omega_{H, \ep}$ is non-decreasing.
\vskip 1mm
\item[(c)] Given a sequence $\ep_j \to 0$, for almost every $H \in \RR_+$ there exist a subsequence, which we do not relabel, and some $c > 0$, such that
\[
0 \leq \frac{d}{dH}\Big( -\frac{\omega_{H, \ep_j}}{H} \Big) \leq c,\quad \text{ for all }j \in \NN.
\]
\end{enumerate}
\end{prop}

\begin{proof}
For part (a), given $H>H'>0$, for any $u \in W^{2, 2}(S^2, S^3)$ with $f\in \cE(u)$, we have by direct computation 
that (see also~\cite[Equation (4.2)]{Struwe88})
\begin{equation}\label{eq:differential of normalized energy}
\frac{E_{H, \ep}(u, f)}{H} - \frac{E_{H', \ep}(u, f)}{H'} = -\frac{H - H'}{H\cdot H'} D_{\ep}(u) \leq 0.
\end{equation}
Next, for any $\delta>0$, there exists $\gamma \in \cP$ such that
\[
\max_{t\in [0, 1]} E_{H', \ep}(\gamma(t), f_{\gamma, t}) \leq \omega_{H', \ep} + \delta.
\]
Therefore,
\begin{equation}\label{eq:supremum-monotonicity}
\frac{\omega_{H, \ep}}{H} \leq \max_{t\in[0, 1]} \frac{E_{H, \ep}(\gamma(t), f_{\gamma, t})}{H} \leq  \max_{t\in [0, 1]} \frac{E_{H', \ep}(\gamma(t), f_{\gamma, t})}{H'} \leq \frac{\omega_{H', \ep}}{H'} + \frac{\delta}{H'},
\end{equation}
where we used~\eqref{eq:differential of normalized energy} to get the second inequality. The monotonicity of $H \mapsto \omega_{H, \ep}/H$ then follows by the arbitrariness of $\delta$.

\vskip 1mm
For part (b), note that for any given $u\in W^{2, 2}(S^2, S^3)$ and $f \in \cE(u)$, the map $\ep \mapsto D_{\ep}(u)$ is non-decreasing. The conclusion then follows the same way as part (a). 

\vskip 1mm
For part (c), by part (a) and Lemma~\ref{lemm:min-max-finite}(b), we know that for each $j$, the derivative $\frac{d}{dH}\Big(-\frac{\omega_{H, \ep_j}}{H} \Big)$ exists almost everywhere for $H\in (0, +\infty)$ and is non-negative. Furthermore, given any $0<a<b<+\infty$, we have
\[
\frac{\omega_{a, \ep_j}}{a} - \frac{\omega_{b, \ep_j}}{b} \geq \int_a^b \frac{d}{dH}\Big(-\frac{\omega_{H, \ep_j}}{H} \Big) dH.
\]
(See for instance~\cite[Theorem (7.21)]{Zyg}.) Then by Fatou's Lemma, 
\begin{align}\label{eq:Fatou}
\int_a^b \liminf_{j \to \infty}\frac{d}{dH}\Big(-\frac{\omega_{H, \ep_j}}{H} \Big) dH &\leq \liminf_{j \to \infty} (\frac{\omega_{a, \ep_j}}{a} - \frac{\omega_{b, \ep_j}}{b}) \leq \frac{\omega_{a, 1}}{a} - \frac{\omega_{b, 0}}{b} < \infty,
\end{align}
where we used part (b) for the second inequality. Therefore, for almost every $H\in [a, b]$, 
\[
\liminf_{j \to \infty}\frac{d}{dH}\Big(-\frac{\omega_{H, \ep_j}}{H} \Big) < +\infty.
\]
The conclusion follows by the arbitrariness of $a$ and $b$.
\end{proof}

The next result shows how the derivative bound in Proposition~\ref{prop:min-max-value}(c) translates into the existence of minimizing sequences with $D_\ep$ uniformly bounded for almost min-max slices.

\begin{lemm}\label{lem:existence of good sweepouts}
Let $H > 0, \ep \in (0, 1)$ and suppose for some constant $c$ we have
\[
0 \leq \frac{d}{dH}\Big( -\frac{\omega_{H, \ep}}{H} \Big) \leq c.
\]
Then $\cP_{H/k, 8H^2c}$ is non-empty for all $k$ sufficiently large.
\end{lemm}

\begin{proof}
Choose $H_k = H - \delta_k$, where
\[
\delta_k = 1/(4ck). 
\]
By assumption, there exists $k_0 \in \NN$ so that for $k\geq k_0$, we have 
\begin{equation}\label{eq:diff-quotient}
\frac{1}{H-H_k} \Big( \frac{\omega_{H_k, \ep}}{H_k} - \frac{\omega_{H, \ep}}{H} \Big) \leq 2 c.
\end{equation}
In particular, 
\begin{equation}\label{eq:H-Hk-difference}
\frac{\omega_{H_k, \ep}}{H_{k}} \leq \frac{\omega_{H, \ep}}{H} + 2c\delta_k = \frac{\omega_{H, \ep}}{H} + \frac{1}{2k}, \text{ for all }k \geq k_0.
\end{equation}
Next, for each such $k \geq k_0$, there exists $\gamma_k\in \cP$, such that 
\begin{equation}\label{eq:lower energy bound}
\frac{1}{H_k} \max_{t\in [0, 1]} E_{H_k, \ep}(\gamma_k(t), f_{\gamma_k, t}) \leq \frac{1}{H_k} \omega_{H_k, \ep} + \frac{1}{2k}.
\end{equation}
Combining this with~\eqref{eq:H-Hk-difference} and the middle inequality of~\eqref{eq:supremum-monotonicity}, we get
\begin{align}\label{eq:minimizing-sequence}
\max_{t\in [0, 1]} E_{H, \ep}(\gamma_k(t), f_{\gamma_k, t})  & \leq H  (\frac{\omega_{H_k, \ep}}{H_k}) + \frac{H}{2k}\nonumber  \\
& \leq \omega_{H, \ep}  + \frac{H}{2k} + \frac{H}{2k} = \omega_{H, \ep} + \frac{H}{k}.
\end{align}

On the other hand, suppose $t\in [0, 1]$ satisfies that
\begin{equation}\label{eq:upper energy bound}
\frac{1}{H} E_{H, \ep}(\gamma_k(t), f_{\gamma_k, t}) \geq \frac{1}{H} \omega_{H, \ep} - \frac{1}{k}.
\end{equation}
Then, for $k \geq k_0$, by subtracting \eqref{eq:upper energy bound} from \eqref{eq:lower energy bound} and using \eqref{eq:differential of normalized energy}, we have 
\[
\begin{aligned}
\frac{D_\ep(\gamma_k(t))}{H\cdot H_k} 
& = \frac{1}{H-H_k}\Big( \frac{E_{H_k, \ep}(\gamma_{k}(t), f_{\gamma_k, t})}{H_k} - \frac{E_{H, \ep}(\gamma_{k}(t), f_{\gamma_k, t})}{H} \Big)\\
& \leq \frac{1}{H-H_k} \Big( \frac{\omega_{H_k, \ep}}{H_k} - \frac{\omega_{H, \ep}}{H} + \frac{3}{2k} \Big)\\
& \leq \frac{1}{H-H_k} \Big( \frac{\omega_{H_k, \ep}}{H_k} - \frac{\omega_{H, \ep}}{H} \Big) + 6c\\
& \leq 8c,
\end{aligned}
\]
where we used~\eqref{eq:diff-quotient} to get the last line. Therefore 
\begin{equation}\label{eq:D-estimate}
D_\ep(\gamma_k(t)) \leq 8 H H_k c \leq 8H^2 c.
\end{equation}
By~\eqref{eq:minimizing-sequence} and~\eqref{eq:D-estimate}, we conclude that $\gamma_{k} \in \cP_{H/k, 8H^2c}$.
\end{proof}

\subsection{Non-trivial critical points of the perturbed functional with uniform energy bound}\label{SS:existence-non-trivial-critical}

The main result of this section is Proposition~\ref{prop:critical-uniform-bound}, which yields the objects described in the title when combined with Proposition~\ref{prop:min-max-value} and Lemma~\ref{lem:existence of good sweepouts}. We will however not bring these two latter results in until we've established one more ingredient in the next section.  The proof of Proposition~\ref{prop:critical-uniform-bound} is interrupted by Lemmas~\ref{lem:pseudo-gradient vector field} and~\ref{lemm:existence-time}, which concern the pseudo-gradient vector field and the existence time of its associated flow. We start with a lemma which will be used to show the non-triviality of min-max critical points.

\begin{lemm}[Compare~\cite{Struwe88}, Lemma 2.4]
\label{lemm:small-D-not-maximum}
Given $H_0>0$, $\ep \in (0, 1]$, there exist positive $\eta_1(\ep, H_0), \eta_2(\ep, H_0)>0$ such that if $\gamma \in \cP$ and the slice at some $t_0$ satisfies
\begin{equation}\label{eq:small-D}
D_{\ep}(\gamma(t_0)) < \eta_1(\ep, H_0), 
\end{equation}
then for $H \in (0, H_0]$ we have
\begin{equation}
E_{H, \ep}(\gamma(t_0), f_{\gamma, t_0}) < \max_{0 \leq t \leq 1} E_{H, \ep}(\gamma(t), f_{\gamma, t}) - \eta_2(\ep, H_0).
\end{equation}
\end{lemm}

\begin{proof}
We first note that there exists $\alpha_0 > 0$ such that if $\gamma \in \cP$ then 
\[
\max_{0 \leq t \leq 1} D_{\ep}(\gamma(t)) \geq \alpha_0,
\]
for otherwise we may use Lemma~\ref{lemm:D-equivalence}(b) and the Poincar\'e inequality to show that the induced map $f_{\gamma, 1} = h_\gamma$ is null-homotopic, contradicting the definition of $\cP$. Hence, if 
\[
\eta_1 < \alpha < \alpha_0,
\]
then there must exist some $t' \neq t_0$ such that $D_{\ep}(\gamma(t'))\geq \alpha$. Without loss of generality we assume $t' > t_0$ and define $t_1 = \inf\{t \geq t_0\ |\ D_{\ep}(\gamma(t)) \geq \alpha\}$. Then 
\begin{equation}\label{eq:t1-not-so-small}
D_{\ep}(\gamma(t_1))  = \alpha,
\end{equation} 
\begin{equation}\label{eq:small-D-path}
D_{\ep}(\gamma(t)) \leq \alpha \text{ for all }t \in [t_0, t_1].
\end{equation}
To continue, note that by Lemma~\ref{lemm:D-equivalence}(b) and the Poincar\'e inequality applied to $\gamma(t_0), \gamma(t_1)$, we get constants $\widetilde{c}_0, \widetilde{c}_1 \in \RR^N$ such that 
\[
\|\gamma(t_i) - \widetilde{c}_i\|_{\infty} < K_1\sqrt{\alpha},\ i = 0, 1,
\] 
where $K_1$ depends on $\ep$ and on $A_0$ from Lemma~\ref{lemm:D-equivalence}. Since the $\gamma(t_i)$ map into $S^3$, we further have  $\dist(\widetilde{c_i}, S^3) < K_1\sqrt{\alpha}$, and hence, provided $K_1\sqrt{\alpha} < \delta_0/2$, where $\delta_0$ is from Lemma~\ref{lemm:volume-properties}(c), we may define $c_i = \Pi(\widetilde{c}_i)$ and find that
\[
\|\gamma(t_i) - c_i\|_{\infty} < 2K_1\sqrt{\alpha} < \delta_0,\ \text{ for } i = 0, 1.
\]

We next want to estimate the volume enclosed between $\gamma(t_0)$ and $\gamma(t_1)$. Note that, since $\|\gamma(t_i) - c_i\|_{\infty} < \delta_0$ we may define $h_1, h_2 : [0, 1] \times S^2 \to S^3$ by 
\[
h_i(s, x) = \Pi(s\gamma(t_i)(x) + (1 - s)c_i),
\]
and observe that by~\eqref{eq:small-D-path} along with Lemma~\ref{lemm:D-equivalence} and Lemma~\ref{lemm:projection-estimates}(a), we have
\begin{equation}\label{eq:hi-D-bound}
D_{\ep}(h_i(s, \cdot)) < K_2\alpha, \text{ for all }s \in [0, 1] \text{ and }i = 1, 2.
\end{equation}
Moreover, if $L$ denotes an upper bound for $|d\Pi|$ on the tubular neighborhood $\cV$ of $S^3$, then a direct computation shows that  
\begin{align}\label{eq:hi-V-bound}
|V(h_i)| \leq L^{3} \|\gamma(t_i) - c_i\|_{\infty} \int_{S^2} |\nabla \gamma(t_i)|^2 &\leq 4L^3K_1\sqrt{\alpha} D_{\ep}(\gamma(t_i)).
\end{align}
To continue, we let $f_{\gamma, t_0, t_1}:[0, 1] \times S^2 \to S^3$ be defined by 
\[
f_{\gamma, t_0, t_1}(s, x) = \gamma\big( s t_1 + (1 - s)t_0, x \big),
\]
and let $f: [0, 1] \times S^2 \to S^3$ be the concatenation obtained by following $h_0, f_{\gamma, t_0, t_1}$ and then the reverse of $h_1$. By~\eqref{eq:small-D-path} and~\eqref{eq:hi-D-bound}, we see that if $(1 + K_2)\alpha < \alpha_0$ then
\[
D_{\ep}(f(s, \cdot)) < \alpha_0 \text{ for all }s \in [0, 1],
\]
and hence $f$ induces a degree-zero map from $S^3$ to $S^3$. This in turn implies that 
\[
V(f_{\gamma, t_1}) - V(f_{\gamma, t_0}) = V(h_1) - V(h_0).
\]
Consequently we may estimate
\begin{align*}
\big| V(f_{\gamma, t_1}) - V(f_{\gamma, t_0})\big| &\leq \big| V(h_0) \big| + \big| V(h_1) \big|\\
&\leq CL^3K_1\sqrt{\alpha}\big( D_{\ep}(\gamma(t_0)) + D_{\ep}(\gamma(t_1)) \big).
\end{align*}

We are now ready to finish the proof. Indeed, recalling the definition of $E_{H, \ep}$ and the assumption $H \leq H_0$, by the triangle inequality we have
\begin{align*}
E_{H, \ep}(\gamma(t_1), f_{\gamma, t_1}) - E_{H, \ep}(\gamma(t_0), f_{\gamma, t_0}) \geq &\ D_{\ep}(\gamma(t_1)) - D_{\ep}(\gamma(t_0)) - H_0 \big| V(f_{\gamma, t_1}) - V(f_{\gamma, t_0})\big|.
\end{align*}
Using the previous inequality to estimate $\big| V(f_{\gamma, t_1}) - V(f_{\gamma, t_0})\big|$ and then rearranging, we get
\begin{align*}
&\ E_{H, \ep}(\gamma(t_1), f_{\gamma, t_1}) - E_{H, \ep}(\gamma(t_0), f_{\gamma, t_0}) \geq\\
 \geq&\  (1 - CH_0L^3 K_1\sqrt{\alpha}) D_{\ep}(\gamma(t_1)) - (1 + CH_0L^3K_1 \sqrt{\alpha})D_{\ep}(\gamma(t_0)) \\
 >&\ (1 - CH_0L^3 K_1\sqrt{\alpha}) \alpha - (1 + CH_0L^3K_1 \sqrt{\alpha})\eta_1,
\end{align*}
where the last line follows from~\eqref{eq:t1-not-so-small} and our assumption on $\gamma(t_0)$. Upon requiring, in addition to the above thresholds on $\alpha$, that $CH_0L^3K_1\sqrt{\alpha} < 1/2$ and then choosing $\eta_1$ such that $\eta_1 < \alpha$ and $(1 + CH_0L^3K_1\sqrt{\alpha})\eta_1 < \alpha/4$, we conclude the proof with $\eta_2 = \alpha/4$. 
\end{proof}

We are now ready to extract a strongly convergent min-max sequence with non-trivial limit out of a sequence of good sweepouts using pseudo-gradient vector fields.

\begin{prop}\label{prop:critical-uniform-bound}
Given $H > 0, \ep \in (0, 1]$, suppose for some $C_0 > 0$ there exist sweepouts $\gamma_k \in \cP_{\alpha_k, C_0}$ for all $k$, where $\alpha_k \to 0$. Then, passing to a subsequence if necessary, there exists $t_k \in [0, 1]$ so that the following hold.
\vskip 2mm
\begin{enumerate}
\item[(a)] $|E_{H, \ep}(\gamma_k(t_k), f_{\gamma_k, t_k}) - \omega_{H, \ep}| \leq \alpha_{k}$. In particular, $E_{H, \ep}(\gamma_k(t_k), f_{\gamma_k, t_k}) \to \omega_{H, \ep}$ as $k \to \infty$.
\vskip 1mm
\item[(b)]  $\gamma_k(t_k)$ converges strongly in $W^{2, 2}(S^2; S^3)$ to some $u$ lying in $\cK_{C_0}$. 
\vskip 1mm
\item[(c)] The limiting map $u$ in part (b) is non-constant. In fact, $D_{\ep}(u) \geq \eta_1(\ep, H)$, where the latter is given by Lemma~\ref{lemm:small-D-not-maximum}.
 \end{enumerate}
\end{prop}
\begin{rmk}\label{rmk:lower-bound-promise}
Conclusion (c) only gives a lower bound in terms of $\ep$. A uniform lower bound will be derived later in Proposition \ref{prop:lower-bound}.
\end{rmk}

\begin{proof}
To begin we define  
\begin{equation}\label{eq:J_k}
J_k = \{t \in [0, 1]\ |\ E_{H, \ep}(\gamma_{k}(t), f_{\gamma_{k}, t}) > \omega_{H, \ep} - \alpha_k\}.
\end{equation}
In view of Proposition~\ref{prop:PS}, we first want to establish the following statement: 
\vskip 2mm
\noindent\textit{For all $\delta > 0$, there exists $k_0 \in \NN$ such that 
\begin{equation}\label{eq:property-star}
\inf \big\{ \| G_{H, \ep}(\gamma_k(t)) \| \ |\  t \in J_k\big\} < \delta \text{, for all }k \geq k_0. \tag{$\ast$}
\end{equation}}

We will prove~\eqref{eq:property-star} by contradiction. That is, suppose we can find some $\delta>0$ and a subsequence of $\gamma_k$, which we do not relabel, such that for all $k$ there holds
\[
\| G_{H, \ep}(\gamma_k(t)) \| \geq \delta, \text{ for all }t \in J_k.
\]

\begin{lemm}
\label{lem:pseudo-gradient vector field}
 There exists a locally Lipschitz continuous map $X: \tilde{V} \to W^{2, 2}(S^2; \RR^N)$,
where
\[
\tilde{V} =  \{ u \in W^{2, 2}(S^2; S^3)\ |\ \delta E_{H, \ep}(u) \neq 0 \} ,
\]
such that 
\begin{enumerate}
\item $X(u) \in \cT_{u}$, for every $u\in \tilde V$.
\vskip 1mm
\item $\|X(u)\|_{2, 2} < 2 \min\{ \|G_{H, \ep}(u)\|, 1 \}$.
\vskip 1mm
\item \[\langle G_{H, \ep}(u), X(u) \rangle < -\frac{\min\{ \| G_{H, \ep}(u)\|,  1\} \|G_{H, \ep}(u)\|}{A_1 \big( 1 + \|u\|_{2, 2}^2 \big)},\] where the constant $A_1$ is from Lemma~\ref{lemm:projection-estimates}.
\end{enumerate}
\end{lemm}
\begin{proof}[Sketch of proof]
The argument is standard, and we only outline the key steps. First, thanks to the continuity property of $G_{H, \ep}$ established in Proposition~\ref{prop:first-var-properties}(c), we may follow the construction in~\cite[Chapter II.3]{StrBook} to obtain a locally Lipschitz map 
\[
\widetilde{X}: \widetilde{V} \to W^{2, 2}(S^2; \RR^N)
\]
such that for all $u\in \widetilde{V}$ there holds
\begin{align}
\|\widetilde{X}(u)\|_{2, 2}& < \frac{2 \min\{ \|G_{H, \ep}(u)\|, 1 \}}{A_1 \big( 1 + \|u\|_{2, 2}^2 \big)} \label{eq:pre-projection-norm}\\
G_{H, \ep}(u)(\widetilde{X}(u)) &< -\frac{\min\{ \|G_{H, \ep}(u)\|,  1\} \|G_{H, \ep}(u)\|}{A_1 \big( 1 + \|u\|_{2, 2}^2 \big)}. \label{eq:pre-projection-first-var}
\end{align}
Now we let $X(u) = P_u (\widetilde{X}(u))$. Then by Lemma~\ref{lemm:projection-estimates}(b) and what we have just arranged, we see that $X$ is a locally Lipschitz map into $W^{2, 2}(S^2; \RR^N)$.  
To finish, note that property (1) in the conclusion is immediate, while Lemma~\ref{lemm:projection-estimates}(b) together with~\eqref{eq:pre-projection-norm} imply property (2). To see property (3), note from the definition of $G_{H, \ep}(u)$ that $G_{H, \ep}(u)(X(u)) = G_{H, \ep}(u)(\widetilde{X}(u))$.
\end{proof}

\vskip 2mm
Returning to the proof of Proposition \ref{prop:critical-uniform-bound}, we consider the flow associated with $X$, denoted 
\[
\Phi: \{(u, \tau)\ |\ u \in \widetilde{V},\ 0 \leq \tau < T(u)\} \to W^{2, 2}(S^2; S^3) \ \subset W^{2, 2}(S^2; \RR^N),
\]
where $T(u)$ denotes the maximal existence time. We next derive a lower bound on the existence time.

\begin{lemm}\label{lemm:existence-time}
For all $K > 0$ and $\delta \in (0, 1)$, there exists $T=T(\delta, K)>0$ such that if $\|G_{H, \ep}(u)\| \geq \delta>0$ and $D_\ep(u)\leq K$, then the existence time $T(u)$ satisfies
\[
T(u) \geq T(\delta, K).
\]
Furthermore, we have
\[
\| G_{H, \ep}(\Phi(u, \tau)) \| \geq \frac{1}{2}\delta, \quad \text{ for all }\tau \leq T(\delta, K).
\]
\end{lemm}
\begin{proof}
By Lemma \ref{lem:pseudo-gradient vector field}(2) and standard ODE theory, we see that if $T(u) < + \infty$ then 
\[
\liminf_{t \to T(u)^-}\|G_{H, \ep}(\Phi(u, t))\| = 0.
\]
Therefore to bound $T(u)$ from below we only need to obtain a lower bound for $\|G_{H, \ep} (\Phi(u, \tau)) \|$. 
First note that by Lemma~\ref{lemm:D-equivalence}(b) and the assumption that $D_{\ep}(u) \leq K$, we get 
\[
\|u\|_{2, 2} \leq A_{0}(1 + \ep^{-2}K).
\]
Next, for all $\tau < \min\{1/2, T(u)\}$, we use property (2) in Lemma~\ref{lem:pseudo-gradient vector field} to estimate
\begin{align}
\| \Phi(u, \tau) - u \|_{2, 2} &\leq \int_0^\tau \| X(\Phi(u, \rho)) \|_{2, 2} d\rho \leq 2 \int_0^\tau \min\{ \|G_{H, \ep}(\Phi(u, \rho))\|, 1 \} d\rho\nonumber\\
&\leq 2\tau < 1. 
\end{align}
In particular, 
\[
\|\Phi(u, \tau)\|_{2, 2} < A_{0}(1 + \ep^{-2}K) + 1:= L, \text{ whenever }\tau < \min\{1/2, T(u)\}.
\]
On the other hand, by Proposition \ref{prop:first-var-properties}(c) and the bounds $\|u\|_{2, 2}, \|\Phi(u, \tau)\|_{2, 2} \leq L$, we have
\[
\|G_{H, \ep} (\Phi(u, \tau)) - G_{H, \ep}(u) \| \leq C_{L} \|\Phi(u, \tau)- u\|_{2, 2} \leq 2C_{L}\tau.
\]
(Here $C_L$ depends on $H$ as well, but since the latter is fixed we omit it from the notation.)
Therefore, we get
\begin{equation}\label{eq:flow-first-var-lower-bound}
\| G_{H, \ep}(\Phi(u, \tau))\| > \frac{\delta}{2}, \text{ for all } \tau < \min\big\{ T(u), \frac{\delta}{4(C_L + 1)} \big\}.
\end{equation}
Since the flow can be continued as long as $\|G_{H, \ep}(\Phi(u, \tau))\|$ stays positive, the above lower bound implies $T(u) \geq \frac{\delta}{4(C_L + 1)} =: T(\delta, K)$, and also gives the second conclusion.
\end{proof}

We continue the proof of Proposition \ref{prop:critical-uniform-bound}. In addition to $J_k$ defined in~\eqref{eq:J_k}, consider also the following subset of the parameter space $I=[0, 1]$:
\[
I_k = \{ t\in [0, 1]\ |\ E_{H, \ep}(\gamma_k(t), f_{\gamma_k, t}) \geq \omega_{H, \ep} - \alpha_k/2 \}.
\]
Then since $t \mapsto E_{H, \ep}(\gamma_k(t), f_{\gamma_k, t})$ is continuous, we see that the set $I_k$ is compact while $J_k$ is open, and hence there exists a continuous cut-off function $\phi_k$ such that $\phi_k(t) = 1$ if $t \in I_k$ and $\phi_k(t) = 0$ if $t \notin J_k$. Also, note that by Lemma~\ref{lemm:small-D-not-maximum}, for $k$ sufficiently large we have $0, 1 \notin J_k$. Now recall that we are assuming by contradiction that 
\begin{equation}\label{eq:contra-lower-bound}
\|G_{H, \ep}(\gamma_k(t))\| \geq \delta, \text{ for all } t\in J_k.
\end{equation}
Furthermore, by the definition of the class $\cP_{\alpha_k, C_0}$, we see that
\begin{equation}\label{eq:contra-upper-bound}
D_{\ep}(\gamma_k(t)) \leq C_0, \text{ for all } t \in J_k.
\end{equation}
Hence Lemma~\ref{lemm:existence-time} guarantees that $T(\gamma_{k}(t)) \geq T(\delta, C_0)$ for all $t \in J_{k}$. For brevity, below we write $T_0$ for $T(\delta, C_0)$.

To continue, we let
\[
\Gamma_{k}(s, t) = \Phi\big(\gamma_k(t), \phi_k(t)T_0 s\big) \text{ for }(s, t) \in [0, 1] \times [0, 1],
\]
and write $\tilde{\gamma}_k$ for $\Gamma_k(1, \cdot)$. Note that $\Gamma_k(s, \cdot) \in \cP$ for all $s \in [0, 1]$ since $0, 1 \notin J_k$, and since $\Phi(u, \tau)$ depends continuously on $u$ and is Lipschitz in $\tau$ (uniformly in $u$). We claim the following three properties: First of all,
\begin{equation}
E_{H, \ep}(\tilde{\gamma}_{k}(t), f_{\tilde{\gamma}_k, t}) = E_{H, \ep}(\gamma_k(t), f_{\gamma_k, t}) + \int_{0}^{\phi_k(t)T_0} \big\langle G_{H, \ep}(\Phi(\gamma_k(t), s)), X(\Phi(\gamma_k(t), s)) \big\rangle ds. \label{eq:change-under-cutoff-flow}
\end{equation}
Secondly, 
\begin{equation}
E_{H, \ep}(\tilde{\gamma}_k(t), f_{\tilde{\gamma}_k, t}) \leq E_{H, \ep}(\gamma_k(t), f_{\gamma_k, t}) \text{ for all }t \in [0, 1] \label{eq:E-no-increase}.
\end{equation}
Thirdly, there exists a constant $\beta > 0$ independent of $k$, such that 
\begin{equation}
E_{H, \ep}(\tilde{\gamma}_k(t), f_{\tilde{\gamma}_k, t}) < \omega_{H, \ep}  - \beta \text{ for all }t \in I_k \text{ and $k$ sufficiently large}. \label{eq:peak-pushed-down}
\end{equation}

To see~\eqref{eq:change-under-cutoff-flow}, note that for all $s_0, t_0 \in [0, 1]$, if we let $\cA$ be a simply-connected neighborhood of $\Gamma_k(s_0, t_0)$, then for all $(s, t)$ sufficiently close to $(s_0, t_0)$, via the map $\Gamma_k$ itself we may construct a homotopy of extensions and apply Lemma~\ref{lemm:volume-properties}(b) to see that 
\[
E_{H, \ep}(\Gamma_k(s, t), f_{\Gamma_k(s, \cdot), t}) = E^{\cA}_{H, \ep}(\Gamma_{k}(s, t)),
\]
where $E^{\cA}_{H, \ep}$ is the local reduction induced by $(\Gamma_k(s_0, t_0), f_{\Gamma_k(s_0, \cdot), t_0})$. In particular, recalling the definition of $\Gamma_k$ in terms of the flow $\Phi$ and differentiating the above identity in $s$ gives 
\[
\frac{d}{ds}\big|_{s = s_0}E_{H, \ep}(\Gamma_k(s, t_0), f_{\Gamma_k(s, \cdot), t_0})  = \phi_k(t_0)T_0 \cdot \langle \delta E_{H, \ep}(\Gamma_k(s_0, t_0)),\ X(\Gamma_k(s_0, t_0)) \rangle.
\]
Since $s_0, t_0 \in [0, 1]$ are arbitrary, we get~\eqref{eq:change-under-cutoff-flow} upon integrating from $s = 0$ to $s = 1$, changing variables, and noting that 
\[
\langle \delta E_{H, \ep}(\Phi(\gamma_k(t), s)),\ X(\Phi(\gamma_k(t), s)) \rangle = \langle G_{H, \ep}(\Phi(\gamma_k(t), s)),\ X(\Phi(\gamma_k(t), s)) \rangle
\]
by property (1) in Lemma~\ref{lem:pseudo-gradient vector field}.

The estimate~\eqref{eq:E-no-increase} is now an easy consequence of~\eqref{eq:change-under-cutoff-flow} and Lemma~\ref{lem:pseudo-gradient vector field}(3). To get~\eqref{eq:peak-pushed-down} when $t \in I_k$, we note that by~\eqref{eq:contra-lower-bound}, ~\eqref{eq:contra-upper-bound}, and the second conclusion of Lemma~\ref{lemm:existence-time}, 
\[
\| G_{H, \ep}(\Phi(\gamma_{k}(t), s))\| \geq \delta/2 \text{ for all } t \in I_k \text{ and }s \in [0, T_0]. 
\]
Furthermore, from the proof of Lemma~\ref{lemm:existence-time} and the upper bound~\eqref{eq:contra-upper-bound}, we see that
\[
\| \Phi(\gamma_k(t), s) \|_{2, 2} \leq A_0(1 + \ep^{-2}C_0) + 1 := L, \text{ for all }t \in I_k \text{ and }s \in [0, T_0].
\]
Thus Lemma~\ref{lem:pseudo-gradient vector field}(3) tells us that 
\[
\langle G_{H, \ep}(\Phi(\gamma_k(t), s)), X(\Phi(\gamma_k(t), s)) \rangle < - \frac{\delta^2/4}{A_1 (1 + L^2)} =: -\tilde{\beta}, \text{ for all }t \in I_k, s \in [0, T_0].
\]
We emphasize that the constant $\tilde\beta$ is independent of $k$. Putting this back into~\eqref{eq:change-under-cutoff-flow}, and recalling our choice of $\gamma_k$, we get, for all $t \in I_k$, that
\begin{align*}
E_{H, \ep}(\tilde{\gamma}_{k}(t), f_{\tilde{\gamma}_k, t}) &< E_{H, \ep}(\gamma_k(t), f_{\gamma_k, t}) - T_0 \cdot \tilde{\beta}\\
&< \omega_{H, \ep} + \alpha_k - T_0 \cdot \tilde{\beta}.
\end{align*}
Letting $\beta = T_0 \cdot \tilde{\beta}/2$, then since $\tilde{\beta}$, $T_0$ are independent of $k$, and since $\alpha_k \to 0$ by assumption, we get~\eqref{eq:peak-pushed-down} for all large enough $k$ and any $t \in I_k$, as claimed. Recalling the definition of $I_k$, we see from~\eqref{eq:E-no-increase} and~\eqref{eq:peak-pushed-down} that if $k$ is sufficiently large so that $\alpha_k < \beta$, then
\[
\max_{0 \leq t\leq 1} E_{H, \ep}(\tilde{\gamma}_k(t), f_{\tilde{\gamma}_k, t}) \leq \omega_{H, \ep} - \frac{\alpha_k}{2} < \omega_{H, \ep}.
\]
Since $\tilde{\gamma}_k \in \cP$, this contradicts the definition of the min-max value. Hence the statement~\eqref{eq:property-star} must hold.

Consequently, there exists a subsequence of $\gamma_k$, which we do not relabel, and a sequence $t_k \in J_k$, such that 
\[
\|G_{H, \ep}(\gamma_k(t_k)) \| \to 0 \text{ as }t \to \infty.
\]
Moreover, as $t_k \in J_k$, we have by the definition of $\cP_{\alpha_k, C_0}$ that 
\[
D_{\ep}(\gamma_{k}(t_k)) \leq C_0 \text{ for all }k. 
\]
Proposition~\ref{prop:PS} now shows that, passing to a subsequence if necessary, $\gamma_k(t_k)$ converges strongly in $W^{2, 2}$ to a limit $u \in W^{2}(S^2; S^3)$ satisfying $\delta E_{H, \ep}(u) = 0$ and $D_{\ep}(u) \leq C_0$.

We are now ready to deduce the conclusions of Proposition~\ref{prop:critical-uniform-bound}. First, part (a) follows immediately from the definitions of $\cP_{\alpha_k, C_0}$ and $J_k$. To prove part (b), it remains to show that there exists some extension $f\in \cE(u)$ such that 
\begin{equation}\label{eq:min-max-extension}
E_{H, \ep}(u, f) = \omega_{H, \ep}.
\end{equation}
To see that, note that by the strong $W^{2, 2}$-convergence established above, for $k$ sufficiently large we may consider the extension $f_k \in \cE(u)$ formed by concatenating $f_{\gamma_k, t_k}$ with the map $(s, x) \mapsto \Pi(su(x) + (1 - s)\gamma_k(t_k)(x))$. A straightforward computation similar to the one leading to~\eqref{eq:hi-V-bound} then shows that 
\[
\lim_{k \to \infty} |V(f_k) - V(f_{\gamma_k, t_k})| = 0,
\]
and consequently, by conclusion (a) and the strong $W^{2, 2}$-convergence of $\gamma_k(t_k)$ to $u$, we have
\[
\lim_{k \to \infty} |E_{H, \ep}(u, f_k)  - \omega_{H, \ep}|= 0.
\]
By Lemma~\ref{lemm:volume-properties}(a) this implies that the sequence $E_{H, \ep}(u, f_k)$ is eventually constantly equal to $\omega_{H, \ep}$. In particular we get some $f \in \cE(u)$ satisfying~\eqref{eq:min-max-extension}.

To prove conclusion (c), suppose by contradiction that $D_{\ep}(u) < \eta_1(\ep, H)$. Then by the strong $W^{2, 2}$-convergence of $\gamma_k(t_k)$ to $u$, for sufficiently large $k$ we may apply Lemma~\ref{lemm:small-D-not-maximum} to see that 
\[
E_{H, \ep}(\gamma_{k}(t_k), f_{\gamma_k, t_k}) < \max_{0 \leq t \leq 1} E_{H, \ep}(\gamma_k(t), f_{\gamma_k, t}) - \eta_2(\ep, H).
\]
Recalling our choice of $\gamma_k$ and the definition of $J_k$, the above implies that 
\[
\omega_{H, \ep} - \alpha_k < \omega_{H, \ep} + \alpha_k - \eta_2,
\]
for all $k$ large enough. Since $\alpha_k \to 0$ as $k \to\infty$, this is a contradiction, and we get (c).
\end{proof}

\subsection{Deformation lemma and index upper bound}\label{subsec:deformation}

In this part we exhibit non-trivial critical points to the perturbed functional with bounded index and energy. The final ingredient needed for this, in addition to results from the previous two sections, is a deformation lemma inspired by the Morse index estimates in the Almgren-Pitts min-max theory by Marques-Neves \cite{Marques-Neves16} and Song \cite{Song19}. In particular, we do not need to perturb the functional to be Morse as in the classical approach (see \cite{Micallef-Moore88}). We obtain critical points with the desired properties in Corollary~\ref{coro:critical-uniform-bound}.

\begin{lemm}[Deformation Lemma]\label{lemm:deformation}
Given $H > 0$, $\ep \in (0, 1]$ and $C_0 > 0$, let $\cK'$ be a compact subset of $\cK_{C_0 + 1}$ and suppose $\Ind_{H, \ep}(v) \geq 2$ for all $v \in \cK'$. Then, for each sequence of sweepouts $\{\gamma_k\}\subset \cP$, where $\gamma_k\in \cP_{\alpha_k, C_0}$ with $\alpha_k\to 0$, there exist another sequence of sweepouts $\{\widetilde\gamma_k\}$ such that
\begin{enumerate}
\item[(a)] $\widetilde\gamma_k \in \cP_{\alpha_k, C_0+1}$ for $k$ large enough.
\vskip 1mm
\item[(b)] For any $u\in \cK'$, there exists $k_0(u), d_0(u)>0$, such that for all $k\geq k_0(u)$, 
\[ 
\inf\{\| \widetilde\gamma_k(t) - u \|_{2,2}: E_{H, \ep}(\widetilde\gamma_k(t), f_{\widetilde{\gamma}_k, t}) \geq \omega_{H, \ep}-\alpha_k \}  \geq d_0(u).
\]
\end{enumerate}
\end{lemm}
Some preparations are in order before we give the proof. Suppose  $v$ is a critical point of $E_{H, \ep}$ lying in $\cK'$. Then by definition there exists $f \in \cE(v)$ such that 
\[
E_{H, \ep}(v, f) = \omega_{H, \ep}.
\]
Next we let $\cA$ be a generalized Morse neighborhood of $v$ given by Proposition~\ref{defiprop:Morse-neighborhood}. Then~\eqref{eq:generalized-Morse} holds with $\cB, \Psi, H_0, H_-, H_+$ and $e$ as in Definition~\ref{defiprop:Morse-neighborhood} and with $E^{\cA}_{H, \ep}$ being the local reduction induced by $(v, f)$. In particular, $\dim H_{-} = \Ind_{H, \ep}(v) \geq 2$ by assumption, and the function $e$ satisfies
\begin{equation}\label{eq:e-bound}
|e(\psi_0) -\omega_{H, \ep} | \leq a(\|\psi_0\|)\|\psi_0\|^2,
\end{equation}
where $\lim_{t \to 0+}a(t) = 0$. For later use we introduce the following additional notation. For $r_1, r_2, r_3 > 0$ small, we let
\[
C(r_1, r_2, r_3) = \Psi(B_{r_1}^{0} \oplus B_{r_2}^{-} \oplus B_{r_3}^{+}),
\]
and also define
\[
\partial_{-}C(r_1, r_2, r_3) = \Psi(B_{r_1}^{0} \oplus \partial B_{r_2}^{-} \oplus B_{r_3}^{+}).
\]
Here  $B_{r_1}^0$, $B_{r_2}^-$ and $B_{r_3}^+$ denote open balls in $H_0$, $H_-$ and $H_+$ respectively. Also, we denote the closure of $C(r_1, r_2, r_3)$ by $\overline{C}(r_1, r_2, r_3)$. We now give the proof of Lemma~\ref{lemm:deformation}. 
\begin{proof}[Proof of Lemma~\ref{lemm:deformation}]
The proof proceeds in five stages and is similar in structure to that of~\cite[Theorem 7]{Song19}. In our case, while the existence of generalized Morse neighborhoods offers some convenience, passing between different local reductions requires extra care.  
\vskip 1mm
\noindent\textbf{Step 1: First covering of $\cK'$.} 
\vskip 1mm
To begin, take $v \in \cK'$ and let $\Psi_v$, $\cB_v$ and $\cA_v$ be as above, where we added subscripts to indicate their dependence on $v$. For $\tau = \tau(v) > 0$ small enough, we have $C(\tau, \tau, \tau) \subset \cA_v$ along with the following properties: First, 
\begin{equation}\label{eq:error-small}
|e(\psi_{0}) - \omega_{H, \ep}| \leq \frac{1}{4}\|\psi_0\|^2 \text{ for }\|\psi_0\| \leq \tau.
\end{equation}
Secondly, for all $w \in C(\tau, \tau, \tau)$,
\begin{equation}\label{eq:E-close}
|E^{\cA_v}_{H, \ep}(w) - \omega_{H, \ep}| < \frac{\delta}{4},
\end{equation}
\begin{equation}\label{eq:D-close}
|D_{\ep}(w) - D_{\ep}(v)| < \frac{1}{2},
\end{equation}
where 
\[
\delta = \min\big\{\frac{\eta_2(\ep, H)}{2}, \frac{H}{4}\Vol_g(S^3), \frac{1}{2}\big\}.
\]
Here $\eta_2$ is from Lemma~\ref{lemm:small-D-not-maximum}. Note also that $\delta$ is independent of $v$. Next we choose $\rho_1(v), \rho_2(v), \rho_3(v) < \frac{\tau(v)}{2}$ so that 
\begin{equation}\label{eq:rho-relation}
\rho_1^2 - 4\rho_2^2 + 4\rho_3^2 < 0.
\end{equation}
(For instance we may fix $\rho_2 = \tau/4$ and then choose $\rho_1, \rho_3$ small.) Then from~\eqref{eq:error-small} and~\eqref{eq:generalized-Morse} we see that 
\begin{equation}\label{eq:lateral-negative}
-\beta = -\beta(v):= \sup\{E^{\cA}_{H, \ep}(w) - \omega_{H, \ep} |\ w \in \partial_{-}C(2\rho_1, 2\rho_2, 2\rho_3)\} < 0.
\end{equation}
Moreover,~\eqref{eq:E-close} and~\eqref{eq:D-close} hold on $\overline{C}(2\rho_1, 2\rho_2, 2\rho_3)$. We then choose $r = r(v)$ such that 
\[
B^{2, 2}_{2r}(v) := \{w \in W^{2, 2}(S^2; S^3)\ |\ \|w - v\|_{2, 2} < 2r\} \subset C(\rho_1, \rho_2, \rho_3).
\]
The collection of open sets $\{B^{2, 2}_{r(v)}(v)\}_{v \in \cK'}$ obviously covers $\cK'$, and we can extract a finite subcovering $\{B^{2, 2}_{r(v_i)}(v_i)\}_{i = 1}^L$. Below we write $r_i = r(v_i)$, $\beta_i = \beta(v_i)$ and let 
\[
\underline{r} = \min_{1 \leq i \leq L} r_i,\ \underline{\beta} = \min_{1 \leq i \leq L}\beta_i.
\]
\vskip 1mm
\noindent\textbf{Step 2: Second covering of $\cK'$.} 
\vskip 1mm
We proceed to describe the second covering. For $w \in \cK'$, let $\cA_{w}, \Psi_{w}$ and $\rho_i = \rho_i(w)$ ($i = 1, 2, 3$)  have the same meaning as in Step 1. Note in particular that $\Psi_{w}^{-1}(w) = 0$. Using successively the continuity of the functions
\[
(u, \xi_{-}) \mapsto \|\Psi_{w}(\Psi_w^{-1}(u) + \xi_-) - u\|_{2, 2}, \text{ and }
\]
\[
(u, \xi_{-}) \mapsto E^{\cA_w}_{H, \ep}(\Psi_{w}(\Psi_w^{-1}(u) + \xi_-)) - E^{\cA_w}_{H, \ep}(u),
\]
defined for $(u, \xi_{-}) \in C(\rho_1, \rho_2, \rho_3) \times B^{-}_{\rho_2}$, we find a radius $s = s(w) > 0$, an element $\xi_{-} = \xi_{-}(w) \in B_{\rho_2}^{-}$ and a threshold $c_w \in (0, \delta)$, with the following properties: First,
\begin{equation}\label{eq:refine}
B^{2, 2}_{s(w)}(w) \subset C(\rho_1(w), \rho_2(w), \rho_3(w)) \cap \big(\cup_{i = 1}^L B^{2, 2}_{r_{i}}(v_{i})\big).
\end{equation}
Second, and more importantly, for $u \in B^{2, 2}_{s(w)}(w)$, writing $u_t = \Psi_{w}(\Psi_w^{-1}(u) + t\xi_{-}(w))$, we have that
\begin{equation}\label{eq:distance-short}
\|u_{t} - u\|_{2, 2} < \underline{r} \text{ for all }t \in [0, 1].
\end{equation}
\begin{equation}\label{eq:energy-decrease}
E^{\cA_{w}}_{H, \ep}(u_1) - E^{\cA_{w}}_{H, \ep}(u) < -\frac{c_w}{2}.
\end{equation}
It follows from~\eqref{eq:distance-short} and the definition of $\underline{r}$ that for all $u \in B^{2, 2}_{s(w)}(w)$ and any $i \in \{1, \cdots L\}$ such that $u \in B^{2, 2}_{r_i}(v_i)$, we have $u_t \in B^{2, 2}_{2r_i}(v_i)$ for all $t \in [0, 1]$. In particular, by our choice of $r_i$, 
\[
|E_{H, \ep}^{\cA_i}(u_1) - \omega_{H, \ep}|,\ |E_{H, \ep}^{\cA_i}(u) - \omega_{H, \ep}| < \frac{\delta}{4}.
\]
Since $u_{t}$ also lies in $C(2\rho_1, 2\rho_2, 2\rho_3)$, by~\eqref{eq:E-close} the above hold with $\cA_i$ replaced by $\cA_w$. Hence~\eqref{eq:energy-decrease} and Lemma~\ref{lemm:volume-properties}(a) give
\begin{equation}\label{eq:energy-decrease-compare}
E^{\cA_{i}}_{H, \ep}(u_1) - E^{\cA_{i}}_{H, \ep}(u) < -\frac{c_w}{2},
\end{equation}
for all $u \in B^{2, 2}_{s(w)}(w)$ and $i \in \{1, \cdots, L\}$ such that $u \in B^{2, 2}_{r_i}(v_i)$. 

Now, out of the open covering $\{B_{s(w)}(w)\}_{w \in \cK'}$ of $\cK'$, we extract a finite subcover $\{B_{s_j}(w_j)\}_{j = 1}^{M}$, where we've written $s_j = s(w_j)$. Then we define 
\[
\underline{c} = \frac{1}{4}\min_{1 \leq j \leq M} c_{w_j}.
\]
Note particularly that $\underline{c} < \frac{\delta}{4}$. Moreover, for $\theta > 0$ we let 
\[
\cN_{\theta} = \cup_{v \in \cK'}B^{2, 2}_{\theta}(v),
\]
and by the compactness of $\cK'$ we may fix a $\theta>0$ sufficiently small so that $\cN_{\theta} \subset \cup_{j = 1}^{M}B^{2, 2}_{s_j}(w_j) \subset \cup_{i = 1}^L B^{2, 2}_{r_i}(v_i)$, where the second containment follows from~\eqref{eq:refine}. We conclude Step 2 with the following lemma which summarizes some of the properties we established above.

\begin{lemm}\label{lemm:perturb-out}
With $\cN_{\theta}$ as above, for all $u \in \cN_{\theta}$ there exists a continuous path $t \mapsto u_t$ with $u_0 = u$ such that for all $i \in \{1, \cdots, L\}$ with $u \in B^{2, 2}_{r_i}(v_i)$, we have $u_{t} \in B^{2, 2}_{2r_i}(v_i)$ for all $t \in [0, 1]$, and that 
\[
E^{\cA_{i}}_{H, \ep}(u_1) \leq E^{\cA_{i}}_{H, \ep}(u) - 2\underline{c}.
\]
(Note that $u$ must belong to some $B^{2, 2}_{r_i}(v_i)$ since $\cN_{\theta} \subset \cup_{i = 1}^L B^{2, 2}_{r_i}(v_i)$.)
\end{lemm}

\begin{proof}
For $u$ as given, since $\cN_{\theta} \subset \cup_{j = 1}^M B_{s_j}(w_j)$ we may choose $j$ such that $u \in B_{s_j}(w_j)$. The conclusion then is just a restatement of~\eqref{eq:distance-short} and~\eqref{eq:energy-decrease-compare}, using the definition of $\underline{c}$. 
\end{proof}
\vskip 1mm
\noindent\textbf{Step 3: Subdivision of the parameter space and deformation of endpoints.}
\vskip 1mm
Take $\gamma_{k} \in \cP_{\alpha_k, C_0}$ and define the intervals
\[
I_{k} = \{t \in [0, 1]\ |\ E_{H, \ep}(\gamma_k(t), f_{\gamma_k, t}) > \omega_{H, \ep} - \delta \text{, and }\gamma_{k}(t) \in \overline{\cN_{3\theta /4}}\},
\]
\[
J_{k} = \{t \in [0, 1]\ |\ E_{H, \ep}(\gamma_k(t), f_{\gamma_k, t}) > \omega_{H, \ep} - \delta \text{, and }\gamma_{k}(t) \in \cN_{\theta}\}.
\]
Note that $0, 1 \notin J_k$ by Lemma~\ref{lemm:small-D-not-maximum} and our choice of $\delta$. Next we claim that for $k$ large enough so that $\alpha_k < \delta$, we have
\[
E_{H, \ep}(\gamma_k(t), f_{\gamma_k, t}) > \omega_{H, \ep} - \frac{\delta}{4} \text{ for all }t \in J_k.
\]
To see that, note that by the definition of $\cP_{\alpha_k, C_0}$ and $J_{k}$, for $t$ belonging to the latter we have
\[
E_{H, \ep}(\gamma_k(t), f_{\gamma_k, t}) \in [\omega_{H, \ep} - \delta, \omega_{H, \ep} + \alpha_k].
\]
On the other hand, by our choice of $\cN_{\theta}$ there exists some $i \in \{1, \cdots, L\}$ such that $\gamma_k(t) \in B^{2, 2}_{r_i}(v_i)$, and hence by~\eqref{eq:E-close} we have
\[
|E_{H, \ep}^{\cA_i}(\gamma_k(t)) - \omega_{H, \ep}| < \frac{\delta}{4}.
\]
Hence, as soon as $k$ is so large that $\alpha_k < \delta$, we may use Lemma~\ref{lemm:volume-properties}(a) and our choice of $\delta$ to infer that 
\[
E_{H, \ep}(\gamma_k(t), f_{\gamma_k, t}) = E_{H, \ep}^{\cA_i}(\gamma_k(t)) > \omega_{H, \ep} - \frac{\delta}{4},
\]
for all $t \in J_k$ as claimed. 

To continue, we drop the subscript $k$ in $\gamma_k, \alpha_k, I_{k}, J_{k}$ and so on. Note that by the above claim, for $k$ sufficiently large, $I$ is a compact subset of the open set $J$, and hence we may find finitely many disjoint intervals $\{[a_n, b_n]\}_{n = 1}^m$ such that 
\begin{equation}\label{eq:interval-choice}
I \subset \cup_{n = 1}^m [a_n, b_n] \subset J, \text{ and that }\gamma(a_n), \gamma(b_n) \notin \cN_{\theta/2}.
\end{equation}
Since $\gamma([a_n, b_n]) \subset \cN_{\theta} \subset \cup_{j = 1}^{M}B^{2, 2}_{s_j}(w_j) \subset \cup_{i = 1}^{L}B^{2, 2}_{r_i}(v_i)$, we can further partition $[a_n, b_n]$ by 
\[
a_n = t_0 < t_1 < \cdots < t_{p} = b_n,
\]
such that for all $l = 0, \cdots, p-1$, there exists $i = i(l)$ with 
\[
\gamma([t_{l}, t_{l + 1}]) \subset B^{2, 2}_{r_i}(v_i).
\]
Note that in the notation we have further suppressed the $n$-dependence of the partition $\{t_l\}$ and the indices $\{i(l)\}$, apart from their $k$-dependence.

For later purposes, we note that for $k$ large enough so that $\alpha < \frac{\delta}{4}$, we have
\begin{equation}\label{eq:volume-agree}
E_{H, \ep}(\gamma(t), f_{\gamma, t}) = E^{\cA_{i(l)}}_{H, \ep}(\gamma(t)), \text{ for all }t \in [t_l, t_{l + 1}].
\end{equation}
Indeed, by the definition of $\cP_{\alpha, C_{0}}$ and $J$, and by~\eqref{eq:E-close}, both sides lie within $\delta$ of $\omega_{H, \ep}$, and hence must agree by Lemma~\ref{lemm:volume-properties}(a).

Next, for $l = 1, \cdots, p-1$, since $\gamma(t_{l}) \in \cN_{\theta}$, we may apply Lemma~\ref{lemm:perturb-out} to find paths $P_{l}:[0, 1] \to W^{2, 2}(S^2; S^3)$ such that $P_{l}(0) = \gamma(t_l)$ and
\vskip 1mm
\begin{enumerate}
\item[(a1)] $P_{l}(t) \subset B_{2r_{i(l-1)}}(v_{i(l-1)}) \cap B_{2r_{i(l)}}(v_{i(l)})$ for all $t \in [0, 1]$.
\vskip 1mm
\item[(a2)] $E_{H, \ep}^{\cA_{i}}(P_{l}(1)) \leq E_{H, \ep}^{\cA_{i}}(\gamma(t_l)) - 2\underline{c}$ for $i = i(l-1), i(l)$.
\end{enumerate}
Note that (a2) and~\eqref{eq:volume-agree} imply that if $\alpha < \underline{c}$ then 
\begin{equation}\label{eq:deform-energy-down}
E_{H, \ep}^{\cA_{i}}(P_{l}(1)) \leq \omega_{H, \ep} - \underline{c} \text{ for }i = i(l-1), i(l).
\end{equation}
Borrowing the notation from Song \cite[Theorem 7]{Song19}, below we use ``$+$'' to denote concatenation of paths and ``$-$'' to denote reversal of orientation of a path, and let
\[
h_{0} = \gamma \big|_{[t_0, t_1]} + P_1,\ h_{p-1} = -P_{p-1} + \gamma\big|_{[t_{p-1}, t_p]}.
\]
For $l = 1, \cdots, p-2$, we let
\[
h_{l} = -P_l + \gamma\big|_{[t_l, t_{l + 1}]} + P_{l +1}.
\]
Note, then, that $\gamma\big|_{[a_n, b_n]}$ is homotopic to $h_0 + h_1 + \cdots + h_{p-1}$. Also, the endpoints $h_{0}(0), h_{p-1}(1) \notin \cN_{\theta/2}$ by~\eqref{eq:interval-choice} above.

\vskip 1mm
\noindent\textbf{Step 4: Replacing $\gamma_k$ on the sub-intervals.}
\vskip 1mm

In this step we will further replace each $h_l$ by homotopic paths. Note that by the definition of $[t_l, t_{l + 1}]$ and by property (a1) in Step 3, the path $h_l$ maps into $B^{2, 2}_{2r_{i(l)}}(v_{i(l)}) \subset C(\rho_1(v_{i(l)}), \rho_2(v_{i(l)}), \rho_3(v_{i(l)})) \subset \overline{C}(2\rho_1(v_{i(l)}), 2\rho_2(v_{i(l)}), 2\rho_3(v_{i(l)}))$. Since the latter is simply-connected, being the homeomorphic image of a convex set, to ensure that the replacement is homotopic to $h_{l}$, it suffices to keep the replacement path inside $\overline{C}(2\rho_1(v_{i(l)}), 2\rho_2(v_{i(l)}), 2\rho_3(v_{i(l)}))$, while keeping endpoints $h_{l}(0), h_{l}(1)$ fixed. To reduce notation, below we drop the subscripts $l$. We also write $\rho_{1, i}, \rho_{2, i}, \rho_{3, i}$ for $\rho_1(v_{i(l)}), \rho_2(v_{i(l)}), \rho_3(v_{i(l)})$, respectively.

We proceed to describe the replacement path. First we define 
\[
\widehat{h} = \Psi_{i}^{-1}\circ h.
\]
Then, for $j = 0, 1$, we pick $\xi(j) \in H_{-}$ by letting
\[
\xi(j) = \left\{
\begin{array}{ll}
\widehat{h}(j)_{-} & \text{, if }\widehat{h}(j)_{-} \neq 0,\\
\text{arbitrary non-zero element in $H_{-}$} & \text{, if }\widehat{h}(j)_{-} = 0.
\end{array}
\right.
\]
Now we define a path $q_j: [0, 2] \to \overline{C}(\rho_{1, i}, 2\rho_{2, i}, \rho_{3, i})$ in terms of $\widehat{q}_j := \Psi_{i}^{-1}\circ q_j$ by letting
\[
\widehat{q}_j(t) = \left\{
\begin{array}{ll}
\widehat{h}(j)_0 + \big[\widehat{h}(j)_{-} + t(2\rho_{2, i} - \|\widehat{h}(j)_{-}\|)\frac{\xi(j)}{\|\xi(j)\|} \big] + \widehat{h}(j)_{+} & \text{, if }0 \leq t \leq 1,\\
\ &\\
\widehat{h}(j)_0 + \big[\widehat{h}(j)_{-} + (2\rho_{2, i} - \|\widehat{h}(j)_{-}\|)\frac{\xi(j)}{\|\xi(j)\|} \big] + (2 - t)\widehat{h}(j)_{+} & \text{, if } 1 \leq t \leq 2.
\end{array}
\right.
\]
Then the paths $q_0, q_1$ have the following properties.
\begin{enumerate}
\item[(b1)] $q_{j}([1, 2]) \subset \partial_{-}C(\rho_{1, i}, 2\rho_{2, i}, \rho_{3, i})$ and $q_j(2) \in \Psi_i(B_{\rho_{1, i}}^0 \oplus \partial B_{2\rho_{2, i}}^- \oplus \{0\})$.
\vskip 1mm
\item[(b2)] $E^{\cA_{i}}_{H, \ep}(q_j(t))$ is decreasing for $t \in [0, 2]$.
\vskip 1mm
\item[(b3)] For $t \in [0, 1]$, we have
\begin{equation}\label{eq:distance-barrier}
\|\widehat{q}_j(t) - \widehat{q}_j(0)\|^2 \leq E^{\cA_{i}}_{H, \ep}(q_j(0)) -  E^{\cA_{i}}_{H, \ep}(q_j(t)).
\end{equation}
\end{enumerate}
\vskip 1mm
Property (b1) and (b2) are obvious, while (b3) follows from direct computation using the definition of $\widehat{q}_j$ and the fact that $\|\widehat{h}(j)_{-}\| \leq \rho_{2, i}$. Indeed, we clearly have
\[
\widehat{q}_j(t) - \widehat{q}_j(0) =  t(2\rho_{2, i} - \|\widehat{h}(j)_{-}\|)\frac{\xi(j)}{\|\xi(j)\|}.
\]
On the other hand, by~\eqref{eq:generalized-Morse}, it's easy to see that 
\begin{align*}
E^{\cA_{i}}_{H, \ep}(q_j(0)) -  &E^{\cA_{i}}_{H, \ep}(q_j(t)) \\
&= \left\{
\begin{array}{ll}
4\rho_{2, i}^2 t^2 & \text{, if }\widehat{h}(j)_{-} = 0,\\
\big(\|\widehat{h}(j)_{-}\| + (2\rho_{2, i}-\|\widehat{h}(j)_{-}\|)t\big)^2 - \|\widehat{h}(j)_{-}\|^2 & \text{, if }\widehat{h}(j)_{-} \neq 0.
\end{array}
\right.
\end{align*}
Since $\|\widehat{h}(j)_{-}\| \leq \rho_{2, i}$, in either case we have
\[
E^{\cA_{i}}_{H, \ep}(q_j(0)) -  E^{\cA_{i}}_{H, \ep}(q_j(t)) \geq (2\rho_{2, i} - \|\widehat{h}(j)_{-}\|)^2 t^2 = \|\widehat{q}_j(t) - \widehat{q}_j(0)\|^2,
\]
as asserted.

From property (b1), $q_0(2)$ and $q_1(2)$ both belong to the (finite-dimensional) set $\Psi_i(B_{\rho_{1, i}}^0 \oplus \partial B_{2\rho_{2, i}}^- \oplus \{0\})$, which is connected since $\dim H_{-} = \Ind_{H, \ep}(v_i) \geq 2$ by assumption. Therefore we may connect $q_0(2)$ to $q_1(2)$ by a continuous path $q_2:[0, 1] \to \Psi_i(B_{\rho_{1, i}}^0 \oplus \partial B_{2\rho_{2, i}}^- \oplus \{0\})$. We then replace the path $h$ by the concatenation
\[
\widetilde{h}:= q_0 + q_2 - q_1.
\]
We now define $\widetilde{\gamma}$ to be the path obtained from $\gamma$ by first replacing each $\gamma\big|_{[a_n, b_n]}$ with $h_0 + \cdots + h_{p-1}$, and then replacing each $h_l$ by $\widetilde{h}_l$. Reparametrizing if necessary, we may assume that $\widetilde{\gamma}\big|_{[a_n, b_n]} = \widetilde{h}_0 + \cdots + \widetilde{h}_{p-1}$ and that
\[
\widetilde{\gamma}\big|_{[t_l, t_{l + 1}]} = \widetilde{h}_{l}, 
\]
with $q_0|_{[0, 1]}, q_0|_{[1, 2]}, q_2, -\big(q_1|_{[1, 2]}\big)$ and $-\big(q_1|_{[0, 1]}\big)$ successively occupying a fifth of $[t_l, t_{l + 1}]$.

\vskip 1mm
\noindent\textbf{Step 5: Verification of properties.}
\vskip 1mm
Below we let 
\[
I' = \cup_{n = 1}^m [a_n, b_n].
\]
By construction, $\gamma(t)  = \widetilde{\gamma}(t)$ for all $t \notin I'$, while $\gamma\big|_{[a_n, b_n]}$ is homotopic to $\widetilde{\gamma}\big|_{[a_n, b_n]}$ for each subinterval of $I'$. Since $0, 1 \notin I'$, this implies that $\widetilde{\gamma} \in \cP$. Moreover, by Lemma~\ref{lemm:volume-properties}(b) we see that 
\begin{equation}\label{eq:unchanged-off-I}
E_{H, \ep}(\gamma(t), f_{\gamma, t}) = E_{H, \ep}(\widetilde{\gamma}(t), f_{\widetilde{\gamma}, t}) \text{ for }t \notin I'.
\end{equation}
On the other hand, for $t$ lying in one of the subintervals $[a_n, b_n]$ of $I'$, note that if $i, j \in \{1, \cdots L\}$ such that 
\[
\widetilde{\gamma}(t) \in B_{2r_i}(v_i) \cap B_{2r_j}(v_j),
\]
then by~\eqref{eq:E-close} we have $|E_{H, \ep}^{\cA_{i}}(\widetilde{\gamma}(t)) - \omega_{H, \ep}|,\ |E_{H, \ep}^{\cA_{j}}(\widetilde{\gamma}(t)) - \omega_{H, \ep}| < \frac{\delta}{4}$, and hence, by Lemma~\ref{lemm:volume-properties}(a),
\[
E_{H, \ep}^{\cA_{i}}(\widetilde{\gamma}(t)) = E_{H, \ep}^{\cA_{j}}(\widetilde{\gamma}(t)).
\]
Combining this with the fact that 
\[
E_{H, \ep}(\widetilde{\gamma}(t_0), f_{\widetilde{\gamma}, t_0}) = E_{H, \ep}(\gamma(t_0), f_{\gamma, t_0}) = E_{H, \ep}^{\cA_{i(0)}}(\gamma(t_0)) = E_{H, \ep}^{\cA_{i(0)}}(\widetilde{\gamma}(t_0)),
\]
where the middle equality follows from~\eqref{eq:volume-agree} from Step 3, we can prove by Lemma~\ref{lemm:volume-properties}(a) and induction on $l$ that 
\begin{equation}\label{eq:volume-tilde-agree}
E_{H, \ep}(\widetilde{\gamma}(t), f_{\widetilde{\gamma}, t}) = E_{H, \ep}^{\cA_{i(l)}}(\widetilde{\gamma}(t)) \text{ for }t \in [t_l, t_{l + 1}],\ l  = 0, \cdots p-1.
\end{equation}
Now by properties (b1), (b2), the estimates~\eqref{eq:lateral-negative}, ~\eqref{eq:deform-energy-down}, and~\eqref{eq:volume-agree}, and the fact that $\widetilde{\gamma}$ agrees at $t = a_n, b_n$ with $\gamma$, where the latter lies in $\cP_{\alpha, C_0}$ by assumption, we see that, when $k$ is large enough, the functional value $E_{H, \ep}^{\cA_{i(l)}}(\widetilde{\gamma}(t))$ is bounded by
\begin{align}\label{eq:q-upper-bound-outer}
&\omega_{H, \ep} + \alpha,\text{ when $t$ belongs to the first fifth of $[t_0, t_1]$ or the last fifth of $[t_{p-1}, t_{p}]$},
\end{align}
and by
\begin{equation}\label{eq:q-upper-bound-inner}
\max\{\omega_{H, \ep} - \underline{\beta}, \omega_{H, \ep} - \underline{c}\} \text{, elsewhere on $[t_0, t_p] = [a_n ,b_n]$}.
\end{equation}
Combining this with~\eqref{eq:unchanged-off-I} and~\eqref{eq:volume-tilde-agree}, we see that for $k$ sufficiently large we have
\begin{equation}\label{eq:tilde-upper-bound}
E_{H, \ep}(\widetilde{\gamma}(t), f_{\widetilde{\gamma}, t}) \leq \omega_{H, \ep} + \alpha \text{ for all }t \in [0, 1].
\end{equation}

Next we verify the second part of the definition of $\cP_{\alpha, C_0 + 1}$ for $k$ large. Take $t \in [0, 1]$ such that $E_{H, \ep}(\widetilde{\gamma}(t), f_{\widetilde{\gamma}, t}) \geq \omega_{H, \ep} - \alpha$. Then by~\eqref{eq:volume-tilde-agree} and~\eqref{eq:q-upper-bound-inner}, for $k$ sufficiently large so that $\alpha < \underline{\beta}, \underline{c}$,  we have only three possibilities:
\begin{enumerate}
\item[(i)] $t \notin I'$.
\vskip 1mm
\item[(ii)] $t$ belongs to the first fifth of $[t_0, t_1]$.
\vskip 1mm
\item[(iii)] $t$ belongs to the last fifth of $[t_{p-1}, t_{p}]$.
\end{enumerate}
If $t \notin I'$ then by~\eqref{eq:unchanged-off-I} we have $E_{H, \ep}(\gamma(t), f_{\gamma, t}) \geq \omega_{H, \ep} - \alpha$, and 
\[
D_{\ep}(\widetilde{\gamma}(t)) = D_{\ep}(\gamma(t)) \leq C_0 < C_0  +1.
\]
On the other hand, cases (ii) and (iii) are similar, so we only consider (ii). We note that, using the notation from Step 4, the endpoint $q_0(0)$ is exactly $\gamma(a_n)$. Combining this with property (b2) and~\eqref{eq:volume-agree} and~\eqref{eq:volume-tilde-agree}, we see that 
\[
E_{H, \ep}(\gamma(a_n), f_{\gamma, a_n}) = E^{\cA_{i(0)}}_{H, \ep}(\gamma(a_n)) \geq E^{\cA_{i(0)}}_{H, \ep}(\widetilde{\gamma}(t)) = E_{H, \ep}(\widetilde{\gamma}(t), f_{\widetilde{\gamma}, t}) \geq \omega_{H, \ep} - \alpha.
\]
Hence $D_{\ep}(\gamma(a_n)) \leq C_0$, but then since $\gamma(a_n)$ and $\widetilde{\gamma}(t)$ both belong to 
\[
\overline{C}(2\rho_1(v_{i(0)}), 2\rho_2(v_{i(0)}), 2\rho_3(v_{i(0)})),
\] 
by~\eqref{eq:D-close} we see that
\[
D_{\ep}(\widetilde{\gamma}(t)) < D_{\ep}(\gamma(a_n)) + 1\leq C_0 + 1.
\]
This proves assertion (a) of Lemma~\ref{lemm:deformation}, that is $\widetilde{\gamma} \in \cP_{\alpha, C_0 + 1}$.

To prove (b), we need to put back the subscripts we dropped. Thus for example $\widetilde{\gamma}$ becomes $\widetilde{\gamma}_k$, the number $p$ of subintervals of $[a_{k, n}, b_{k, n}]$ is $p_{k, n}$, the interval $[t_{l}, t_{l + 1}]$ is actually $[t_{k, n, l}, t_{k, n, l + 1}]$, and the indices $i(l)$ are now $i(k, n, l)$, while the paths $q_j$ are now $q_{k, n, l, j}$. To continue, assume by contradiction that there exists $u \in \cK'$ and a subsequence of $(\widetilde{\gamma}_k, \alpha_k)$, which we do not relabel, along with times $t_k \in [0, 1]$, such that 
\begin{equation}\label{eq:contradiction-min-max-energy}
E_{H, \ep}(\widetilde{\gamma}_k(t_k), f_{\widetilde{\gamma}_k, t_k}) \geq \omega_{H, \ep} - \alpha_k,
\end{equation}
and that
\begin{equation}\label{eq:contradiction-convergence}
\lim_{k \to \infty}\|\widetilde{\gamma}_k(t_k) - u\|_{2, 2}  = 0.
\end{equation}
First note that we must eventually have $t_k \in I'_k$, since for all $t \notin I'_k$, either 
\[
E_{H, \ep}(\widetilde{\gamma}_k(t), f_{\widetilde{\gamma}_k, t}) = E_{H, \ep}(\gamma_k(t), f_{\gamma_k, t}) \leq \omega_{H, \ep} - \delta,
\]
or 
\[
\widetilde{\gamma}_k(t) = \gamma_k(t ) \notin \cN_{\theta/2}.
\]
Also, we see by~\eqref{eq:volume-tilde-agree}, ~\eqref{eq:q-upper-bound-inner} and~\eqref{eq:contradiction-min-max-energy} that there exist a sequence $n_k \in \NN$ such that for all large enough $k$, the slice $\widetilde{\gamma}_k(t_k)$ lies in either $q_{k, n_k, 0, 0}([0, 1])$ or $q_{k, n_k, (p_{k, n_k} - 1), 1}([0, 1])$. Passing to a subsequence if necessary, we may assume without loss of generality that the first alternative always happens, and, by slight abuse of notation, we write
\[
\widetilde{\gamma}_k(t_k) = q_{k, n_k, 0, 0}(t_k) =: Q_k(t_k).
\] 
Moreover, since $i(k, n_k, 0) \in \{1, \cdots, L\}$, up to taking a further subsequence, we may also assume that the indices $i(k, n_k, 0) = 1$ for all $k$. In particular, $Q_k(t_k)$ lies in $\overline{C}(2\rho_1(v_1), 2\rho_2(v_1), 2\rho_3(v_1))$ for all $k$ large, and hence so does the limit $u$. Consequently we may apply $\Psi_1^{-1}$ to deduce from~\eqref{eq:contradiction-convergence} that
\begin{equation}\label{eq:converge-in-chart}
\lim_{k \to \infty}\|\Psi_1^{-1}(Q_k(t_k)) - \Psi_1^{-1}(u)\| = 0.
\end{equation}

Now~\eqref{eq:contradiction-min-max-energy}, the monotonicity property (b2) in Step 4, the equality~\eqref{eq:volume-tilde-agree}, plus the fact that $E_{H, \ep}(\widetilde{\gamma}_k(t), f_{\widetilde{\gamma}_k, t}) \leq \omega_{H, \ep} + \alpha_k$ imply that, in the notation of Step 4, 
\[
\lim_{k \to \infty}E^{\cA_1}_{H, \ep}(Q_k(0))  - E^{\cA_1}_{H, \ep}(Q_k(t_k))  = 0.
\]
Hence by property (b3), we have
\[
\|\Psi_{1}^{-1}(Q_k(t_k)) - \Psi_{1}^{-1}(Q_k(0))\| \to 0 \text{ as }k \to \infty.
\]
Combining this with~\eqref{eq:converge-in-chart} and using the continuity of $\Psi_1$, we infer that 
\[
\lim_{k \to \infty}\| Q_{k}(0)  - u\|_{2, 2} = 0.
\]
Since $Q_{k}(0) = \gamma_k(a_{k, n_k}) \notin \cN_{\theta/2}$ by~\eqref{eq:interval-choice}, this is a contradiction. Hence conclusion (b) of Lemma~\ref{lemm:deformation} is also verified.
\end{proof}

We may now improve Proposition~\ref{prop:critical-uniform-bound} to obtain non-trivial critical points with Morse index at most one.
\begin{prop}\label{prop:critical-uniform-bound-index}
Under the hypotheses of Proposition~\ref{prop:critical-uniform-bound}, that is, suppose $H > 0$, $\ep \in (0, 1]$, and that $\gamma_k \in \cP_{\alpha_k, C_0}$ for each $k$, with $\alpha_k \to 0$, then there exists $u \in \cK_{C_0 + 1}$ which satisfies $D_{\ep}(u) \geq \eta_1(\ep, H)$ and $\Ind_{H, \ep}(u) \leq 1$.
\end{prop}

\begin{proof}
Define 
\[
\cK' = \{v \in \cK_{C_0 + 1}\ |\ D_{\ep}(v) \geq \eta_1(\ep, H)\}.\]
Then $\cK'$ is clearly a compact subset of $\cK_{C_0 + 1}$, and is non-empty by Proposition~\ref{prop:critical-uniform-bound}. Now assume by contradiction that $\Ind_{H, \ep}(v) \geq 2$ for all $v \in \cK'$, and let $\{\widetilde{\gamma}_k\}$ be the sequence of sweepouts obtained by applying Lemma~\ref{lemm:deformation} to $\{\gamma_k\}$ with $\cK'$ as above. Then by Lemma~\ref{lemm:deformation}(a), followed by Proposition~\ref{prop:critical-uniform-bound} applied to $\{\widetilde{\gamma}_k\}$, we obtain a subsequence of $(\widetilde{\gamma}_k, \alpha_k)$, which we do not relabel, and a sequence $t_k \in [0, 1]$, such that 
\[
E_{H, \ep}(\widetilde{\gamma}_k(t_k), f_{\widetilde{\gamma}_k, t_k}) \geq \omega_{H, \ep} - \alpha_k,
\]
and that
\[
\lim_{k \to \infty}\| \widetilde{\gamma}_k(t_k) - u \|_{2, 2}  = 0 \text{ for some }u \in \cK'.
\]
However this is in contradiction with Lemma~\ref{lemm:deformation}(b), and hence we conclude that $\cK'$ must contain an element $u$ with $\Ind_{H, \ep}(u) \leq 1$. The proof of Proposition~\ref{prop:critical-uniform-bound-index} is complete.
\end{proof}

\begin{coro}\label{coro:critical-uniform-bound}
For almost every $H \in \RR_{+}$, there exists a constant $C_0 > 0$ and a sequence $\ep_j \to 0$, such that for each $j \in \NN$ there exists $u_j \in W^{2, 2}(S^2; S^3)$ with the following properties:
\vskip 1mm
\begin{enumerate}
\item[(a)] $\delta E_{H, \ep_j}(u_j) = 0$.
\vskip 1mm
\item[(b)] $\eta_1(\ep_j, H) \leq D_{\ep_{j}}(u_j) \leq C_0 + 1$.
\vskip 1mm
\item[(c)] $\Ind_{H, \ep_j}(u_j) \leq 1$.
\end{enumerate}
\end{coro} 

\begin{proof}
We first pick an arbitrary sequence $\ep_j \to 0$. Then Proposition~\ref{prop:min-max-value} and Lemma~\ref{lem:existence of good sweepouts} imply that for almost every $H \in \RR_{+}$, there exists $c > 0$ and a subsequence of $\ep_j$, which we do not relabel, such that for each fixed $j$, there exists a sweepout $\gamma_k \in \cP_{H/k, 8H^2c}$ for sufficiently large $k$. (Here of course we are suppressing the $j$-dependence of $\gamma_k$ and $\cP_{H/k, 8H^2c}$.) Proposition~\ref{prop:critical-uniform-bound-index} applied with $C_0 = 8H^2c$ and $\alpha_k = H/k$ now gives a critical point $u_j$ satisfying all three of the asserted properties. 
\end{proof}

\section{Convergence to a constant mean curvature $S^2$}\label{S:convergence}

In this section, we study the convergence of the non-trivial critical points of the perturbed functionals found in the previous section, and conclude the proof of Theorem~\ref{thm:main1}. We start by a ``small energy $\Rightarrow$ regularity'' result, which is usually called an $\epsilon$-regularity theorem in other contexts.

\begin{prop}[See also \cite{Lamm}, Theorem 2.9]
\label{prop:eta-regularity}
Let $r_0 < \text{inj}(S^2)/4$. Given $H_0 > 0$, there exists some $\eta_0 \in (0, 1)$ and $C > 0$ depending on $H_0$, such that if $\ep \leq  r \leq  r_0$, $0 \leq H \leq H_0$ and $u$ is a solution to $\delta E_{H, \ep}(u) = 0$ satisfying
\begin{equation}\label{eq:small-energy}
\int_{B_{4r}(x)}\ep^2 |\Delta u|^2 + |\nabla u|^2 \leq \eta_0.
\end{equation}
Then 
\[
\int_{B_{2r}(x)} |\nabla^2 u|^2 + |\nabla u|^4 \leq Cr^{-2}\int_{B_{4r}(x)} \ep^2 |\Delta u|^2 + |\nabla u|^2, 
\]
\[
\int_{B_r(x)} |\nabla^3 u|^2 + |\nabla^2 u|^3 + |\nabla u|^6 \leq Cr^{-4}\int_{B_{4r}(x)} \ep^2 |\Delta u|^2 + |\nabla u|^2. 
\]
Note that the left-hand sides contain no factors of $\ep$.
\end{prop}
\begin{proof} This is essentially a consequence of~\cite[Theorem 2.9]{Lamm}, whose proof applies equally well to solutions of the equation $\delta E_{H, \ep}(u) = 0$. For the reader's convenience, we outline the first two steps of the inductive argument in~\cite[Section 2]{Lamm} leading to Theorem 2.9 there, which suffice for the estimate asserted above. The reader who wishes to see the details may consult~\cite{Lamm}.

First note that by~\cite[Corollary 2.5]{Lamm} and the assumption that $\ep/r \leq 1$, one has the following estimate, whose derivation we briefly recall in the Appendix.
\begin{align}
\int_{B_{3r}(x)} \ep^2 \big( |\nabla^2 u|^2 + |\nabla u|^4 \big) + |\nabla u|^2 & \leq C\Big(1 + \int_{B_{4r}(x)}|\nabla u|^2 \Big)\Big(\int_{B_{4r}(x)} \ep^2 |\Delta u|^2 + |\nabla u|^2\Big) \nonumber\\
&\leq  C\eta_0. \label{eq:W22-L4}
\end{align}
Here we used the assumption~\eqref{eq:small-energy} and the fact that $\eta_0 < 1$ to get the second line. The next step is to differentiate the equation satisfied by $u$, which we recall below:
\begin{equation}\label{eq:Euler-Lagrange-recall}
\ep^2\Delta^2u - \Delta u = f_1 + \ep^2 f_2 + \ep^2 \Div F,
\end{equation}
where 
\begin{align*}
f_1 &= -A(u)(\nabla u, \nabla u) - H \ast(u^{\ast}Q)\\ 
f_2 &= -\Delta(P \circ u)(\Delta u)\\
F&= \nabla \big( A(u)(\nabla u, \nabla u) \big) + 2\nabla(P \circ u)(\Delta u).
\end{align*}
Differentiating both sides of~\eqref{eq:Euler-Lagrange-recall}, and using the following relation to commute $\nabla$ with $\Delta$ and $\Delta^2$,
\begin{equation}\label{eq:Delta-nabla-commutator}
\Delta \nabla^k u = \nabla^k \Delta u + R\ast \nabla^k u + \cdots + \nabla^{k - 1}R \ast \nabla u,
\end{equation}
we find that 
\begin{align}\label{eq:EL-diff-1}
 \ep^2 \Delta^2 \nabla u - \Delta \nabla u =&\   \nabla \big( f_1 + \ep^2 f_2 + \ep^2 \Div(F) \big) \nonumber \\
 &\ + \ep^2 \big( R\ast \nabla^3 u + \Delta (R \ast \nabla u))+ R\ast \nabla u.
\end{align}
Note that, when applying the covariant derivative, we are viewing both sides of~\eqref{eq:Euler-Lagrange-recall} as maps  from $S^2$ into $\RR^N$ and differentiating component-wise. In particular, commuting derivatives only produces curvatures of $S^2$, but not of $S^{3}$.

Next, continue following~\cite{Lamm}, we choose a cut-off function $\zeta$ so that 
\[
\zeta = 1 \text{ on }B_{2r}(x),\ \zeta = 0 \text{ outside }B_{3r}(x) \text{ and }|\nabla^k \zeta| \leq Cr^{-k},
\]
with $C$ independent of $x$ and $r \leq r_0$, and test~\eqref{eq:EL-diff-1} against $\psi = \zeta^6 \nabla u$. That is, consider
\begin{align*}
\int_{B_{3r}(x)} \ep^2\langle \Delta \nabla u \cdot \Delta \psi \rangle + \langle \nabla^2 u, \nabla \psi  \rangle =&\ \int_{B_{3r}(x)} \langle \nabla f_1 + \ep^2 \nabla f_2, \psi \rangle - \ep^2 \Div (F) \cdot \Div (\psi)\\
&\ + \int_{B_{3r}(x)} \big(  \ep^2 \big( R\ast \nabla^3 u + \Delta (R \ast \nabla u) \big)+ R\ast \nabla u \big) \cdot \psi.
\end{align*}
Using the definitions of $\psi, f_1, f_2$ and $F$, the assumption that $H \in [0, H_0]$, as well as the properties of $\zeta$, one finds after several applications of Young's inequality that (below $\delta_1> 0$ is to be chosen later)
\begin{align}
&\ \int_{B_{3r}(x)} \zeta^6 \ep^2 |\Delta \nabla u|^2 + \zeta^6 |\nabla^2 u|^2 \nonumber\\
\leq &\ \frac{\delta_1}{2} \int_{B_{3r}(x)} \zeta^6\ep^2 |\nabla^3 u|^2 + \zeta^6 |\nabla^2 u|^2 \nonumber\\
&\ + C\int_{B_{3r}(x)} \ep^2\big( \zeta^6 |\nabla u|^6 +\zeta^6 |\nabla^2 u|^3 \big) + \zeta^6|\nabla u|^4 \nonumber\\
&\ + C\int_{B_{3r}(x)}\ep^2\big( r^{-2}|\nabla^2 u|^2 + r^{-2}|\nabla u|^4 + r^{-4}|\nabla u|^2  \big) + r^{-2}|\nabla u|^2 \nonumber\\
=&\ \frac{\delta_1}{2} I_1 + CI_2 + CI_3 \label{eq:1st-iteration-step1}
\end{align}
The integrand $\zeta^6 |\Delta \nabla u|^2$ on the left-hand side can be replaced by $\zeta^6 |\nabla^3 u|^2$, up to introducing lower derivatives of $u$ on the right-hand side. Indeed, recall that for any smooth tensor $T$ on $S^2$ we have
\[
\Delta \nabla T = \nabla \Delta T + \nabla R \ast T + R \ast \nabla T.
\]
Multiplying this by $\nabla T$ and integrating by parts give
\begin{equation}\label{eq:tensor-commute}
\int_{S^2}|\nabla^2 T|^2 \leq \int_{S^2} |\Delta T|^2 + C|\nabla T|^2 + C|T|^2.
\end{equation}
Applying this to $T = \nabla (\zeta^3 (u - a))$, where $a = \fint_{B_{3r}(x)}u$, and performing some routine calculations (see~\eqref{eq:commute-cutoff-1} and~\eqref{eq:commute-cutoff-2} in the proof of~\eqref{eq:W22-L4}), we find that (assuming $r_0 \leq 1$) 
\[
\int_{B_{3r}(x)} \zeta^6 |\nabla^3 u|^2 \leq 2\int_{B_{3r}(x)}\zeta^6 |\Delta \nabla u|^2 + C\int_{B_{3r}(x)}r^{-2}|\nabla^2 u|^2 + r^{-4}|\nabla u|^2.
\]
Using this in~\eqref{eq:1st-iteration-step1}, we find that
\begin{equation}\label{eq:1st-iteration-step2}
I_1 \leq \delta_1 I_1  + C_{\delta_1}I_2 + C_{\delta_1}I_3.
\end{equation}

The next step is to bound $I_2$ in terms of $I_3$ plus a small multiple of $I_1$. This is where the assumption~\eqref{eq:small-energy} comes in. The main ingredient for this step is the following Sobolev inequality for compactly supported functions:
\begin{equation}\label{eq:Sobolev}
\int_{S^2}|h|^2 \leq C\Big( \int_{S^2}|\nabla h| \Big)^{2}, \text{ for all }h \in W^{1, 1}_{0}(B_{r_0}(x)),
\end{equation}
valid since we are on a two-dimensional domain. By Corollary 2.5 in~\cite{Lamm}, as well as the estimates in p.134-135 there, and recalling~\eqref{eq:W22-L4}, one gets
\begin{align}
\int_{B_{3r}(x)} \zeta^6 |\nabla u|^4 \leq&\ C\big( \int_{B_{3r}(x)} |\nabla u|^2 \big)\Big( \int_{B_{3r}(x)}r^{-2}|\nabla u|^2 + \zeta^6 |\nabla^2 u|^2 \Big)\nonumber \\
\leq&\ C\eta_0 I_{3} + C\eta_0 I_1.\\
 \int_{B_{3r}(x)} \ep^2\zeta^6  |\nabla^2 u|^3 \leq&\ \delta_2 \int_{B_{3r}(x)}\zeta^6 |\nabla^2 u|^2\nonumber \\
&\ + C_{\delta_2}\big( \ep^2\int_{B_{3r}(x)}|\nabla^2 u|^2 \big) \Big(\int_{B_{3r}(x)} \ep^2 r^{-2}|\nabla^2 u|^2 + \zeta^6 \ep^2|\nabla^3 u|^2 \Big)\nonumber\\
\leq&\ (\delta_2 + C_{\delta_2}\eta_0)I_1 + C_{\delta_2}\eta_0 I_3.\\
\int_{B_{3r}(x)} \ep^2\zeta^6 |\nabla u|^6 \leq&\ C\big( \int_{B_{3r}(x)}|\nabla u|^2 \big)\big( \int_{B_{3r}(x)}\ep^2 r^{-2}|\nabla u|^4 \big)\nonumber \\
&\ + C\big( \int_{B_{3r}(x)}\ep^2 |\nabla u|^4 \big)\big( \int_{B_{3r}(x)} \zeta^6 |\nabla^2 u|^2 \big)\nonumber\\
\leq&\ C\eta_0 I_3 + C\eta_0 I_1.
\end{align}
Adding up the three inequalities above gives
\begin{equation}\label{eq:1st-iteration-step3}
I_2 \leq (\delta_2 + C_{\delta_2}\eta_0)I_1 + C_{\delta_2}I_3.
\end{equation} 
Combining this with~\eqref{eq:1st-iteration-step2} yields
\begin{equation}\label{eq:1st-iteration-step4}
I_1 \leq (\delta_1 + C_{\delta_1}\delta_2 + C_{\delta_1, \delta_2}\eta_0)I_{1} + C_{\delta_1, \delta_2}I_3.
\end{equation}
Hence, by successively choosing $\delta_1, \delta_2$ and $\eta_0$ small enough, we arrive at $I_{1} \leq CI_3$, and hence $I_1+  I_2 \leq CI_3$, by~\eqref{eq:1st-iteration-step3}. In turn, this implies 
\begin{align}\label{eq:W32-L6}
&\ \int_{B_{2r}} \ep^2 \big( |\nabla^3 u|^2 + |\nabla^2 u|^3 + |\nabla u|^6 \big) + |\nabla^2 u|^2 + |\nabla u|^4 \nonumber \\
\leq&\ Cr^{-2}\int_{B_{3r}} \ep^2\big( |\nabla^2 u|^2 + |\nabla u|^4\big) + |\nabla u|^2 \nonumber\\
\leq&\  Cr^{-2}\int_{B_{4r}} \ep^2 |\Delta u|^2 + |\nabla u|^2 \leq Cr^{-2}\eta_0,
\end{align}
where~\eqref{eq:W22-L4} is used to get the last line. This concludes one iteration of the inductive argument and gives the first assertion of Proposition~\ref{prop:eta-regularity}.

For the second inductive step we differentiate the equation~\eqref{eq:Euler-Lagrange-recall} twice, commute the two derivatives with $\Delta$ and $\Delta^2$ using~\eqref{eq:Delta-nabla-commutator}, and test the resulting equation against $\psi = \zeta^8 \nabla^2 u$, this time with 
\[
\zeta = 1 \text{ on }B_{r}(x) \text{ and }\zeta = 0 \text{ outside }B_{2r}(x).
\]
After applying Young's inequality several times and using~\eqref{eq:tensor-commute} to replace $\int \zeta^8 |\Delta \nabla^2 u|^2$ by $\int \zeta^8 |\nabla^4 u|^2$ up to lower powers of $u$, one arrives at
\begin{equation}\label{eq:2nd-iteration-step1}
J_1 \leq \delta_3 J_1 +  C_{\delta_3}J_2 + C_{\delta_3}J_3,
\end{equation}
where $J_1, J_2$ and $J_3$ have the following meanings:
\begin{align*}
J_1 =& \int_{B_{2r}(x)}\zeta^8 \big( \ep^2 |\nabla^4 u|^2 + |\nabla^3 u|^2 \big), \\
J_2 = & \ \int_{B_{2r}(x)} \ep^2 \zeta^8 \big( |\nabla u|^8 + |\nabla^2 u|^4 + |\nabla^3 u|^{8/3} \big) + \zeta^8 \big( |\nabla u|^6  + |\nabla^2 u|^3 \big),\\
J_3 =& \int_{B_{2r}(x)} \ep^2 r^{-2}\big( |\nabla^3 u|^2 + |\nabla^2 u|^3 + |\nabla u|^6  \big) + r^{-2}\big(  |\nabla^2 u|^2 + |\nabla u|^4 \big) \\
&\ + \int_{B_{2r}(x)} \ep^2 r^{-4}\big( |\nabla^2 u|^2 + |\nabla u|^4 \big) + \ep^2 r^{-6}|\nabla u|^2 + r^{-4}|\nabla u|^2.
\end{align*}
With $\eta_0$ so small that~\eqref{eq:W32-L6} holds, each term in the integral $J_2$ can be estimated as in~\cite[p.134-135]{Lamm} (or using Lemma 2.7, Lemma 2.8 of the same paper) to yield
\begin{equation}\label{eq:2nd-iteration-step2}
J_2 \leq (\delta_4 + C_{\delta_4}\eta_0)J_1 + C_{\delta_4}J_3,
\end{equation}
which is an interpolation inequality analogous to~\eqref{eq:1st-iteration-step3}. For example, the term in $J_2$ involving $|\nabla^3 u|^{8/3}$ is treated as follows: first, by Young's inequality
\begin{equation}\label{eq:J2-hard-term}
\ep^2\int_{B_{2r}}\zeta^8|\nabla^{3}u|^{8/3} \leq \delta \int_{B_{2r}(x)}\zeta^8 |\nabla^3 u|^2 + C_{\delta}\ep^6 \int_{B_{2r}(x)} \zeta^8 |\nabla^3 u|^4.
\end{equation}
The second integral is then estimated using~\eqref{eq:Sobolev} and H\"older's inequality.
\begin{align*}
\int_{B_{2r}(x)}\zeta^8 |\nabla^3 u|^4 \leq&\ C\Big( \int_{B_{2r}(x)}\Big|\nabla (\zeta^4 |\nabla^3 u|^2) \Big|  \Big)^2\\
\leq&\ C\Big( \int_{B_{2r}(x)}\zeta^3 |\nabla \zeta||\nabla^3 u|^2 \Big)^2 + C\Big( \int_{B_{2r}(x)}\zeta^4 |\nabla^3 u||\nabla^4 u| \Big)^2\\
\leq&\ C\Big( \int_{B_{2r}(x)} |\nabla^3 u|^2 \Big)\Big( \int_{B_{2r}(x)}r^{-2}|\nabla^3 u|^2 \Big)\\
&\ + C\Big( \int_{B_{2r}(x)} |\nabla^3 u|^2 \Big)\Big( \int_{B_{2r}(x)} \zeta^8 |\nabla^4 u|^2 \Big)\\
\leq&\ C\ep^{-4} r^{-2}\eta_0 \Big( \int_{B_{2r}(x)} \ep^2 r^{-2}|\nabla^3 u|^2 + \ep^2 \zeta^8|\nabla^4 u|^2 \Big),
\end{align*}
where~\eqref{eq:W32-L6} is used to get the last line. Putting this back into~\eqref{eq:J2-hard-term} and recalling that $\ep^2 r^{-2} \leq 1$ yields 
\[
\ep^2 \int_{B_{2r}(x)} \zeta^8 |\nabla^{8/3}u|^3 \leq (\delta + C_{\delta}\eta_0)J_1 + C_{\delta}J_3.
\]
Once~\eqref{eq:2nd-iteration-step2} and~\eqref{eq:2nd-iteration-step1} are proven, we successively choose $\delta_3$, $\delta_4$ and then $\eta_0$ small to deduce that $J_1 + J_2 \leq CJ_3$. Using~\eqref{eq:W22-L4} and~\eqref{eq:W32-L6}, we see that
\begin{equation*}
J_3 \leq Cr^{-4}\int_{B_{4r}} \ep^2|\Delta u|^2 + |\nabla u|^2,
\end{equation*}
which finishes the proof.
\end{proof}

The next result consists of a uniform energy lower bound for non-trivial critical points of the perturbed functional, as promised in Remark~\ref{rmk:lower-bound-promise}.

\begin{prop}\label{prop:lower-bound}
Given $H_0 > 0$, let $r_0, \eta_0$ be as in Proposition~\ref{prop:eta-regularity}. Then there exists $\beta \in (0, \eta_0)$ depending only on $H_0$ such that if $0 < \ep \leq r_0$, $0 \leq H \leq H_0$ and if $u\in W^{2, 2}(S^2; S^3)$ is a non-constant solution to $\delta E_{H, \ep}(u) =0$, then $D_{\ep}(u) \geq \beta$.
\end{prop}
\begin{proof}
We prove the contrapositive. Namely, we will find a small enough constant $\beta \in (0, \eta_0)$ such that if $\delta E_{H, \ep}(u) = 0$ with $H \in [0, H_0]$, $\ep \in (0, r_0]$, and if $D_{\ep}(u) < \beta$, then $u$ must be constant. To that end, take a finite sub-collection of $\{B_{r_{0}}(x)\}_{x \in S^2}$ that covers $S^2$. Then since $\ep \leq  r_0$ and $D_{\ep}(u) \leq \beta <  \eta_0$, we may apply Proposition~\ref{prop:eta-regularity} to each $B_{4r_{0}}(x)$ in the finite subcover and add up the result to deduce that 
\[
\int_{S^2} |\nabla^2 u|^2 + |\nabla u|^4 \leq C N_0 r_{0}^{-2}\beta,
\]
where $N_0$ is the number of balls in the finite subcovering. By Poincar\'e inequality this implies the existence of some $a \in \RR^N$ such that 
\begin{equation}\label{eq:osc-small}
\|u - a\|_{\infty; S^2} \leq C\beta^{1/2}.
\end{equation}
Note that we've absorbed the factors  {$N_0$, }$r_{0}^{-2}$ into the constant $C$ since they are fixed once and for all. We now substitute $\psi = u - a$ into the equation~\eqref{eq:Euler-Lagrange-weak}, integrate by parts, and apply Young's inequality to get
\begin{align}
\int_{S^2}\ep^2 |\Delta u|^2 + |\nabla u|^2 \leq &\  C\ep^2 \int_{S^2} |\nabla u|^2 |\Delta u| + \big(|\Delta u|^2  + |\nabla u|^2|\Delta u|\big)|u - a| \nonumber\\
&\ +C(1 + H) \int_{S^2} |\nabla u|^2 |u - a|  \nonumber\\
\leq&\ \int_{S^2}\delta\ep^2 |\Delta u|^2 +C_{\delta}\ep^2 |\nabla u|^4 \nonumber \\
&\ + C\beta^{1/2}\int_{S^2}\ep^2 |\Delta u|^2 + \ep^2|\nabla u|^4 + (1 + H) |\nabla u|^2 \nonumber\\
\leq&\ (\delta + C_1\beta^{1/2})\int_{S^2}\ep^2|\Delta u|^2 + |\nabla u|^2 + (C_{\delta} + C_2\beta^{1/2})\int_{S^2}\ep^2 |\nabla u|^4. \label{eq:global-lower-bound-step1}
\end{align}
Note that we used the assumption $H \in [0, H_0]$ and absorbed $H_0$ into the constant $C_1$ in the last line. Taking $\delta$ and $\beta$ sufficiently small so that $\delta + C_1\beta^{1/2} < 1/2$, we obtain
\begin{equation}\label{eq:global-bound-by-L4}
\int_{S^2} \ep^2 |\Delta u|^2 + |\nabla u|^2 \leq 2(C_{\delta} + C_2\beta^{1/2})\int_{S^2}\ep^2 |\nabla u|^4.
\end{equation}
To estimate the right-hand side, we write it as $|\nabla u|^2 \langle \nabla u, \nabla(u - a) \rangle$ and integrate by parts to get
\begin{align}\label{eq:global-L4-bound}
\int_{S^2}|\nabla u|^4 \leq&\ C\int_{S^2} |u - a||\nabla u|^2 |\nabla^2 u| \leq C_3\beta^{1/2}\int_{S^2}|\nabla u|^4 + |\nabla^2 u|^2.
\end{align}
Further requiring that $C_3\beta^{1/2} < 1/2$, we deduce that 
\begin{equation}\label{eq:bound-L4-by-W22}
\int_{S^2} |\nabla u|^4 \leq 2C_3\beta^{1/2} \int_{S^2}|\nabla^2 u|^2.
\end{equation}
Next we use Lemma~\ref{lemm:D-equivalence} to link the term $|\Delta u|^2$ on the left-hand side of ~\eqref{eq:global-bound-by-L4} with $|\nabla^2 u|^2$. This gives
\[
\int_{S^2} \ep^2|\nabla^2 u|^2 \leq \int_{S^2}\ep^2|\Delta u|^2 + A_0\ep^2|\nabla u|^2 \leq (1 + A_0)\int_{S^2}\ep^2|\Delta u|^2 + |\nabla u|^2,
\]
where for the last inequality we used the assumption that $\ep \leq r_0 < 1$. Combining this with~\eqref{eq:global-bound-by-L4} and~\eqref{eq:bound-L4-by-W22} gives
\begin{equation}
\int_{S^2} \ep^2 |\nabla^2 u|^2 \leq 4(1 + A_0)(C_{\delta} + C_2\beta^{1/2})C_3\beta^{1/2}\int_{S^2} \ep^2|\nabla^2 u|^2.
\end{equation}
Decreasing $\beta$ further if necessary, we deduce from the above that $\nabla^2 u$ vanishes identically, which implies by~\eqref{eq:bound-L4-by-W22} that $u$ is constant. The proof is complete.
\end{proof}

A second consequence of Proposition~\ref{prop:eta-regularity}, as is well-known, is strong subsequential convergence away from energy concentration points. This is the content of the next result.

\begin{prop}\label{prop:convergence-mod-bubble}
Let $H$ be in the full-measure set of values yielded by Corollary~\ref{coro:critical-uniform-bound}, and let $C_0, \ep_j$ and $u_j$ be as in the conclusion of Corollary~\ref{coro:critical-uniform-bound}. There exist a subsequence of $u_j$, which we do not relabel, a finite number of points $p_1, \cdots, p_L \in S^2$ and a map $u \in W^{2, 2}_{\loc}(S^2 \setminus \{p_1, \cdots, p_L\}; S^3)$ such that 
\[
u_j \to u \text{ strongly in }W^{2, 2}(K; \RR^N) \text{ for any compact set }K \cap \{p_1, \cdots, p_L\} = \emptyset.
\]
Moreover, $u$ extends to a solution to~\eqref{equ:CMC equation1} from $S^2$ into $S^3$, and is smooth.
\end{prop}
\begin{proof}
Thanks to Corollary~\ref{coro:critical-uniform-bound}(b), Proposition~\ref{prop:eta-regularity} and an argument which is by now standard, we obtain a subsequence of $u_j$, which we do not relabel, a finite number of points $p_1, \cdots, p_L \in S^2$, and some $u \in W^{2, 2}_{\loc}(S^2 \setminus \{p_1, \cdots, p_L\})$ such that 
\[
u_j \to u \text{ strongly in }W^{2, 2}(K; \RR^N) \text{ for any compact set }K \cap \{p_1, \cdots, p_L\} = \emptyset,
\]
and that for any $1 \leq i \leq L$ and $t > 0$, there holds
\begin{equation}\label{eq:concentration}
\liminf_{j \to \infty}\int_{B_t(p_i)}\ep_j^2|\Delta u_j|^2 + |\nabla u_j|^2 \geq \eta_0/2,
\end{equation}
where $\eta_0$ is the constant from Proposition~\ref{prop:eta-regularity}, while at any $x \notin \{p_1, \cdots, p_L\}$, and for any subsequence, there holds 
\begin{equation}\label{eq:no-concentration}
\inf_{t > 0}\big(\liminf_{k \to \infty}\int_{B_t(x)}\ep_{j_k}^2|\Delta u_{j_k}|^2 + |\nabla u_{j_k}|^2\big) = 0.
\end{equation}

For any $\psi \in C^2_{c}(S^2\setminus \{p_1, \cdots, p_L\}; \RR^N)$, the strong $W^{2,2}$-convergence of $u_j$ and the fact that $\ep_j \to 0$ allow us to pass to the limit in the equation $G_{H, \ep_j}(u_j)(\psi) = 0$
 to get 
 \[
 \int_{S^2}\langle \nabla u, \nabla \psi \rangle + A(u)(\nabla u, \nabla u)\cdot \psi + H \int_{S^2}\psi \cdot \ast (u^\ast Q) = 0.
 \]
This implies that $u$ is a weak solution to \eqref{equ:CMC equation1} away from $p_1, \cdots, p_L$. Since $u$ lies in $W^{2, 2}_{\loc}(S^2 \setminus \{p_1, \cdots , p_L\})$, standard elliptic theory implies that $u$ is in fact smooth away from $p_1, \cdots, p_L$.

On the other hand, again by Corollary~\ref{coro:critical-uniform-bound}(b), we see that
\[
\int_{K}|\nabla u|^2  = \lim_{j \to\infty}\int_{K}|\nabla u_j|^2 \leq C_0 + 1, 
\] 
for any compact set $K$ disjoint from $\{p_1, \cdots, p_L\}$. Thus, $u$ is a weak solution to~\eqref{equ:CMC equation1} on $S^2$, and we may apply Proposition~\ref{prop:removable-singularity} to conclude that $u$ is in fact smooth on all of $S^2$.
\end{proof}

In the case that $\{p_1, \cdots, p_L\}$ is non-empty, we fix some $\rho_0 < \frac{1}{4}\text{inj}(S^2)$ with 
\[
4\rho_0 < d(p_i, p_j) \text{ for all }i \neq j.
\]
To examine the behavior of $u_j$ near, say, $p_1$, we introduce, as in~\cite{Lamm}, the function
\[
Q_j(t) = \max_{x \in \overline{B_{2\rho_0}(p_1)}}\int_{B_t(x)} \ep_j^{2}|\Delta u_j|^2 + |\nabla u_j|^2
\]
to help us find suitable rates at which to perform the rescaling. The following result ensures that the perturbation parameters, after rescaling, still converge to zero.

\begin{prop}[See also \cite{Lamm}, Lemma 3.1]\label{prop:prepare-blowup} 
There exist a subsequence $(u_{j_k}, \ep_{j_k})$ of $(u_j, \ep_j)$, along with a sequence of radii $t_k \to 0$ and a sequence of points $x_k \to p_1$ such that 
\begin{enumerate}
\item[(a)] $Q_{j_k}(t_k) = \int_{B_{t_k}(x_k)}\ep_{j_k}^2|\Delta u_{j_k}|^2 + |\nabla u_{j_k}|^2 = \eta_0/3$.
\vskip 1mm
\item[(b)] $\liminf_{k \to \infty} \ep_{j_k} /t_k =  0$.
\end{enumerate}
\end{prop}
\begin{proof}
For part (a), we note that $Q_j$ is continuous and $Q_j(0) = 0$. Moreover, by~\eqref{eq:concentration} we have, for all $k \in \NN$, that 
\[
\liminf_{j \to \infty}Q_j(1/k) \geq \eta_0/2 > \eta_0/3.
\]
This allows us to choose a subsequence $u_{j_k}$ and sequences $t_k \in [0, 1/k]$, $x_k \in \overline{B_{2\rho_0}(p_1)}$, such that
\[
Q_{j_k}(t_k) = \int_{B_{t_{k}}(x_k)}\ep^2_{j_k} |\Delta u_{j_k}|^2 + |\nabla u_{j_k}|^2 = \eta_0/3.
\]
The proof of part (a) is complete upon noting that $x_k$ must converge to $p_1$, for otherwise by the compactness of $\overline{B_{2\rho_0}(p_1)}$ and our choice $\rho_0$, upon passing to a further subsequence if necessary, there would be a point $p \notin \{p_1, \cdots, p_L\}$ where 
\[
\inf_{t > 0}\big(\liminf_{k \to \infty}\int_{B_t(p)}\ep_{j_k}^2|\Delta u_{j_k}|^2 + |\nabla u_{j_k}|^2\big) > 0,
\]
which contradicts~\eqref{eq:no-concentration}.

For part (b), by slight abuse of notation we denote the subsequence $(u_{j_k}, \ep_{j_k})$ still by $(u_j, \ep_j)$, and write $t_j, x_j$ for the sequences $t_k, x_k$. Now, assume by contradiction that there exists $\alpha > 0$ such that 
\[\ep_{j}/\alpha \geq t_j \text{ for all sufficiently large $j$.}\]
We are now going to rescale the maps $u_j$. To that end, we use the exponential map to identify $B_{2\rho_0}(p_1)$ with the ball $B_{2\rho_0} = B_{2\rho_0}(0) \subset \RR^2$ equipped with a Riemannian metric $g$, and then define
\[
v_j(y) = u_{j}(\tau_j(y)) := u_j(x_j + \ep_j y).
\]
Then it's not hard to see that $v_j$ has the following properties. (Below, the operators $\Delta, \nabla, \Div$ and $\ast$ are with respect to the metric $g_j = \ep_j^{-2}\tau_j^{\ast}g$ when applied to $v_j$, and with respect to $g$ when applied to $u_j$.)
\begin{enumerate}
\item[(i)] For $j$ large enough, $v_j$ satisfies the following system on $B_{\rho_0/\ep_j}$: 
\begin{align*}
&\Delta^2 v_j - \Delta \big( A(v_j)(\nabla v_j, \nabla v_j) \big) -2 \Div \big( \nabla(P\circ v_j )(\Delta v_j) \big) + \Delta (P\circ v_j)(\Delta v_j) \\
&- \Delta v_j + A(v_j)(\nabla v_j, \nabla v_j) + H \ast (v_j^{\ast}Q) = 0.
\end{align*}

\vskip 1mm
\item[(ii)] For all $R > 0$ and $j$ large enough, 
\begin{align*}
\int_{B_R}| \nabla^2 v_j|^2 + |\nabla v_j|^2 &= \int_{B_{\ep_j R}(x_j)}\ep_{j}^2|\nabla^2 u_j|^2 + |\nabla u_j|^2 \leq A_0(C_0 + 1),
\end{align*}
where we used Lemma~\ref{lemm:D-equivalence}(b) and Corollary~\ref{coro:critical-uniform-bound}(b) to get the inequality.
\vskip 1mm
\item[(iii)] For $j$ large enough,
\[\int_{B_{\alpha^{-1}}} |\Delta v_j|^2 + |\nabla v_j|^2  \geq \int_{B_{t_j}(x_j)}\ep_j^2|\Delta u_j|^2 + |\nabla u_j|^2 = \eta_0/3.\]
\end{enumerate}
Now, the estimate in (ii) and the fact that the metrics $g_j$ converge to the Euclidean metric smoothly locally on $\RR^2$ as $j \to \infty$ imply that $v_j$ is uniformly bounded in $W^{2, 2}$ on compact subsets of $\RR^2$. Thus, we may use (i) and argue as in Proposition~\ref{prop:weak-solution-smooth} to obtain the estimates required to extract a subsequence of $v_j$, which we do not relabel, that converges smoothly locally on $\RR^2$ to some $v \in C^{\infty}(\RR^2; S^3)$. Passing to the limit in (i) and recalling the proof of Proposition~\ref{prop:first-var-properties}(a), we see that $v$ satisfies 
\begin{equation}\label{eq:blowup-limit}
P_v(\Delta^2 v) - P_v(\Delta v) + H \ast (v^{\ast}Q) = 0.
\end{equation}
Moreover, (ii) implies that 
\begin{equation}\label{eq:finite}
\int_{\RR^2}|\nabla^2 v|^2 + |\nabla v|^2 < \infty.
\end{equation}
We next show that~\eqref{eq:blowup-limit} and~\eqref{eq:finite} together imply that $v$ is constant, which contradicts (iii) and finishes the proof of (b). Indeed, from~\eqref{eq:blowup-limit} we have, as in~\cite[p.146]{Lamm}, that 
\begin{equation}\label{eq:orthogonality}
\big(\Delta^2 v(x) - \Delta v(x) \big)\perp \Span\{\partial_1 v(x), \partial_2 v(x) \} \text{ for all }x \in \RR^2.
\end{equation}
Thanks to this property and the bound~\eqref{eq:finite}, we can follow the computation in~\cite{Lamm} to see that in fact $v$ is harmonic, and hence must be constant by~\eqref{eq:finite}. 
\end{proof}
\begin{proof}[Proof of the harmonicity of $v$.]
For completeness, we include the details of the computation in~\cite[p.146]{Lamm} below leading to the harmonicity of $v$. To begin, take a cut-off function $\zeta \in C^{\infty}_{c}(\RR^2)$ with $\zeta = 1 \text{ on }B_{1} \text{ and }\zeta = 0 \text{ outside }B_{2}$, and write $\zeta_R = \zeta(\cdot/R)$. In particular, for $k \geq 1$, the support of $\nabla^k\zeta_R$ is contained in $B_{2R}\setminus B_{R}$, and $|\nabla^k\zeta_{R}| \leq CR^{-k}$ with $C$ independent of $R$. Next let $f(x) = \zeta_{R}(x)^2 x^{i}\partial_i v$, where $i$ is summed from $1$ to $2$. Then by~\eqref{eq:orthogonality}, we have
\begin{equation}\label{eq:orthogonality-integrated}
\int_{\RR^2} \big( \Delta^2 v - \Delta v \big)\cdot f = 0.
\end{equation}
We now compute
\begin{align}
-\int_{\RR^2}\Delta v \cdot f &= \int_{\RR^2} \langle \nabla v, \nabla f \rangle \nonumber\\
&= \int_{\RR^2}2 (\zeta_R \partial_j\zeta_R) (x^{i}\partial_i v) \cdot \partial_j v +\zeta_R^2 |\nabla v|^2 + \zeta_R^2 x^{i}\partial_i\Big( \frac{|\nabla v|^2}{2} \Big) \nonumber\\
&= \int_{\RR^2} 2 (\zeta_R \partial_j\zeta_R) (x^{i}\partial_i v) \cdot \partial_j v - \zeta_R( x^{i}\partial_i\zeta_R) |\nabla v|^2.
\end{align}
In particular, 
\begin{equation}\label{eq:Delta-computation}
\Big| \int_{\RR^2}\Delta v\cdot f \Big| \leq C\int_{B_{2R}\setminus B_{R}} |x||\nabla\zeta_R||\nabla v|^2.
\end{equation}
On the other hand, 
\begin{align*}
\int_{\RR^2} \Delta^2 v \cdot  f =&\  \int_{\RR^2} \Delta v \cdot \Delta f\\
= &\  \int_{\RR^2} 2\big( |\nabla\zeta_R|^2 + \zeta_R\Delta\zeta_R \big)x^i\partial_i v \cdot \Delta v + 4(\zeta_R\partial_j \zeta_R) \big( \delta_{ij} \partial_i v + x^{i}\partial_{ij}v \big)\cdot \Delta v\\
&\ + \int_{\RR^2}2\zeta_R^2  |\Delta v|^2  +\zeta_R^2 x^{i}\partial_i \Big( \frac{|\Delta v|^2}{2} \Big)\\
=&\ \int_{\RR^2} 2\big( |\nabla\zeta_R|^2 + \zeta_R\Delta\zeta_R \big)x^i\partial_i v \cdot \Delta v + 4(\zeta_R\partial_j \zeta_R) \big( \delta_{ij} \partial_i v + x^{i}\partial_{ij}v \big)\cdot \Delta v \nonumber\\
&\ + \int_{\RR^2} \zeta_R^2 |\Delta v|^2 - |\Delta v|^2x^{i}\zeta_R \partial_i\zeta_R,
\end{align*}
where we integrated by parts to get the last line. In particular
\begin{align}\label{eq:Delta-2-computation}
\int_{\RR^2}\zeta_R^2 |\Delta v|^2 \leq&\  \Big| \int_{\RR^2}\Delta^2 v \cdot f \Big| \nonumber\\
&\  + C\int_{B_{2R}\setminus B_{R}} \big(|\nabla\zeta_R|^2 + |\Delta \zeta_R|\big)|x||\Delta v||\nabla v|\nonumber\\
&\ + C\int_{B_{2R}\setminus B_{R}} |\nabla\zeta_R| |\nabla v||\Delta v| + |\nabla\zeta_R||x||\nabla^2 v|^2.
\end{align}
Recalling that $\Big| \int_{\RR^2}\Delta^2 v \cdot f \Big| = \Big| \int_{\RR^2}\Delta v\cdot f \Big|$ by~\eqref{eq:orthogonality-integrated}, and using the properties of $\zeta_R$, we finally get
\[
\int_{\RR^2}\zeta_{R}^2 |\Delta v|^2 \leq C\int_{B_{2R} \setminus B_{R}} |\nabla v|^2 + R^{-1}|\nabla v||\nabla^2 v| + |\nabla^2 v|^2.
\]
Using~\eqref{eq:finite} we see that the right-hand side converges to zero as $R \to \infty$, and hence $\int_{\RR^2}|\Delta v|^2 = 0$. Thus $v$ is harmonic, as asserted above.
\end{proof}

\subsection*{Proof of Theorem \ref{thm:main1}.} 

Let $H$ be one of the values given by Corollary~\ref{coro:critical-uniform-bound} and let $C_0$, $u_j, \ep_j$ be as in the conclusion of Corollary~\ref{coro:critical-uniform-bound}. Applying Proposition~\ref{prop:convergence-mod-bubble}, we get a set of points $\{p_1, \cdots, p_L\} \subset S^2$ and subsequences of $u_j$, $\ep_j$, which we do not relabel, with the properties asserted there. How the argument continues from here depends on whether $\{p_1, \cdots, p_L\}$ is empty or not.
\vskip 1mm
\noindent\textbf{Case 1:  $\{p_1, \cdots, p_L\} = \emptyset$.}\ 
\vskip 1mm
If the set $\{p_1, \cdots, p_L\}$ is empty, then by Proposition~\ref{prop:convergence-mod-bubble}, the convergence of $u_j$ to $u$ takes place in $W^{2, 2}(S^2; S^3)$. Recalling that each $u_{j}$ is non-constant, we see from Proposition~\ref{prop:lower-bound} applied with $H_0 = H$ that 
\[
\int_{S^2}\ep_j^2 |\Delta u_j|^2 + |\nabla u_j|^2 \geq \beta \text{ for all }j.
\]
The strong $W^{2, 2}$-convergence of $u_j$ on all of $S^2$ then implies that 
\[
\int_{S^2} |\nabla u|^2 \geq \beta.
\]
In particular, the solution $u$ yielded by Proposition~\ref{prop:convergence-mod-bubble} is non-constant. To finish the proof in this case we need to show that $\Ind_{H}(u) \leq 1$. To do that, it suffices to show that if we have linearly independent elements $\psi_1, \cdots, \psi_d \in \cT_u$ such that $\delta^2 E_{H}(u)$ restricted to their span is negative-definite, then $d \leq 1$. With $\psi_i$ as such, and letting $\psi_{i, j} = P_{u_j}(\psi_i)$, then each $\psi_{i, j} \in \cT_{u_j}$ and moreover $\psi_{i, j} \to \psi_i$ strongly in $W^{2, 2}(S^2; \RR^N)$ as $j \to \infty$ by Lemma~\ref{lemm:projection-estimates}(b) since $u_j \to u$ strongly in $W^{2, 2}(S^2; S^3)$. In particular, it's not hard to see from Proposition~\ref{prop:second-var-formula-intrinsic} that 
\[
\lim_{j \to \infty}\delta^2 E_{H, \ep_j}(u_j)(\psi_{i, j}, \psi_{k, j}) = \delta^2 E_{H}(u)(\psi_i, \psi_k) \text{ for all }i, k  = 1, \cdots, d.
\]
Thus for large enough $j$ the matrix $\big(\delta^2 E_{H, \ep_j}(u_j)(\psi_{i, j}, \psi_{k, j})\big)_{i, k = 1, \cdot, d}$ is negative-definite. Conclusion (c) of Corollary~\ref{coro:critical-uniform-bound} then forces $d \leq 1$ as desired, and we are done with Case 1.

\vskip 1mm
\noindent\textbf{Case 2:  $\{p_1, \cdots, p_L\} \neq \emptyset$.}\ 
\vskip 1mm
If $\{p_1, \cdots, p_L\}$ is non-empty, we choose $\rho_0$ as in the paragraph before Proposition~\ref{prop:prepare-blowup} and use the exponential map to identify $B_{2\rho_0}(p_1)$ with $B_{2\rho_0}(0) \subset \RR^2$, equipped with a metric $g$. Next we invoke Proposition~\ref{prop:prepare-blowup}, and by the same abuse of notation as in the proof of part (b) there, we write $u_{j}, \ep_j, t_j$ and $x_j$ for $u_{j_k}, \ep_{j_k}, t_k$ and $x_k$ respectively. By Proposition~\ref{prop:prepare-blowup}(b), passing to another subsequence if necessary, we may also assume that 
\[
\widetilde{\ep}_j := \ep_j/t_j \to 0.
\] 
We then rescale the maps $u_j$ by letting
\begin{equation}\label{eq:case2-blowup}
\widetilde{u}_{j}(y) = u_{j}(\sigma_j(y)) :=  u_{j}(x_j + t_j y),
\end{equation}
The new sequence $\widetilde{u}_j$ has the following properties:
\begin{enumerate}
\item[(i)] For sufficiently large $j$, the map $\widetilde{u}_{j}$ satisfies the following on $B_{\rho_0/t_j}$:
\begin{align*}
&\widetilde{\ep}_j^2\Big(\Delta^2 \widetilde{u}_{j} - \Delta \big( A(\widetilde{u}_{j})(\nabla \widetilde{u}_{j}, \nabla \widetilde{u}_{j}) \big) -2 \Div \big( \nabla(P\circ \widetilde{u}_{j} )(\Delta \widetilde{u}_{j}) \big) + \Delta (P\circ \widetilde{u}_{j})(\Delta \widetilde{u}_{j})\Big) \\
&- \Delta \widetilde{u}_{j} + A(\widetilde{u}_{j})(\nabla \widetilde{u}_{j}, \nabla \widetilde{u}_{j}) + H \ast (\widetilde{u}_{j}^{\ast}Q) = 0.
\end{align*}
\vskip 1mm
\item[(ii)] For all $j$ and $R < \rho_0/t_j$, 
\[
\int_{B_R}\widetilde{\ep}_j^2 |\Delta \widetilde{u}_j|^2 + |\nabla \widetilde{u}_j|^2 = \int_{B_{t_j R}(x_j)} \ep_j^2 |\Delta u_j|^2 + |\nabla u_j|^2 \leq C_0 + 1. 
\]
\vskip 1mm
\item[(iii)] By Proposition~\ref{prop:prepare-blowup}(a), for $j$ large enough and 
$|y|\leq \rho_0/t_j$, we have
\[
\int_{B_{1}(y)}\widetilde{\ep}_j^2 |\Delta \widetilde{u}_j|^2 + |\nabla \widetilde{u_j}|^2 = \int_{B_{t_j}(x_j + t_j y)} \ep_j^2 |\Delta u_j|^2 + |\nabla u_j|^2 \leq \eta_0/3.
\]
Moreover, equality holds when $y = 0$ by our choice of $x_j$.
\end{enumerate}
(Here and below, the volume form in the integral, the operators $\Delta, \nabla, \Div$ and $\ast$ are with respect to the metric $g_j := t_j^{-2}\sigma_j^{\ast}g$ when applied to $\widetilde{u}_j$, and with respect to $g$ when applied to $u_j$.)

By property (iii) and the fact that $\widetilde{\ep}_j \to 0$ as $j \to \infty$, we may apply Proposition~\ref{prop:eta-regularity} on each $B_{1}(y)$ when $j$ is large enough to get 
\[
\int_{B_{1/4}(y)} |\nabla^3 \widetilde{u}_j|^2 \leq C\int_{B_{1}(y)}\widetilde{\ep}_j^2|\Delta \widetilde{u}_j|^2 + |\nabla \widetilde{u}_j|^2 \leq \eta_0/3.
\]
This in turn yields a subsequence of $\widetilde{u}_j$, which we do not relabel, converging strongly in $W^{2, 2}$ on compact subsets of $\RR^2$ to a limit map $u$. Moreover, $u$ must be non-constant since 
\[
\int_{B_1(0)}|\nabla u|^2 = \lim_{j \to \infty} \int_{B_{1}(0)}\widetilde{\ep}_j^2|\Delta \widetilde{u}_j|^2 + |\nabla \widetilde{u}_j|^2 = \eta_0/3,
\]
where the last equality comes form property (iii). We then use the strong $W^{2, 2}$-convergence of $\widetilde{u}_j$ to $u$ and the fact that $g_j$ converges to the standard metric $g_{\RR^2}$ smoothly locally on $\RR^2$ to pass to the limit as $j \to \infty$ in property (i) above and deduce that $u$ is a weak solution to~\eqref{equ:CMC equation1} on $\RR^2$ lying in $W^{2,2}$, and hence is smooth by standard elliptic theory. 

To obtain from $u$ a solution to~\eqref{equ:CMC equation1} on $S^2$, we note that the strong $W^{2,2}$-convergence also allows us to pass to the limit as $j \to \infty$ in property (ii) above to deduce that $\int_{B_{R}} |\nabla u|^2 \leq  C_0+1$ for all $R >0$. Consequently
\[
\int_{\RR^2}|\nabla u|^2 < \infty.
\]
Using the stereographic projection $\sigma: S^2\setminus \{p\} \to \RR^2$ and the fact that both~\eqref{equ:CMC equation1} and the Dirichlet energy are conformally invariant, we see from the above energy bound and Proposition~\ref{prop:removable-singularity} that $v:= u\circ\sigma$ extends to a non-constant solution to the CMC equation~\eqref{equ:CMC equation1} on all of $S^{2}$, and it remains to show that $\Ind_{H}(v) \leq 1$.  As in the last part of Case 1, we take $\psi_1, \cdots, \psi_d \in \cT_v$ linearly independent such that $\delta^2 E_{H}(v)$ restricted to their span is negative-definite, and show that $d \leq 1$. By the classical logarithmic cut-off trick, we may assume that there exists some $R> 0$ such that $\supp(\psi_i) \subset \sigma^{-1}(B_{R})$ for all $i = 1, \cdots, d$. Letting $\widetilde{\psi}_i = \psi_i \circ \sigma^{-1}$, then each $\widetilde{\psi}_i$ is supported in $B_R$ by what we just arranged. Moreover, by conformal invariance, we have
\begin{equation}\label{eq:2nd-var-conformal}
\delta^2 E_{H}(u)(\widetilde{\psi}_i, \widetilde{\psi}_k; g_{\RR^2}) = \delta^2 E_H (v)(\psi_i, \psi_k; g_{S^2}) \text{ for all }i, k = 1, \cdots, d,
\end{equation}
where we've indicated the metrics used to compute the defining integrals. Here $g_{\RR^2}$ and $g_{S^2}$ denote, respectively, the standard metrics on $\RR^2$ and $S^2$. Now we let $\widetilde{\psi}_{i, j} = P_{\widetilde{u}_j}(\widetilde{\psi}_i)$ and note that since $\widetilde{u}_j \to u$ strongly in $W^{2, 2}(B_R; S^3)$, we have by Lemma~\ref{lemm:projection-estimates}(b) that $\widetilde{\psi}_{i, j} \to \widetilde{\psi}_i$ strongly in $W^{2, 2}(\RR^2; \RR^N)$ as $j \to \infty$. Recalling further that $g_j$ converges to $g_{\RR^2}$ smoothly on compact subset of $\RR^2$, we deduce from Proposition~\ref{prop:second-var-formula-intrinsic} that 
\begin{equation}\label{eq:2nd-var-limit}
\lim_{j \to \infty} \delta^2 E_{H, \widetilde{\ep}_j}(\widetilde{u}_j)(\widetilde{\psi}_{i, j}, \widetilde{\psi}_{k, j}; g_j) = \delta^2 E_{H}(u)(\widetilde{\psi}_i, \widetilde{\psi}_k; g_{\RR^2}).
\end{equation}
We next undo the rescaling~\eqref{eq:case2-blowup} and define $\widehat{\psi}_{i, j}(x) = \widetilde{\psi}_{i, j}\big(\frac{x - x_j}{t_j}\big)$. Then each $\widehat{\psi}_{i, j}$ lies in $\cT_{u_j}$ and is supported in $B_{t_j R}(x_j)$. Moreover, we see from~\eqref{eq:case2-blowup}, ~\eqref{eq:2nd-var-intrinsic}, 
the definition of $g_j$, and the relation $\widetilde{\ep}_j = \frac{\ep_j}{t_j}$ that 
\[
\delta^2 E_{H, \ep_j}(u_j)(\widehat{\psi}_{i, j}, \widehat{\psi}_{k, j}; g) = \delta^2 E_{H, \widetilde{\ep}_j}(\widetilde{u}_j)(\widetilde{\psi}_{i, j}, \widetilde{\psi}_{k, j}; g_j).
\]
Combining this with~\eqref{eq:2nd-var-conformal} and~\eqref{eq:2nd-var-limit}, we deduce that the matrix 
\[
\big(\delta^2 E_{H, \ep_j}(u_j)(\widehat{\psi}_{i, j}, \widehat{\psi}_{k, j})\big)_{i, k = 1, \cdots, d}
\]
is negative-definite for large enough $j$. Thus, again by Corollary~\ref{coro:critical-uniform-bound}(c), we must have $d \leq 1$. The proof is complete.

\section{Improved existence result under curvature assumption}\label{sec:improved}

In this section we prove Theorem~\ref{thm:main2}. The main step is to derive, under the curvature condition in Remark~\ref{rmk:main2}, an a priori energy bound for solutions to~\eqref{equ:CMC equation1},~\eqref{equ:CMC equation2} with Morse index at most $1$, which allows us to reach other values of $H$ from the full measure set produced by Theorem~\ref{thm:main1}. We obtain this energy bound by first transferring the index bound on $\delta^2 E_H(u)$ to another bilinear form $B_H(u)$ to be defined in Section~\ref{subsec:index-comparison}. The desired bound is obtained in Section~\ref{subsec:uniform-energy-bound} via the standard conformal balancing argument (see for example~\cite{LiYau}). The proof of Theorem~\ref{thm:main2} is completed at the very end of the section.

\subsection{Index comparison}\label{subsec:index-comparison}
We begin by setting up notation and recalling some general facts. Let $u$ be a non-constant solution to~\eqref{equ:CMC equation1},~\eqref{equ:CMC equation2}. We denote the pull-back bundle $u^{\ast}TS^3$ by $E$, and let $E_{\CC}$ be the complexification $E \otimes \CC$. The metric and the Levi-Civita connection on $S^{3}$ induce a metric $\langle\ \cdot\ ,\ \cdot\ \rangle$ and a connection $D$ on $E$, which we extend to be complex linear on $E_{\CC}$. In particular $D$ is compatible with the Hermitian metric $\langle\ \cdot\ ,\overline{\ \cdot\ }\rangle$. Given $p \in S^2$, introducing isothermal coordinates $(x^1, x^2)$ centered at $p$, we may write~\eqref{equ:CMC equation1},~\eqref{equ:CMC equation2} as
\begin{align}
2D_{\bar{z}}u_z &= \sqrt{-1}HQ(u_{\bar{z}}, u_{z}), \label{eq:CMC-equation-complex}\\
\langle u_z, u_z \rangle & = 0, \label{eq:conformal-complex}
\end{align}
where as usual we let $\paop{z} = \frac{1}{2}\big( \paop{x^1} - \sqrt{-1}\paop{x^2} \big),\ \paop{\bar{z}} = \frac{1}{2}\big( \paop{x^1} + \sqrt{-1}\paop{x^2} \big)$. It follows from~\eqref{eq:CMC-equation-complex} that if $u_z$ vanishes at $p$, then in terms of $z = x^1 + \sqrt{-1}x^2$ we have
\begin{equation}\label{eq:twisting}
u_{z} = z^{b}v,
\end{equation}
where $b = b_p \in \NN$ and $v$ is a local section of $E_{\CC}$ with $v(p) \neq 0$.  In particular the set $\cS$ of branch points of $u$ is finite, as is well known. Moreover thanks to~\eqref{eq:conformal-complex}, the real and imaginary parts of $v$ allow us to extend the image of $TS^2\big|_{S^2\setminus \cS}$ in $E$ under $du\big|_{S^2 \setminus \cS}$ across each branch point to an oriented, real rank-two subbundle $\xi$ of $E$. The orientation and induced metric on $\xi$ give it the structure of a complex line bundle and induce a splitting of $\xi_{\CC}:= \xi \otimes \CC$ into $\xi^{1, 0} \oplus \xi^{0, 1}$ in such a way that locally away from $\cS$, the fibers of $\xi^{1, 0}$ and $\xi^{0, 1}$ are spanned, respectively, by $u_z$ and $u_{\bar{z}}$. Next, denoting by $\nu$ the orthogonal complement of $\xi$ in $E$, then $\nu$ is trivial as there is a non-vanishing, unit-length section $\bm n$ coming from $Q$ and the orientation on $\xi$. Finally, the connection $D$ induces (metric-compatible) connections $D^{\perp}$ and $D^{T}$ on $\nu$ and $\xi$, respectively, and both $\xi^{1,0}$ and $\xi^{0, 1}$ are preserved by $D^{T}$.

The main result of this section compares $\Ind_H(u)$ with the index of the following bilinear form: For $s \in \Gamma(\nu)$ we let
\begin{equation}
B_H(u)(s, s) = \int_{S^2}\big[ |\nabla f|^2 - \frac{|\nabla u|^2}{2}\big(\frac{H^2}{2} + \Ric_g(\bm n, \bm n)\big)f^2 \big] d\Vol_{S^2},
\end{equation}
where $f = \langle s, \bm n \rangle \in C^{\infty}(S^2; \RR)$ and $\Ric_g$ denotes the Ricci tensor of $(S^3, g)$. Note that $B_H(u)$ is essentially the second variation of a weighted area functional (see for example~\cite[Equation (1.3)]{Zhou-Zhu19}), with the second fundamental form term replaced by $H^2/2$ to avoid issues at branch points of $u$. The index of $B_H(u)$ is by definition the supremum of  $\dim V$ over all finite-dimensional subspaces $V$ of $\Gamma(\nu)$ on which $B_H(u)$ is negative-definite. We now state the index comparison result.
\begin{prop}\label{prop:index-comparison}
For a non-constant solution $u$ to~\eqref{equ:CMC equation1} and~\eqref{equ:CMC equation2}, the index of $B_H(u)$ is less than or equal to $\Ind_{H}(u)$.
\end{prop}
The essential ingredient for proving Proposition~\ref{prop:index-comparison} is the following computation adapted from~\cite[Theorem 2.1]{EM}, which we explain before proving Proposition~\ref{prop:index-comparison}.

\begin{lemm}\label{lemm:index-computation}
For any $\sigma \in \Gamma(\xi)$ and $s \in \Gamma(\nu)$, write $v = s + \sigma \in \Gamma(E)$ and define $\eta \in \Gamma(\Lambda^{1, 0}T^{\ast}S^2 \otimes \xi^{0, 1})$ by $\eta = \big((D_z s)^{0, 1} + D_z^{T}\sigma^{0, 1}\big) dz$. Then 
\begin{equation}\label{eq:index-computation}
\delta^2 E_{H}(u)(v, v) \leq B_H(u)(s, s) + 4\int_{S^2} |\eta|^2  d\Vol_{S^{2}}.
\end{equation}
\end{lemm}
\begin{proof}[Proof of Lemma~\ref{lemm:index-computation}]
Throughout this proof, we use $R$ and $\Ric$ to denote the curvature tensor and Ricci tensor of $(S^3, g)$, respectively. To begin, note that by the logarithmic cut-off trick it suffices to prove~\eqref{eq:index-computation} in the case where $s$ and $\sigma$ are supported away from the set $\cS$ of branch points. Also, as the conclusion, particularly the form $\eta$, suggests, we will work with complexified versions of $\delta^2 E_H(u)$ and $B_H(u)$. To begin, we use Proposition~\ref{prop:second-var-formula-intrinsic} and follow the calculations in~\cite[p.208-209]{Micallef-Moore88} to rewrite $\delta^2 E_H(u)(v, v)$ in terms of complex coordinates as follows:
\begin{align}\label{eq:d2E-complex}
\delta^2 E_{H}(u)(v, v) = 4\int_{S^2} |D_{z} v|^2 - &\langle R(v, u_{z})u_{\bar{z}}, v \rangle dx^1 \wedge dx^2\nonumber\\
& +4\int_{S^2} \re\langle D_{z}v, \overline{\sqrt{-1}HQ(u_z, v)} \rangle dx^1 \wedge dx^2.
\end{align}
In fact, our situation is simpler than in~\cite{Micallef-Moore88} since here $v$ is a section of $E$ rather than $E_{\CC}$. Also, note that the 2-forms in the integrals, written in isothermal coordinates $(x^1, x^2)$, are conformal invariant, and hence make sense globally on $S^2$. For later use we write $\mu$ for the conformal factor such that $g_{S^2} = \mu^2((dx^1)^2 + (dx^2)^2)$. To continue we note the following identities, both immediate consequences of~\eqref{eq:CMC-equation-complex}.
\begin{align}
2(D_z s)^{1, 0} &= -\sqrt{-1}HQ(u_z, s), \label{eq:s10} \\
2(D_z \sigma^{0, 1})^{\perp} & = -\sqrt{-1}HQ(u_z, \sigma^{0, 1}). \label{eq:sigma-01-perp}
\end{align}
Now we look at the second integral in~\eqref{eq:d2E-complex}. Splitting $v = s + \sigma^{0, 1} + \sigma^{1, 0}$, noting that $Q(u_z, \sigma^{1, 0}) = 0$ and using~\eqref{eq:s10} and~\eqref{eq:sigma-01-perp}, we get
\begin{align*}
&\re\langle D_{z}v, \overline{\sqrt{-1}HQ(u_z, v)} \rangle\\
=\ & \re\langle D_{z}v, \overline{\sqrt{-1}HQ(u_z, \sigma^{0, 1})} \rangle + \re\langle D_{z}v, \overline{\sqrt{-1}HQ(u_z, s)} \rangle\\
=\ &-2 \re\langle (D_{z}v)^{\perp}, \overline{(D_z \sigma^{0, 1})^{\perp}} \rangle - 2 \re\langle (D_{z}v)^{T}, \overline{(D_z s)^{1, 0}} \rangle.
\end{align*}
Combining the above calculation with the gradient term in the first line of~\eqref{eq:d2E-complex} and writing $|D_z v|^2 = |(D_z v)^{\perp}|^2 + |(D_z v)^{T}|^2$, we get
\begin{align}
&|D_z v|^2 +  \re\langle D_{z}v, \overline{\sqrt{-1}HQ(u_z, v)} \rangle \nonumber\\
=\ & \Big(|(D_z v)^{\perp}|^2 -2 \re\langle (D_{z}v)^{\perp}, \overline{(D_z \sigma^{0, 1})^{\perp}} \rangle\Big)\nonumber\\
& +  \Big(|(D_z v)^{T}|^2 -2 \re\langle (D_{z}v)^{T}, \overline{(D_z s)^{1, 0}} \rangle \Big)\nonumber \\
=\ & \Big( |(D_z v)^{\perp} - (D_z \sigma^{0, 1})^{\perp}|^2 - |(D_z \sigma^{0, 1})^{\perp}|^2\Big) \nonumber\\
&+ \Big(|(D_z v)^{T} - (D_z s)^{1, 0}|^2 - |(D_z s)^{1, 0}|^2\Big) \nonumber \\
=\ & |(D_z s)^{\perp} + (D_z \sigma^{1, 0})^{\perp} |^2 - |(D_z \sigma^{0, 1})^{\perp}|^2 \nonumber \\
& + | (D_z s)^{0, 1} + (D_z \sigma^{0, 1})^{T} + (D_z \sigma^{1, 0})^{T} |^2 - |(D_z s)^{1, 0}|^2, \label{eq:gradient-volume}
\end{align}
where in getting the last line we wrote $(D_z v)^{\perp} = (D_z s)^{\perp} + (D_z \sigma^{1, 0})^{\perp} + (D_z \sigma^{0, 1})^{\perp}$ and $(D_z v)^{T} = (D_z s)^{1, 0} + (D_z s)^{0, 1} + (D_z \sigma^{1, 0})^{T} + (D_z \sigma^{0, 1})^{T}$. Expanding the two square terms in~\eqref{eq:gradient-volume} with plus signs, we get
\begin{align}
& |(D_z s)^{\perp} + (D_z \sigma^{1, 0})^{\perp} |^2 = |(D_z s)^{\perp}|^2 + |(D_z \sigma^{1, 0})^{\perp}| + 2\re\langle (D_z s)^{\perp}, \overline{D_z \sigma^{1, 0}} \rangle \label{eq:square-expanded-1}\\
& |(D_z s)^{0, 1} + (D_z \sigma^{0, 1})^{T} + (D_z \sigma^{1, 0})^{T}|^2 = \frac{\mu^2}{2}|\eta|^2 + |(D_z \sigma^{1, 0})^{T}|^2,\label{eq:square-expanded-2}
\end{align}
Next we add~\eqref{eq:square-expanded-1} and~\eqref{eq:square-expanded-2}, put it back into~\eqref{eq:gradient-volume} and use the following identities
\begin{align*}
|(D_z \sigma^{1, 0})^{\perp}|  + |(D_z \sigma^{1, 0})^{T}|^2 &= |D_z \sigma^{1, 0}|^2\\
\re\langle (D_z s)^{\perp}, \overline{D_z \sigma^{1, 0}} \rangle &= \re\langle D_z s, \overline{D_z \sigma^{1, 0}} \rangle - \re\langle (D_z s)^{1, 0}, \overline{D_z \sigma^{1, 0}} \rangle\\
|(D_z s)^{1, 0}|^2 &= |(D_z s)^{T}|^2 - |(D_z s)^{0, 1}|^2
\end{align*}
to obtain 
\begin{align}\label{eq:gradient-volume-2}
&|D_z v|^2 +  \re\langle D_{z}v, \overline{\sqrt{-1}HQ(u_z, v)} \rangle\nonumber\\
=\ & \frac{\mu^2}{2}|\eta|^2 + |(D_z s)^{\perp}|^2  - |(D_z s)^{T}|^2 + |(D_z s)^{0, 1}|^2 - |(D_z \sigma^{0, 1})^{\perp}|^2\nonumber\\
& + |(D_z \sigma^{1, 0})|^2 + 2\re\langle D_z s, \overline{D_z \sigma^{1, 0}} \rangle -2 \re\langle (D_z s)^{1, 0}, \overline{D_z \sigma^{1, 0}} \rangle.
\end{align}
The integral of the term $|D_z \sigma^{1, 0}|^2$ is treated as in~\cite[p.7, equation (2.25)]{EM} using integration by parts and the Bianchi identity. The result is
\begin{equation}\label{eq:by-parts-1}
\int_{S^2} |D_z \sigma^{1, 0}|^2 dx^1 \wedge dx^2 = \int_{S^2} |D_z\sigma^{0, 1}|^2 + \langle R(\sigma^{1, 0}, u_{\bar{z}})u_z, \sigma^{0, 1} \rangle dx^1 \wedge dx^2.
\end{equation}
Similarly, for the integral of $\re\langle D_z s, \overline{D_z \sigma^{1, 0}} \rangle$, we integrate by parts (see~\cite[p.6]{EM}) to get
\begin{align}\label{eq:by-parts-2}
&\int_{S^2} \re\langle D_z s, D_{\bar{z}}\sigma^{0, 1} \rangle  dx^1 \wedge dx^2\\
=\ &\int_{S^2}-\re\langle s, R(u_z, u_{\bar{z}})\sigma^{0, 1} \rangle + \re\langle D_{z} s, D_{\bar{z}}\sigma^{1, 0} \rangle dx^1 \wedge dx^2  \nonumber\\
=\ & \int_{S^2}-\re\langle s, R(u_z, u_{\bar{z}})\sigma^{0, 1} \rangle + \re\langle (D_{z} s)^{0, 1}, (D_{\bar{z}}\sigma^{1, 0})^{T} \rangle dx^1 \wedge dx^2\nonumber\\
& + \int_{S^2}\re\langle (D_{z} s)^{\perp}, (D_{\bar{z}}\sigma^{1, 0})^{\perp} \rangle dx^1 \wedge dx^2. \nonumber
\end{align}
Putting~\eqref{eq:by-parts-1} and~\eqref{eq:by-parts-2} back into~\eqref{eq:gradient-volume-2}, and noting that 
\[
 |(D_z s)^{0, 1}|^2 + |(D_z \sigma^{0, 1})^{T}|^2 + 2 \re\langle (D_{z} s)^{0, 1}, (D_{\bar{z}}\sigma^{1, 0})^{T} \rangle = \frac{\mu^2}{2}|\eta|^2,
\]
and that $dx^1 \wedge dx^2 = \mu^{-2}d\Vol_{S^2}$, we get 
\begin{align}\label{eq:gradient-volume-3}
&\int_{S^2} |D_z v|^2 + \re\langle D_{z}v, \overline{\sqrt{-1}HQ(u_z, v)} \rangle dx^1 \wedge dx^2\nonumber\\
=\ & \int_{S^2}|\eta|^2 d\Vol_{S^2} + \int_{S^2}|(D_z s)^{\perp}|^2  - |(D_z s)^{T}|^2 dx^1 \wedge dx^2 \nonumber \\
& + \int_{S^2} \langle R(\sigma^{1, 0}, u_{\bar{z}})u_{z}, \sigma^{0, 1} \rangle-2\re\langle s, R(u_{z}, u_{\bar{z}})\sigma^{0, 1} \rangle dx^1 \wedge dx^2 \nonumber\\
 &+ \int_{S^2}2\re\langle (D_{z} s), (D_{\bar{z}}\sigma^{1, 0})^{\perp} \rangle - 2\re\langle (D_z s)^{1, 0}, \overline{D_z \sigma^{1, 0}} \rangle  dx^1 \wedge dx^2.
\end{align}
By~\eqref{eq:s10} and~\eqref{eq:sigma-01-perp} and an integration by parts using the fact that $Q$ is parallel with respect to the Levi-Civita connection on $S^3$, it is not hard to see that the two terms in the last line cancel. Moreover, as in~\cite[p.7, equation (2.22)]{EM}, the third line combines with the curvature term in~\eqref{eq:d2E-complex} to give
\begin{align*}
&\int_{S^2} \langle R(\sigma^{1, 0}, u_{\bar{z}})u_{z}, \sigma^{0, 1} \rangle-2\re\langle s, R(u_{z}, u_{\bar{z}})\sigma^{0, 1} \rangle - \langle R(v, u_z)u_{\bar{z}}, v \rangle dx^1 \wedge dx^2 \\
 = &-\int_{S^2} \langle R(s, u_z)u_{\bar{z}}, s \rangle dx^1 \wedge dx^2.
\end{align*}
Therefore we obtain
\begin{align}\label{eq:index-compare-almost}
\delta^2 E_{H}(u)(v, v) =&\ 4\int_{S^2} |\eta|^2 d\Vol_{S^2}\nonumber\\
& + 4\int_{S^2} |(D_z s)^{\perp}|^2 - |(D_z s)^{T}|^2 - \langle R(s, u_z)u_{\bar{z}}, s \rangle dx^1 \wedge dx^2.
\end{align}
To finish, we write $s = f\bm n$ and note that 
\[
|(D_z s)^{\perp}|^2 = |\nabla_z f|^2 = \frac{\mu^2}{4}|\nabla f|^2, \text{ and }
\]
\[
\langle R(s, u_z)u_{\bar{z}}, s \rangle =  \frac{\mu^2 |\nabla u|^2}{8}f^2 \Ric(\bm n, \bm n), \text{ while}
\]
\[
|(D_z s)^{T}|^2 \geq |(D_z s)^{1, 0}|^2 = \frac{H^2}{4}f^2 |u_z|^2 = \frac{\mu^2 |\nabla u|^2}{2}\frac{H^2}{8}f^2.
\]
Therefore the second integral on the right-hand side of~\eqref{eq:index-compare-almost} is bounded above by
\[
\int_{S^2}|\nabla f|^2 - \frac{|\nabla u|^2}{2}\Big( \frac{H^2}{2} + \Ric(\bm n, \bm n) \Big)f^2 d\Vol_{S^2},
\]
which gives the inequality~\eqref{eq:index-computation} as asserted.
\end{proof}

We are now ready to give the proof of Proposition~\ref{prop:index-comparison}.
\begin{proof}[Proof of Proposition~\ref{prop:index-comparison}]
With Lemma~\ref{lemm:index-computation} at our disposal, we may finish the proof exactly as in~\cite[p.7-8]{EM}. We reproduce a sketch of the argument for the genus-zero case, which is all we need. Below we write $D' = (D_z \ \cdot)^{T}dz$ and $D'' = (D_{\bar{z}}\  \cdot)^T d\bar{z}$.

In view of~\eqref{eq:index-computation}, given $s \in \Gamma(\nu)$ we want to solve the following equation for $\sigma^{0, 1} \in \Gamma(\xi^{0, 1})$: 
\begin{equation}\label{eq:dbar-pde}
D'\sigma^{0, 1}
 = -(D_z s)^{0, 1} dz =: \alpha.
\end{equation}
Since there are no $(2, 0)$-forms on a Riemann surface, we see by the Fredholm alternative that~\eqref{eq:dbar-pde} has a solution if and only if $\alpha$ is $L^{2}$-orthogonal to the kernel of the adjoint $(D')^{\ast}$ of $D'$, defined with respect to the metrics on $S^2$ and $\xi^{0, 1}$. To continue, recall that the Koszul-Malgrange theorem (see for instance~\cite[Theorem 2.1.53]{DK}) yields a holomorphic structure on $\xi^{0, 1}$ compatible with $D''$, and that $\Ker((D')^{\ast})$ has the same dimension as the space of holomorphic sections of $\Lambda^{1,0}T^{\ast}S^2 \otimes \xi^{0, 1}$. However, a direct computation shows that $c_1(\xi^{0, 1}) = -2 - \sum_{p \in \cS}b_p$, and hence
\begin{align*}
c_1(\Lambda^{1,0}T^{\ast}S^2 \otimes \xi^{0, 1}) &= c_1(\Lambda^{1,0}T^{\ast}S^2) + c_1(\xi^{0, 1}) = -4-\Sigma_{p \in \cS}b_p < 0.
\end{align*}
Hence $\Lambda^{1,0}T^{\ast}S^2 \otimes \xi^{0, 1}$ has no holomorphic sections, and for each $s \in \Gamma(\nu)$, the equation~\eqref{eq:dbar-pde} has a solution. To finish, let $s_1, \cdots, s_d$ be linearly independent sections of $\nu$ such that $B_H(u)$ is negative-definite on their span $V$. For $i = 1, \cdots, d$ we choose a solution $\sigma_i^{0, 1} \in \Gamma(\xi^{0, 1})$ to~\eqref{eq:dbar-pde} with $s_i$ in place of $s$ and define $\sigma_i = \sigma_i^{0, 1} + \overline{\sigma_i^{0, 1}}$. Extending by linearity the correspondence $s_i \mapsto s_i + \sigma_i$ to a linear map $T: V \to \Gamma(E)$, we see that $T$ is injective, so $\dim T(V) = \dim V = d$. Furthermore, Lemma~\ref{lemm:index-computation} along with the linearity of both sides of equation~\eqref{eq:dbar-pde} (with respect to $\sigma^{0, 1}$ and $s$, respectively) imply that
\[
\delta^2 E_H(u)(T(s), T(s)) \leq B_H(u)(s, s) < 0 \text{ for all }s \in V.
\]
This implies that $\Ind_H(u)$ is at least as large as the index of $B_H(u)$, as asserted.
\end{proof}

\subsection{Uniform energy bound}\label{subsec:uniform-energy-bound}
In this section we use the index comparison above together with the standard conformal balancing argument (see for instance~\cite{LiYau}) to get the uniform energy bound mentioned in the remarks above Section~\ref{subsec:index-comparison}. The main result is the following.
\begin{prop}\label{prop:uniform-area-energy}
Suppose for some $c_0 > 0$ we have
\begin{equation}\label{eq:Ricci-condition}
\Ric_g + \frac{H^2}{2}g \geq c_0 g,
\end{equation}
and let $u$ be a solution to~\eqref{equ:CMC equation1} and~\eqref{equ:CMC equation2} with $\Ind_H(u) \leq 1$. Then
\begin{equation}\label{eq:uniform-area-energy}
D(u) \leq 8\pi/c_0.
\end{equation}
\end{prop}
\begin{proof}
Of course it suffices to prove the proposition assuming that $u$ is non-constant. By assumption and Proposition~\ref{prop:index-comparison}, we see that $B_H(u)$ has index at most $1$. Substituting $f \equiv 1$ into $B_H(u)(f, f)$ then implies by assumption~\eqref{eq:Ricci-condition} that $B_H(u)$ has index exactly $1$. In particular, 
\begin{equation}\label{eq:positive-if-orthogonal}
B_H(u)(f, f) \geq 0 \text{ provided }\int_{S^2} f\varphi d\Vol_{S^2} = 0,
\end{equation}
where $\varphi > 0$ is a lowest eigenfunction for the operator $-\Delta - \frac{|\nabla u|^2}{2}(\frac{H^2}{2} + \Ric_g(\bm n, \bm n))$. By a degree theory argument as in~\cite[p. 274]{LiYau}, we obtain a conformal map $F: S^2 \to S^2$ so that 
\[
\int_{S^2} (x^i \circ F) \varphi d\Vol_{S^2} = 0, \text{ for }i = 1, 2, 3.
\]
Here $x^1, x^2, x^3$ denote the coordinate functions of the standard embedding $S^2 \to \RR^3$. Hence by~\eqref{eq:positive-if-orthogonal} we have
\[
\int_{S^2}|\nabla (x^i \circ F)|^2 - \frac{|\nabla u|^2}{2} \big( \frac{H^2}{2} + \Ric_g(\bm n, \bm n) \big)(x^i \circ F)^2 d\Vol_{S^2} \geq 0.
\]
Rearranging the above inequality, summing over $i$ and using~\eqref{eq:Ricci-condition}, we get
\begin{align*}
c_0 \int_{S^2}\frac{|\nabla u|^2}{2}\sum_{i = 1}^3 (x^i \circ F)^2 d\Vol_{S^2} &\leq \sum_{i = 1}^3 \int_{S^2} |\nabla (x^i \circ F)|^2 d\Vol_{S^2} = \sum_{i = 1}^3\int_{S^2} |\nabla x^i|^2 d\Vol_{S^2},
\end{align*}
where we used the conformal invariance of the Dirichlet integral to get the equality. Recalling that $\sum_{i = 1}^3\int_{S^2} |\nabla x^i|^2 d\Vol_{S^2} = 8\pi$ and that $\sum_{i = 1}^3 (x^i)^2 = 1$, we get $c_0D(u) \leq 8\pi$, which gives~\eqref{eq:uniform-area-energy}.
\end{proof}
\begin{rmk}
If in addition $u$ is an immersion, namely without branch points, then a uniform bound on $D(u)$ depending only on $c_0$ holds assuming only that $R_g + \frac{3H^2}{2} \geq c_0>0$, where $R_g$ is the scalar curvature of $(S^3, g)$. Indeed, in this case Proposition~\ref{prop:index-comparison} holds with $\frac{H^2}{2}$ replaced by $|A|^2$ in the definition of $B_H(u)$, where the metric $\frac{|\nabla u|^2}{2}g_{S^2}$ is used on $S^2$ in computing the norm of the second fundamental form. The desired conclusion then follows from an adaptation (see for instance~\cite[p.229]{Rosenberg2006}) of the Schoen-Yau rearrangement~\cite[p.139]{Schoen-Yau1979} and the same conformal balancing argument as above. That said, ruling out branch points, or getting an a priori estimate on the total branching order for that matter, seems to be a difficult question.
\end{rmk}

\subsection*{Proof of Theorem~\ref{thm:main2}}
We shall prove the following stronger statement:
\vskip 2mm
\textit{For all $H > 0$ such that $\Ric_g + \frac{H^2}{2}g > 0$, there exists a non-constant solution $u$ to~\eqref{equ:CMC equation1} and~\eqref{equ:CMC equation2} with $\Ind_H(u) = 1$}. 
\vskip 2mm
To begin, note that by compactness there exists $c_0 > 0$ so that~\eqref{eq:Ricci-condition} holds. Next, thanks to Theorem~\ref{thm:main1} we can find a sequence $H_k$ increasing to $H$ such that for all $k$ there exists a non-constant solution $u_k$ to~\eqref{equ:CMC equation1},~\eqref{equ:CMC equation2} with $H_k$ in place of $H$ which satisfies $\Ind_{H_k}(u_k) \leq 1$. It is also clear from the proof of Theorem~\ref{thm:main1} that 
\begin{equation}\label{eq:lower-bound-k}
D(u_k) \geq \min\{\beta, \eta_0/3\} =: \beta_1,
\end{equation}
where the constants $\eta_0$ and $\beta$ are given respectively by Proposition~\ref{prop:eta-regularity} and Proposition~\ref{prop:lower-bound} with $H$ in place of $H_0$. Now since $H_k \to H$, we have by~\eqref{eq:Ricci-condition}, Proposition~\ref{prop:uniform-area-energy} and the index bound $\Ind_{H_k}(u_k) \leq 1$ that 
\begin{equation}\label{eq:approximation-energy-bound}
D(u_k) \leq 16\pi/c_0 \text{ for large enough }k.
\end{equation}
A straightforward adaptation of the proof of~\cite[Theorem 2.2]{Sch} yields constants $\delta_1, C > 0$ depending only on $H$ such that if $\int_{B_{2r}(x)} |\nabla u_k|^2 < \delta_1$, then
\[
r^2 \sup_{B_{r}(x)}|\nabla u_k|^2 \leq C\int_{B_{2r}(x)}|\nabla u_k|^2,
\]
and bootstrapping in~\eqref{equ:CMC equation1} gives all higher-order estimates. From this and~\eqref{eq:approximation-energy-bound}, and arguing as in Proposition~\ref{prop:convergence-mod-bubble}, we again see a dichotomy as in the proof of Theorem~\ref{thm:main1}: Either (i) $u_k$ has a subsequence converging smoothly on $S^2$ to a limit $u$, which must solve~\eqref{equ:CMC equation1},~\eqref{equ:CMC equation2} and satisfy $D(u) \geq \beta_1$, or (ii) the sequence $u_k$ exhibits energy concentration at some point, in which case we obtain a solution $v$ to~\eqref{equ:CMC equation1},~\eqref{equ:CMC equation2} satisfying $D(v) \geq \delta_1/3$ by rescaling and using Proposition~\ref{prop:removable-singularity}. In both cases, the index upper bound is established exactly as in Theorem~\ref{thm:main1}, so we omit the details. We finish upon noting that, as in the proof of Proposition~\ref{prop:uniform-area-energy}, the curvature condition~\eqref{eq:Ricci-condition} and the non-constancy of $u$ implies that $B_H(u)$ has positive index, and thus $\Ind_H(u) > 0$. The proof of Theorem~\ref{thm:main2} is complete.

\appendix

\section{Proofs of some standard estimates} 
\label{sec:proof of two lemmas}

As promised above, below we outline the proofs of Lemma \ref{lemm:D-equivalence}, \ref{lemm:projection-estimates} and equation~\eqref{eq:W22-L4}.

\begin{proof}[Proof of Lemma \ref{lemm:D-equivalence}]
For part (a), we first note that for $h \in C^3(S^2)$, testing the B\^ochner formula $\Delta\nabla h = \nabla \Delta h + \Ric^{S^2}(\nabla h)$ against $\nabla h$ and integrating by parts lead to
\begin{equation}\label{eq:Bochner-by-parts}
\int_{S^2} |\nabla^2 h|^2 = \int_{S^2} |\Delta h|^2 - \Ric^{S^2}(\nabla h, \nabla h).
\end{equation}
By approximation, the above identity also holds for any $h \in W^{2, 2}(S^2)$. Applying this to $u$ as a map in $W^{2, 2}(S^2; \RR^N)$ and rearranging, we get
\begin{equation}\label{eq:hessian-Delta-difference}
\Big | \int_{S^2}|\Delta u|^2 - \int_{S^2}|\nabla^2 u|^2 \Big| \leq C\int_{S^2}|\nabla u|^2,
\end{equation}
which gives (a) and the first inequality in (b).
The second inequality of (b) follows immediately from the first and the fact that $u$ maps into a compact submanifold of $\RR^N$. Finally it is clear from the proof that~\eqref{eq:hessian-Delta-difference}, and hence part (a), holds for $u \in W^{2, 2}(S^2; \RR^N)$.
\end{proof}

\begin{proof}[Proof of Lemma \ref{lemm:projection-estimates}]
For part (a), the fact that $\Pi(\widetilde{v})$ belongs to $W^{2, 2}(S^2; S^3)$ follows from Sobolev embedding along with the chain rule and product rule for weak derivatives. (See for instance~\cite[Chapter 7]{GT}.) To prove the estimate, we differentiate $\Pi(\widetilde{v})$ to see that
\begin{equation}\label{eq:projection-a}
\int_{S^2}|\nabla v|^2 + \ep^2 |\nabla^2 v|^2 \leq C\int_{S^2}|\nabla \widetilde{v}|^2 + \ep^2 |\nabla \widetilde{v}|^4 +\ep^2 |\nabla^2 \widetilde{v}|^2.
\end{equation}
To estimate the integral of $\ep^2 |\nabla\widetilde{v}|^4$, recall that since we are on a two-dimensional domain, there holds the following Sobolev inequality 
\begin{equation}\label{eq:Sobolev-W12-L4}
\int_{S^2} |\nabla \widetilde{v}|^4 \leq C\|\nabla\widetilde{v}\|_{1, 2}^2\int_{S^2}|\nabla\widetilde{v}|^2 + |\nabla^2\widetilde{v}|^2.
\end{equation}
Multiplying both sides by $\ep^2$ and putting this back into~\eqref{eq:projection-a} yield the estimate in part (a). To prove the last assertion of part (a), note that given $u_0 \in W^{2, 2}(S^2; \cV)$, there exists $\delta \in (0, 1)$ such that for all $\|u - u_0\|_{2, 2}, \|v - u_0\|_{2, 2} < \delta$, we have that $tv + (1 - t)u$ still maps into $\cV$ for all $t \in [0, 1]$, and that $\|u\|_{2, 2}, \|v\|_{2, 2} < \|u_0\|_{2, 2} + 1 := K$. We may then differentiate $\Pi(tv + (1 - t)u)$ and use the fundamental theorem of calculus to get 
\[
\|\Pi(u) - \Pi(v)\|_{2, 2} \leq C_K\|u - v\|_{2, 2},
\]
after some routine calculations using H\"older's inequality and Sobolev embedding. 

The proof of part (b) is very similar. Since $y \mapsto P_y$ is a smooth matrix-valued function for $y \in \cV$, we may argue as in part (a) to see that $u \mapsto P_u$ is a locally Lipschitz map from $W^{2, 2}(S^2; \cV) \to W^{2, 2}(S^2; \RR^{N \times N})$, and that $\int_{S^2} |\nabla (P_u)|^2 + |\nabla^2 (P_u)|^2 \leq C(1 + \|u\|_{2, 2}^2)\int_{S^2} |\nabla u|^2 + |\nabla^2 u|^2$. Consequently, since $P: \cV \to \RR^{N \times N}$ is bounded, there holds 
\[
\|P_u\|_{2, 2}^2 \leq C(1 + \|u\|_{2, 2}^2)^2,
\]
for some universal constant $C$. All the remaining conclusions of part (b) follow from the previous estimate, the observation that $P_u(\widetilde{\psi})$ is a matrix-vector product, and the fact that $fh$ belongs to $W^{2, 2}(S^2)$ whenever both $f$ and $h$ do, in which case 
\begin{equation}\label{eq:W22-algebra}
\|f h\|_{2,2} \leq C\|f\|_{2, 2}\|h\|_{2, 2},
\end{equation} 
where $C$ is again some universal constant. (See for example~\cite[Theorem 4.39]{Adams} or~\cite[Corollary 9.7]{Pal} for this last fact.)
\end{proof}

\begin{proof}[Proof of~inequality \eqref{eq:W22-L4}]
We choose a cut-off function $\zeta \in C^{\infty}(B_{4r}(x))$ which equals $1$ on $B_{7r/2}(x)$ and satisfies $|\nabla^k\zeta| \leq Cr^{-k}$. Moreover, we let $a = \fint_{B_{4r}(x)}u$. Applying~\eqref{eq:hessian-Delta-difference} to $\zeta (u - a)$, we get
\[
\int_{B_{4r}(x)} |\nabla^2(\zeta (u - a))|^2 \leq \int_{B_{4r}(x)} |\Delta (\zeta(u - a))|^2 + C|\nabla (\zeta (u - a))|^2.
\]
By a simple application of Young's inequality we have, for any $\delta > 0$, that
\begin{align}\label{eq:commute-cutoff-1}
\int_{B_{4r}(x)} \zeta^2 |\nabla^2 u|^2 \leq& (1 + \delta)\int_{B_{4r}(x)} |\nabla^2(\zeta (u - a))|^2\nonumber\\
& + C_{\delta}\int_{B_{4r}(x)} |\nabla \zeta|^2 |\nabla u|^2 + |u - a|^2 |\nabla^2 \zeta|^2.
\end{align}
An estimate to the opposite effect holds with $\nabla^2$ replaced by $\Delta$:
\begin{align}\label{eq:commute-cutoff-2}
\int_{B_{4r}(x)}|\Delta (\zeta (u - a))|^2  \leq& (1 + \delta)\int_{B_{4r}(x)} \zeta^2 |\Delta u|^2 \nonumber\\
& + C_{\delta}\int_{B_{4r}(x)} |\nabla \zeta|^2 |\nabla u|^2 + |u - a|^2 |\nabla^2 \zeta|^2.
\end{align}
Choosing $\delta$ small enough, we get after some straightforward computation using the properties of $\zeta$ that
\begin{align}
\int_{B_{7r/2}(x)}\ep^2 |\nabla^2 u|^2 &\leq 2\int_{B_{4r}(x)} \ep^2|\Delta u|^2 + C\big( \ep^2 + (\ep/r)^2 \big) \int_{B_{4r}(x)} |\nabla u|^2 + r^{-2}|u - a|^2 \nonumber \\
&\leq 2 \int_{B_{4r}(x)} \ep^2|\Delta u|^2 + C\big( \ep^2 + (\ep/r)^2 \big) \int_{B_{4r}(x)} |\nabla u|^2 \label{eq:W22-Delta-L2},
\end{align}
where the second line follows by the Poincar\'e inequality. To bound the integral of $|\nabla u|^4$, we choose another cut-off function $\xi \in C^{\infty}_c(B_{7r/2}(x))$ which equals $1$ on $B_{3r}(x)$ and satisfies $|\nabla^k\xi| \leq Cr^{-k}$. Applying the Sobolev inequality~\eqref{eq:Sobolev} to $h = \xi^2 |\nabla u|^2$ and using H\"older's inequality a few times, we get
\[
\int_{B_{3r}(x)}  |\nabla u|^4 \leq Cr^{-2}\Big( \int_{B_{7r/2}(x)} |\nabla u|^2 \Big)^2 + C\Big( \int_{B_{7r/2}(x)}|\nabla u|^2  \Big)\Big( \int_{B_{7r/2}(x)}|\nabla^2 u|^2 \Big)
\]
Hence 
\begin{equation}
\int_{B_{3r}(x)}\ep^2 |\nabla u|^4 \leq C\big( 1 + (\ep/r)^2 \big)\Big( \int_{B_{7r/2}(x)} |\nabla u|^2 \Big)\Big( \int_{B_{7r/2}(x)} \ep^2 |\nabla^2 u|^2 + |\nabla u|^2\Big).
\end{equation}
Combining this with~\eqref{eq:W22-Delta-L2} gives~\eqref{eq:W22-L4}.
\end{proof}

\bibliographystyle{amsplain}
\bibliography{cmc-existence-arxiv-update}
\end{document}